\documentclass{amsart}
\usepackage[utf8]{inputenc}
\usepackage[english]{babel}

\usepackage[mathscr]{eucal}
\usepackage{amsmath}
\usepackage{amsfonts}
\usepackage{amssymb}
\usepackage{amsthm}
\usepackage{cmap}
\usepackage{hyperref}
\usepackage{mathrsfs}
\usepackage{tikz-cd}
\usepackage{mathtools}

\makeatletter
\newcommand{\leqnomode}{\tagsleft@true}
\newcommand{\reqnomode}{\tagsleft@false}
\makeatother
\newcommand{\sSet}{\mathsf{sSet}}
\newcommand{\sVect}{\mathsf{sVect}}
\newcommand{\ssVect}{\mathsf{ssVect}}
\newcommand{\soVect}{\mathsf{s}_0\mathsf{Vect}}

\newcommand{\Lie}{\mathsf{Lie}}
\newcommand{\rLie}{\mathsf{Lie}^\mathit{r}}
\newcommand{\srLie}{\mathsf{sLie}^\mathit{r}}
\newcommand{\ssrLie}{\mathsf{ssLie}^\mathit{r}}

\newcommand{\srfinLie}{\mathsf{sLie}^\mathit{r}_{\mathit{fin}}}
\newcommand{\Vect}{\mathsf{Vect}}

\DeclareMathOperator{\Alg}{\mathsf{Alg}}

\DeclareMathOperator{\sAlg}{\mathsf{sAlg}}

\newcommand{\CoAlg}{\mathsf{CoAlg}}
\newcommand{\trcoalg}{\mathsf{CoAlg}^{\mathit{tr}}}

\newcommand{\Hopf}{\mathsf{Hopf}}
\newcommand{\prHopf}{\mathsf{Hopf}^{\mathit{pr}}}

\newcommand{\Grp}{\mathsf{Grp}}

\newcommand{\csSet}{\mathsf{csSet}}

\newcommand{\strcoalg}{\mathsf{sCoAlg}^{\mathit{tr}}}
\newcommand{\sotrcoalg}{\mathsf{s}_{0}\mathsf{CoAlg}^{\mathit{tr}}}

\newcommand{\csotrcoalg}{\mathsf{cs}_{0}\mathsf{CoAlg}^{\mathit{tr}}}

\newcommand{\sstrcoalg}{\mathsf{ssCoAlg}^{\mathit{tr}}}
\newcommand{\sostrcoalg}{\mathsf{s}_{0}\mathsf{sCoAlg}^{\mathit{tr}}}

\newcommand{\scoalg}{\mathsf{sCoAlg}}
\newcommand{\ccoalg}{\mathsf{cCoAlg}}

\newcommand{\sprHopf}{\mathsf{sHopf}^{\mathit{pr}}}

\newcommand{\Mod}{\mathsf{Mod}}
\newcommand{\sMod}{\mathsf{sMod}}

\newcommand{\sC}{\mathsf{sC}}

\newcommand{\sL}{\mathit{s}\EuScript{L}}
\newcommand{\sCA}{\mathit{s}\EuScript{CA}}
\newcommand{\sLxi}{\mathit{s}\EuScript{L}_{\xi}}

\newcommand{\free}{\mathit{L^r}}
\newcommand{\oblv}{\mathsf{oblv}}
\newcommand{\triv}{\mathsf{triv}}
\newcommand{\trivxi}{\mathsf{triv}_{\xi}}

\newcommand{\barW}{\overline{W}}
\newcommand{\Sym}{\mathrm{Sym}}
\newcommand{\Ab}{\mathsf{Ab}}
\newcommand{\Abxi}{\mathsf{Ab}_{{\xi}}}

\newcommand{\Tor}{\mathrm{Tor}}
\newcommand{\Ext}{\mathrm{Ext}}

\newcommand{\Prim}{\mathsf{Prim}}

\newcommand{\aug}{\mathit{aug}}

\newcommand{\Dec}{\mathrm{Dec}}

\newcommand{\sk}{\mathrm{sk}}
\newcommand{\cosk}{\mathrm{cosk}}

\newcommand{\Fun}{\mathrm{Fun}}

\DeclareMathOperator{\map}{\mathrm{map}}
\DeclareMathOperator{\Map}{\mathrm{Map}}
\DeclareMathOperator{\Hom}{\mathrm{Hom}}
\newcommand{\RHom}{\mathbb{R}\mathrm{Hom}}

\DeclareMathOperator{\id}{\mathsf{id}}

\newcommand{\6}{\partial}

\newcommand{\holim}{\operatorname*{holim}}
\newcommand{\colim}{\operatorname*{colim}}

\newcommand{\fib}{\operatorname*{fib}}
\newcommand{\cofib}{\operatorname*{cofib}}

\newcommand{\Tot}{\operatorname*{Tot}}

\DeclareMathOperator{\bfLie}{\mathbf{Lie}}

\DeclareMathOperator{\F}{\mathbf F}

\DeclareMathOperator{\N}{\mathbf N}
\DeclareMathOperator{\Z}{\mathbf Z}
\DeclareMathOperator{\Q}{\mathbf Q}

\newcommand{\bbT}{\mathbb T}
\newcommand{\bbR}{\mathbb R}
\newcommand{\kk}{\mathbf F}

\newcommand{\calC}{\EuScript{C}}
\newcommand{\calF}{\EuScript{F}}

\newcommand{\calA}{\mathcal{A}}
\newcommand{\calM}{\mathcal{M}}
\newcommand{\calU}{\mathcal{U}}
\newcommand{\calK}{\mathcal{K}}

\DeclareMathOperator{\ad}{\mathrm{ad}}
\DeclareMathOperator{\im}{\mathrm{im}}

\newcommand{\e}{\varepsilon}

\DeclareMathOperator{\coker}{\mathrm{coker}}

\mathchardef\mhyphen="2D

\numberwithin{equation}{section}% reset equation counter for sections
\numberwithin{equation}{subsection}
\renewcommand*{\theequation}{%
  \ifnum\value{subsection}=0 %
    \thesection
  \else
    \thesubsection
  \fi
  .\arabic{equation}%
}

\newtheorem{prop}[equation]{Proposition}
\newtheorem*{prop*}{Proposition}
\newtheorem{lmm}[equation]{Lemma}
\newtheorem{thm}[equation]{Theorem}
\newtheorem{varthm}{Theorem}

\newtheorem*{thm*}{Theorem}

\newtheorem*{conj*}{Conjecture}
\newtheorem{cor}[equation]{Corollary}
\newtheorem*{varcor}{Corollary}
\theoremstyle{definition}
\newtheorem{dfn}[equation]{Definition}

\newtheorem{exmp}[equation]{Example}
\newtheorem*{ntt*}{Notation}
\theoremstyle{remark}
\newtheorem{rmk}[equation]{Remark}

\title{Koszul duality for simplicial restricted Lie algebras}
\author{Nikolay Konovalov}
\address{\parbox{\linewidth}{Faculty of Mathematics, HSE University, \\6 Usacheva ulitsa, Moscow 119048, Russia \\
\\
Max Planck Institute for Mathematics, \\Vivatsgasse 7, 53111 Bonn, Germany}}
\email{nikolay.konovalov.p@gmail.com, konovalov@mpim-bonn.mpg.de}
\date{}
\begin{document}
\maketitle
\begin{abstract}

Let $\mathsf{s}_0\mathsf{Lie}^r$ be the category of $0$-reduced simplicial restricted Lie algebras over a fixed perfect field of positive characteristic $p$. We prove that there is a full subcategory $\mathrm{Ho}(\mathsf{s}_0\mathsf{Lie}^r_{\xi})$ of the homotopy category $\mathrm{Ho}(\mathsf{s}_0\mathsf{Lie}^r)$ and an equivalence $\mathrm{Ho}(\mathsf{s}_0\mathsf{Lie}^r_{\xi})\simeq\mathrm{Ho}(\mathsf{s}_1\mathsf{CoAlg}^{tr})$. Here $\mathsf{s}_1\mathsf{CoAlg}^{tr}$ is the category of $1$-reduced simplicial truncated coalgebras; informally, a coaugmented cocommutative coalgebra $C$ is truncated if $x^p=0$ for any $x$ from the augmentation ideal of the dual algebra $C^*$. Moreover, we provide a sufficient and necessary condition in terms of the homotopy groups $\pi_*(L_\bullet)$ for $L_\bullet \in \mathrm{Ho}(\mathsf{s}_0\mathsf{Lie}^r)$ to lie in the full subcategory $\mathrm{Ho}(\mathsf{s}_0\mathsf{Lie}^r_{\xi})$. 

As an application of the equivalence above, we construct and examine an analog of the unstable Adams spectral sequence of A.~K.~Bousfield and D.~Kan in the category $\srLie$. We use this spectral sequence to recompute the homotopy groups of a free simplicial restricted Lie algebra.
\end{abstract}

\tableofcontents

\section{Introduction}\label{section: introduction}

%Recall that a differential graded Lie algebra $L_*$ is a chain complex $(L_*,d)$ equipped with a graded Lie bracket $[-,-]\colon L_*\times L_* \to L_*$ such that the differential $d$ satisfies the Leibniz identity.

%Let us denote by $\mathsf{DGL}_1$ the category of differential graded Lie algebras 

In~\cite[Theorem~I]{Quillen_rational}, D.~Quillen proved that there exists an equivalence of homotopy categories:
\begin{equation}\label{equation: qullen equivalence}
\mathrm{Ho}(\mathsf{DGL}_0) \simeq \mathrm{Ho}(\mathsf{DGC}_1),
\end{equation}
where $\mathsf{DGL}_0$ is the category of $0$-connected differential graded Lie algebras over the rationals $\Q$, $\mathsf{DGC}_1$ is the category of $1$-connected differential graded cocommutative coalgebras over $\Q$, and the homotopy categories are taken with respect to quasi-isomorphisms.

Later, the equivalence~\eqref{equation: qullen equivalence} was generalized in a wide range of new contexts; an interested reader might check the following long but not exhaustive list of references:~\cite{Moore70}, \cite{Freese_koszul}, \cite{GF12}, \cite{AyalaFrancis}, and~\cite{HarperChing}. We refer to phenomena like~\eqref{equation: qullen equivalence} as \emph{(derived) Koszul duality}. %The goal of this paper is to extend the Koszul duality~\eqref{equation: qullen equivalence} to a one more setting.
We notice that each of these contexts of Koszul duality from the list above can be formally regarded as relating algebras over an operad with divided power coalgebras over a cooperad.  The purpose of this paper is to study Koszul duality for restricted Lie algebras, which can be regarded as a divided power algebra over an operad, and therefore does not fit into the previously studied contexts. 

Quillen's work provides a \emph{Lie model} for the homotopy category $\mathrm{Ho}(\EuScript{S}_{\Q})$ of rational simply-connected topological spaces. The existence of a Lie model for the category $\mathrm{Ho}(\EuScript{S}^{\wedge}_p)$ of \emph{$p$-adic homotopy types} is a more difficult problem and it still remains unresolved, see e.g.~\cite{Mandell_cochains}. However, A.~K.~Bousfield and E.~Curtis~\cite{BC70} showed that there is a certain relation between $\mathrm{Ho}(\EuScript{S}^{\wedge}_p)$ and the (homotopy) category $\mathrm{Ho}(\srLie)$ of simplicial restricted Lie algebras; in particular, they constructed the unstable Adams spectral sequence whose $E_1$-term is the homotopy groups of a free simplicial restricted Lie algebra and which converges to the homotopy groups of a space. Moreover, they showed that the $E_1$-term can be expressed in terms of an explicit differential graded algebra; namely, in terms of the lambda algebra $\Lambda$. The Koszul duality described in this paper will explain and conceptualize some of their calculations.  Furthermore, the results of this paper are used in~\cite{AGSS} to study the interaction of the unstable Adams spectral sequence and the Goodwillie tower. 

%Quillen's work essentially relates the rational homotopy groups of a space to its rational homology groups through the rational unstable Adams spectral sequence.  The work of A.~K.~Bousfield and E.~Curtis~\cite{BC70} relates the mod-$p$ unstable Adams spectral sequence to simplicial restricted Lie algebras, and in particular they provided an explicit differential graded algebra (the lambda algebra $\Lambda$) whose cohomology gives the $E_2$-page of this spectral sequence.  The Koszul duality described in this paper will explain and conceptualize this mathematics.  It will be applied in [ref to paper] to study the interaction of the unstable Adams spectral sequence and the Goodwillie tower. 

\subsection{Results}\label{section: results} Throughout this paper, $p$ is a fixed prime number and $\kk$ is a fixed perfect field of characteristic $p$. Recall from~\cite[Definition~V.4]{Jacobson79} that a \emph{restricted Lie algebra} $(L,\xi)$ over $\kk$ is a Lie algebra $L$ equipped with a (non-additive, in general) \emph{$p$-operation} $\xi\colon L\to L$ (Definition~\ref{definition:p-operation}). We write $\rLie$ for the category of restricted Lie algebras (over $\kk$) and we denote by $\srLie$ the category of \emph{simplicial objects} in $\rLie$; i.e. $\srLie$ is the category of contravariant functors from the simplex category~$\Delta$ to $\rLie$. The category $\srLie$ will be the main object of this paper.

In~\cite{Hochschild54} G.~Hochschild defined the cohomology groups $H^*(L;\kk)$ for a restricted Lie algebra $L\in\rLie$, and later, his definition was extended to simplicial restricted Lie algebras by S.~Priddy in~\cite{Priddy70long}. More precisely, S.~Priddy constructed a functor
\begin{equation}\label{equation: intro, wu^r}
\barW U^r\colon \srLie \to \scoalg^{aug}
\end{equation}
such that 
$$H^*(L;\kk) \cong \Hom(\pi_*(\barW U^r(L)),\kk),\; L\in \rLie. $$
Here $\CoAlg^{aug}$ is the category of coaugmented cocommutative coalgebras over $\kk$ and $\scoalg^{aug}$ is the category of simplicial objects in $\CoAlg^{aug}$.

Let $C=(C,\eta\colon \kk \to C)\in \CoAlg^{aug}$ be a \emph{finite-dimensional} coaugmented cocommutative coalgebra and let $C^*=(C^*,\eta^*\colon C^* \to \kk)$ be its dual augmented algebra. We say that $C$ is \emph{truncated} if, for every $x\in \ker(\eta^*)$, we have $x^p=0$. An infinite-dimensional coalgebra $C\in \CoAlg^{aug}$ is called truncated if $C$ is a union of finite-dimensional truncated sub-coalgebras. We write $\trcoalg$ for the full subcategory of $\CoAlg^{aug}$ spanned by truncated ones. (In the main text, we will use a different but equivalent Definition~\ref{definition: truncated coalgebra} for truncated coalgebras.)

By~\cite[Proposition~5.10]{Priddy70long} and~\cite[Lemma~8.4]{May70operations}, the essential image of the functor $\overline{W} U^r$ is contained in the full subcategory $\strcoalg\subset \scoalg^{aug}$ of simplicial truncated coalgebras. Moreover, for every $L_\bullet\in \srLie$, the simplicial coalgebra $\barW U^r(L_\bullet)$ is \emph{reduced}; i.e. the coalgebra $\barW U^r(L_\bullet)_0$ of $0$-simplices is isomorphic to~$\kk$. 

We write $\sotrcoalg$ for the category of reduced simplicial truncated coalgebras and we say that a map $f\colon C_\bullet \to D_\bullet$ in $\sotrcoalg$ is a \emph{weak equivalence} if $f$ is a weak equivalence of underlying simplicial vector spaces, i.e. the induced map $f_*\colon \pi_*(C_\bullet) \to \pi_*(D_\bullet)$ is an isomorphism.

\begin{varthm}[Theorem~\ref{theorem: kan loop}]\label{theorem: intro, A}
The functor $\barW U^r\colon \srLie \to \sotrcoalg$ has a left adjoint
$$PG\colon \sotrcoalg \to \srLie $$
such that the unit map
\begin{equation}\label{equation: intro, unit}
\eta\colon C_\bullet \to \barW U^r \circ PG(C_\bullet) 
\end{equation}
%induced from the identity map $\id\colon PG(C_\bullet)\to PG(C_\bullet)$ by the adjunction $PG \dashv \barW U^r$
is a weak equivalence for any reduced simplicial truncated coalgebra $C_\bullet\in \sotrcoalg$.
\end{varthm}

Our proof of Theorem~\ref{theorem: intro, A} has two main ingredients. The first one is the classical observation made in~\cite{MilnorMoore65} that the category $\prHopf$ of primitively generated Hopf algebras is simultaneously  equivalent to $\rLie$ and to the category $\Grp(\trcoalg)$ of \emph{group objects} in $\trcoalg$. The second ingredient was inspired by~\cite{Stevenson12}. Namely, we adapt to our setting Stevenson's proof of the Kan theorem~\cite[Theorem~7.1]{Kan_combinatorial} (see also~\cite[Corollary~V.6.4]{GoerssJardine}), which states that the homotopy category $\mathrm{Ho}(\mathsf{s}_0\mathsf{Set})$ of reduced simplicial sets is equivalent to the homotopy category $\mathrm{Ho}(\mathsf{sGrp})$ of simplicial groups.

Similarly, we say that $f\colon L'_\bullet \to L_\bullet \in \srLie$ is a weak equivalence if $f_*\colon \pi_*(L'_\bullet) \to \pi_*(L_\bullet)$ is an isomorphism. We notice, however, that the dual of Theorem~\ref{theorem: intro, A} is not fulfilled. Namely, according to Example~\ref{example: lie algebra with trivial coeffients}, there is a simplicial restricted Lie algebra $L_\bullet \in \srLie$ such that the counit map
$$PG\circ \barW U^r(L_\bullet) \to L_\bullet $$
is not a weak equivalence in $\srLie$. Therefore we introduce the notion of an \emph{$\kk$-equivalence}: a map $f\colon L'_\bullet \to L_\bullet$ in $\srLie$ is an $\kk$-equivalence if and only if $\barW U^r(f)$ is a weak equivalence in $\sotrcoalg$ (Definition~\ref{definition: barW-equivalence}). By Corollary~\ref{corollary: barW Ur preserves weak equivalences}, any weak equivalence in $\srLie$ is an $\kk$-equivalence.

\begin{varthm}\label{theorem: intro, B}
Let $\mathcal{W}_{\srLie}$ (resp. $\mathcal{W}_{\kk}$) be the class of weak equivalences (resp. $\kk$-equivalences) in $\srLie$. Then there are model structures $(\srLie, \mathcal{W}_{\srLie}, \calC, \calF)$ and $(\srLie, \mathcal{W}_{\kk}, \calC_{\kk}, \calF_{\kk})$ on the category $\srLie$ such that
\begin{enumerate}
\item $f \in \calF$ if and only if $f$ is a fibration in $\sVect_{\kk}$ (Remark~\ref{remark: surjective on components});
\item the classes of cofibrations coincide, $\calC_{\kk}=\calC$;
\item there is an inclusion $\calF_{\kk}\subset \calF$;
\item both model structures are simplicial and combinatorial;
\item the model structure $(\mathcal{W}_{\srLie}, \calC, \calF)$ is right proper and $(\mathcal{W}_{\kk}, \calC_{\kk}, \calF_{\kk})$ is left proper.
\end{enumerate}
\end{varthm}

Theorem~\ref{theorem: intro, B} is a combination of Theorems~\ref{theorem:modelsrlie} and~\ref{theorem:model structure srlie, barW-equivalence} from the main text. We will abuse notation and denote by $\srLie$ (resp. by $\srLie_\xi$) the model category $(\srLie, \mathcal{W}_{\srLie}, \calC, \calF)$ (resp. $(\srLie, \mathcal{W}_{\kk}, \calC_{\kk}, \calF_{\kk})$) from Theorem~\ref{theorem: intro, B}. We notice that the model category $\srLie_\xi$ is a (left) Bousfield localization  of $\srLie$ (\cite[Definition~3.3.1]{Hirschhorn03}), and so the homotopy category $\mathrm{Ho}(\srLie_\xi)$ is a full subcategory of $\mathrm{Ho}(\srLie)$. It follows from Theorem~\ref{theorem: intro, A} that functors $\barW U^r$ and $PG$ induce an equivalence of homotopy categories:
\begin{equation}\label{equation: intro, trcoalg vs F-complete} 
\mathrm{Ho}(\sotrcoalg) \simeq \mathrm{Ho}(\srLie_\xi).
\end{equation}
Moreover, there is a simplicial combinatorial model structure on $\sotrcoalg$ and the equivalence~\eqref{equation: intro, trcoalg vs F-complete} can be enhanced to an equivalence between the underlying $\infty$-categories (see \cite[Section~A.2]{HTT}) of simplicial model categories $\sotrcoalg$ and $\srLie_\xi$. The next theorem follows from Propositions~\ref{proposition: coalgebras and lie, infty-categorical, part1},~\ref{proposition: F-complete is an accessible localization}, and Theorem~\ref{theorem: coalgebras and lie algebras}. 

\begin{varthm}\label{theorem: intro, C}
Let us denote by $\sL$ (resp. $\sLxi$, $\sCA_0$) the underlying $\infty$-category of the simplicial model category $\srLie$ (resp. $\srLie_\xi$, $\sotrcoalg$). Then there is an equivalence of presentable $\infty$-categories:
$$\sCA_0 \simeq \sLxi$$
and $\sLxi\subset \sL$ is a localization of the presentable $\infty$-category $\sL$.
\end{varthm}

Our next goal is to describe the full subcategory $\sLxi\subset \sL$ together with the \emph{$\kk$-completion} functor $$L_\xi \colon \sL \to \sLxi \hookrightarrow \sL,$$
see Definition~\ref{definition: F-completion}. Since any object in $\srLie$ is fibrant, this problem is equivalent to identifying fibrant objects and fibrant replacements in the model category $\srLie_\xi$. It seems to be difficult in general, so we restrict ourselves to the case of \emph{connected} simplicial restricted Lie algebras. Here we say that a simplicial restricted Lie algebra $L_\bullet\in \srLie$ is connected if $\pi_0(L_\bullet)=0$.

Let $L_\bullet\in\srLie$ be a simplicial restricted Lie algebra. The $p$-operation $\xi\colon L_\bullet \to L_\bullet$ is a map of simplicial sets, and so it induces a map of homotopy groups:
$$\xi_*\colon \pi_{n}(L_\bullet) \to \pi_n(L_\bullet), \; n\geq 0 $$
which is additive for $n\geq 1$. Moreover, since the $p$-operation $\xi$ is semi-linear, the map $\xi_*$ is semi-linear as well, i.e. $\xi_*(ax)=a^p\xi_*(x)$, $a\in \kk$, $x\in \pi_*(L_\bullet)$. In this way, all homotopy groups $\pi_n(L_\bullet), n\geq 1$ are naturally left modules over the ring of twisted polynomials $\kk\{\xi\}$, see Definition~\ref{definition: twisted polynomial ring}. 

The ring $\kk\{\xi\}$ is non-commutative (if $\kk\neq \F_p$), however it still shares a lot of common properties with the polynomial ring $\kk[t]$, see Section~\ref{section: xi-complete modules}. In particular, one can still define the \emph{$\xi$-adic completion}
$$\widehat{M}=\lim_r M/\xi^r(M)$$
of a left $\kk\{\xi\}$-module $M\in \Mod_{\kk\{\xi\}}$, see Definition~\ref{definition: xi-completion}. %We show that the functor $M \mapsto \widehat{M}$ is exact for finitely generated modules (Proposition~\ref{proposition: xi-completion is exact}). However, for all left $\kk\{\xi\}$-modules, 
The $\xi$-adic completion is not an exact functor, and so we introduce in Section~\ref{section: derived xi-complete modules} its \emph{left derived functors} $L_0$ and $L_1$. The functor $L_0$ is equipped with a natural transformation
\begin{equation}\label{equation: intro, completion transform}
\phi_M\colon M\to L_0(M),\; M\in \Mod_{\kk\{\xi\}};
\end{equation}
and we say that a left $\kk\{\xi\}$-module $M$ is \emph{derived $\xi$-adic complete} if $L_1(M)=0$ and the map $\phi_M$ is an isomorphism.

We say that $L_\bullet \in \srLie$ is \emph{$\kk$-complete} if $L_\bullet$ is a fibrant object of $\srLie_\xi$, see Section~\ref{section: F-completion}.
\begin{varthm}[Corollary~5.3.12]\label{theorem: intro, D}
Let $L_\bullet \in \srLie$ be a connected simplicial restricted Lie algebra, $\pi_0(L_\bullet)=0$. Then $L_\bullet$ is $\kk$-complete if and only if all homotopy groups $\pi_n(L_\bullet), n\geq 1$ are derived $\xi$-adic complete left modules over the ring $\kk\{\xi\}$.
\end{varthm}

Let $\sCA_1$ be the $\infty$-category of $1$-connected simplicial truncated coalgebras, i.e. $C_\bullet \in \sCA_1$ if and only if $\pi_0(C_\bullet)\cong \kk$ and $\pi_1(C_\bullet)=0$. Combining Theorem~\ref{theorem: intro, C} with Theorem~\ref{theorem: intro, D} yields the following corollary. 
\begin{varcor}
There is an equivalence of $\infty$-categories 
$$\sCA_1 \simeq \sL_{\xi,0},$$
where $\sL_{\xi,0}$ is the full subcategory of $\sL$ spanned by connected simplicial restricted Lie algebras whose homotopy groups are derived $\xi$-complete. 
\end{varcor}

We note that it seems plausible to derive the last corollary from the results of~\cite{BM19} on the deformation theory of simplicial commutative $\kk$-algebras. The comparison between this paper and~\cite{BM19} will be presented elsewhere. 

\subsection{Applications}\label{section: applications} At the end of the paper, we provide several applications of Theorem~\ref{theorem: intro, C}. Following~\cite{Priddy70long}, we define the reduced \emph{cohomology groups} $\widetilde{H}^*(L_\bullet;\kk)$ of a simplicial restricted Lie algebra $L_\bullet\in \srLie$ by the formula
$$\widetilde{H}^q(L_\bullet;\kk)=\Hom(\pi_q(\barW U^r(L_\bullet);\kk), \; q\geq 1 \;\; \text{and} \;\; \widetilde{H}^0(L_\bullet;\kk)=0. $$
By the Eilenberg-Zilber theorem, we observe that $\widetilde{H}^*(L_\bullet;\kk)$ is a non-unital graded commutative algebra over $\kk$. Moreover, by~\cite[Proposition~5.3]{Priddy70long} and~\cite[Theorem~8.5]{May70operations}, there is an action of the Steenrod operations $\beta^{\e}P^a, a\geq 0, \e=0,1$ (resp. $Sq^a, a\geq 0$ if $p=2$) on the cohomology groups $\widetilde{H}^*(L_\bullet;\kk)$ such that $P^0$ (resp. $Sq^0$) acts by zero. These Steenrod operations are semi-linear and still satisfy the Adem relations, see Section~\ref{section: steenrod}. So, the cohomology groups $\widetilde{H}^*(L_\bullet;\kk)$ is also a left module (as an $\F_p$-vector space) over the \emph{homogenized mod-$p$ Steenrod algebra}~$\calA^h_p$, see Definition~\ref{definition: homogenized steenrod algebra}. We recall that the classical mod-$p$ Steenrod algebra~$\calA_p$ has $P^0=1$ (resp. $Sq^0=1$), and so Adem relations in $\calA_p$ have both quadratic and linear parts, while Adem relations in $\calA^h_p$ have only the quadratic part (Remark~\ref{remark: filtration on steenrod}).

We say that a positively graded (non-unital) commutative $\kk$-algebra $A_*=\oplus_{q>0}A_q$ is an \emph{unstable $\calA^h_p$-algebra} (Definition~\ref{definition: unstable algebra}) if $A_*$ is equipped with a semi-linear action of the homogenized mod-$p$ Steenrod algebra $\calA^h_p$ such that the following non-stability relations are satisfied:
\begin{enumerate}
\item $\beta^{\e}P^a(x)=0$ if $2a+\e>|x|$ (resp. $Sq^a(x)=0$ if $a>|x|$ and $p=2$);
\item $P^a(x)=x^p$ if $2a=|x|$ (resp. $Sq^a(x)=x^2$ if $a=|x|$).
\end{enumerate}
%We write $\calU^h$ for the category of unstable $\calA^h_p$-algebras and
We observe that a cohomology ring $\widetilde{H}^*(L_\bullet;\kk), L_\bullet\in \srLie$ is an unstable $\calA^h_p$-algebra (Example~\ref{example: steenrod algebra, restricted Lie}). We use Theorem~\ref{theorem: intro, C} together with~\cite[Proposition~6.2.1]{Priddy73} in order to construct an analog of the Bousfield-Kan spectral sequence~\cite{BK72_spectral_sequence} in the setting of the category $\srLie$.

\begin{varthm}[Corollary~\ref{corollary: BKSS, restricted Lie algebra}]\label{theorem: intro, E}
Let $L_\bullet$ be an $\kk$-complete simplicial restricted Lie algebra such that its cohomology groups $\widetilde{H}^*(L_\bullet;\kk)$ are degreewise finite-dimensional. Then there is a completely convergent spectral sequence
\begin{equation*}
E^2_{s,t}=\Ext^s_{\calU^h}(\widetilde{H}^*(L_\bullet;\kk), \Sigma^{t+1} \kk) \Rightarrow \pi_{t-s}(L_\bullet), \;\; d^r\colon E^r_{s,t} \to E^r_{s+r,t+r-1}.
\end{equation*}
Here $\Ext^*_{\calU^h}$ are non-abelian $\Ext$-groups in the category of unstable $\calA^h_p$-algebras, see Definition~\ref{definition: unstable exts}.
\end{varthm}

Finally, we use Theorem~\ref{theorem: intro, E} to tie together two classical computations. Let $\free(V_\bullet)\in\srLie$ be a free simplicial restricted Lie algebra (Example~\ref{example:freelie}) generated by a simplicial vector space $V_\bullet\in \sVect_{\kk}$. The homotopy groups 
$$\pi_*(\free(V_\bullet)) $$
were computed in~\cite[Theorem~8.5]{BC70} and~\cite[Proposition~13.2]{Wellington82} in terms of the algebra $\Lambda$ of~\cite{6authors} and $\pi_*(V_\bullet)$. At the same time, by~\cite[Section~7]{Priddy70}, the algebra $\Lambda$ is anti-isomorphic to the Koszul dual algebra $\calK^*_p$ of $\calA^h_p$, see Section~\ref{section: unstable koszul}. In Corollaries~\ref{corollary: ASS, degenerates, p=2} and~\ref{corollary: ASS, degenerates, p is odd, l is odd}, we use the Curtis theorem~\cite{Curtis_lower}, Theorem~\ref{theorem: intro, E}, and the paper~\cite{Priddy70} to redo the computations of A.~K.~Bousfield, E.~Curtis, and R.~Wellington. Our approach is not easier than theirs, but perhaps, it is more fundamental and flexible for plausible generalizations. In particular, we derive the following theorem from Corollaries~\ref{corollary: ASS, degenerates, p=2} and~\ref{corollary: ASS, degenerates, p is odd, l is odd}.

\begin{varthm}\label{theorem: intro, F}
Let $V_\bullet\in \sVect_{\kk}$ be a simplicial vector space such that $\pi_*(V_\bullet)$ is one-dimensional. Then the spectral sequence of Theorem~\ref{theorem: intro, E} 
\begin{equation}\label{equation: intro, ASS}
E^2_{s,t}=\Ext^s_{\calU^h}(\widetilde{H}^*(\free(V_\bullet);\kk), \Sigma^{t+1} \kk) \Rightarrow \pi_{t-s}(L_\xi \free(V_\bullet))
\end{equation}
degenerates at the second page. Here $L_\xi \free(V_\bullet)$ is the $\kk$-completion of the free simplicial restricted Lie algebra $\free(V_\bullet)$.
\end{varthm}

The assumption that $\pi_*(\free(V_\bullet))$ is one-dimensional is essential here; if it is not fulfilled, then it seems likely that this spectral sequence is highly non-trivial, see Remark~\ref{remark: ASS, highly dimensionsional}.

\subsection{Organization}\label{section: organization}

%The paper is organized as follows. 
In Section~\ref{section: algebraic background} we recall crucial facts about the categories $\rLie$ and $\trcoalg$ needed to prove Theorem~\ref{theorem: intro, A}. First, in Section~\ref{section: restricted lie algebras} we recall that there is an equivalence
\begin{equation*}%\label{equation: intro, primitive and ur}
\begin{tikzcd}
P: \prHopf \arrow[shift left=.6ex]{r}
&\rLie :U^r. \arrow[shift left=.6ex,swap]{l}
\end{tikzcd}
\end{equation*}
between the category of restricted Lie algebras $\rLie$ and the category $\prHopf$ of primitively generated Hopf algebras. Here $U^r$ is the universal enveloping algebra functor and $P$ is the functor of primitive elements. After that, in Section~\ref{section: truncated coalgebras} we define truncated coalgebras and prove several basic facts about them, and in Section~\ref{section: primitively generated Hopf algebras} we show that the category $\prHopf$ is also equivalent to the category $\mathsf{Grp}(\trcoalg)$ of \emph{group objects} in $\trcoalg$ (Corollary~\ref{corollary: group objects in trcoalg}). Finally, in Proposition~\ref{proposition: free group objects}, we prove that the forgetful functor from $\prHopf$ to $\trcoalg$ has a left adjoint $H$ and the algebra $H(C), C\in \trcoalg$ is free associative if we forget about the comultiplication in $H(C)$.

In Section~\ref{section: kan loop functor} we adapt the argument from~\cite{Stevenson12} to our context and prove Theorem~\ref{theorem: intro, A} as Theorem~\ref{theorem: kan loop}.

In Section~\ref{section: model structures} we construct simplicial combinatorial model structures on the categories $\srLie$ and $\sotrcoalg$, see Theorems~\ref{theorem:modelsrlie} and~\ref{theorem:modelsotrcoalg} respectively. We also introduce the notion of an $\kk$-equivalence (Definition~\ref{definition: barW-equivalence}) and construct the model category $\srLie_\xi$ in Theorem~\ref{theorem:model structure srlie, barW-equivalence}. These results imply Theorem~\ref{theorem: intro, B}. At the end of Section~\ref{section: model structures}, we prove Propositions~\ref{proposition: coalgebras and lie, infty-categorical, part1},~\ref{proposition: F-complete is an accessible localization}, and Theorem~\ref{theorem: coalgebras and lie algebras}, which together imply Theorem~\ref{theorem: intro, C}.

In Section~\ref{section:homotopy theory} we provide technical tools needed for the proof of Theorem~\ref{theorem: intro, D}. In Section~\ref{section: principal fibration} we introduce the notion of \emph{principal fibrations} in the category $\srLie$ (Definition~\ref{definition: lie principal fibration}). Then we show that a connected simplicial restricted Lie algebra has a \emph{Postnikov tower} (Corollary~\ref{corollary: postnikov tower}) and each stage in this tower is weakly equivalent to a principal fibration (Corollary~\ref{corollary: fibration with KM is principal}). Finally, in Section~\ref{section: serre spectral sequence} we construct an analog of the \emph{Serre spectral sequence} for principal fibrations (Corollary~\ref{corollary: SSS fiber is EM}) in the category $\srLie$.

%We prove Theorem~\ref{theorem: intro, D} in Section~\ref{section: F-complete}. 
In Section~\ref{section: xi-complete modules} we prove a few basic facts about the ring of twisted polynomials $\kk\{\xi\}$ and we define the $\xi$-adic completion in Definition~\ref{definition: xi-completion}. In Section~\ref{section: derived xi-complete modules} we define the left derived functors for the $\xi$-adic completion. Finally, we prove Theorem~\ref{theorem: intro, D} in Section~\ref{section: F-completion} as Corollary~\ref{corollary: F-complete are derived Xi-complete}. Our proof is based on the classical proof (given e.g. in~\cite[Theorem~11.1.1]{MayPonto}) that a simply-connected space $X$ is $p$-complete if and only if its homotopy groups $\pi_n(X),n\geq 2$ are derived $p$-complete.

In Section~\ref{section:adams} we illustrate possible applications of the previous results. In Section~\ref{section: steenrod} we recall properties of the Steenrod operations and we show that the cohomology ring $\widetilde{H}^*(L_\bullet;\kk), L_\bullet\in\srLie$ is a left $\calA^h_p$-module over the homogenized mod-$p$ Steenrod algebra $\calA^h_p$, see Example~\ref{example: steenrod algebra, restricted Lie}. Moreover, we introduce the category $\calU^h$ of unstable $\calA^h_p$-algebra (Definition~\ref{definition: unstable algebra}), the category $\calM^h$ of \emph{unstable $\calA^h_p$-modules} (Definition~\ref{definition: unstable module}), and the category $\calM^h_0$ of \emph{strongly unstable $\calA^h_p$-modules} (Definition~\ref{definition: strongly unstable module}). Both categories $\calM^h$ and $\calM^h_0$ are abelian and closely related to $\calU^h$, see Remark~\ref{remark: unstable algebras and unstable modules}. In Section~\ref{section: BKSS} we prove Theorem~\ref{theorem: intro, E} as Corollary~\ref{corollary: BKSS, restricted Lie algebra}.

In Section~\ref{section: unstable koszul} we recall the definition of the lambda algebra $\Lambda$ of~\cite{6authors}. (We point out that in this work we use the convention for $\Lambda$ from~\cite[Definition~7.1]{Wellington82} but not from the original paper.) Then, we compute unstable abelian $\Ext$-groups $\Ext^s_{\calM^h}(W,\Sigma^t\kk)$ and $\Ext^s_{\calM^h_0}(W,\Sigma^t\kk)$ in terms of the algebra $\Lambda$ for a trivial $\calA^h_p$-module $W\in \Vect^{gr}_{\kk}$ (Corollary~\ref{corollary: unstable exts between trivial}). 

In Section~\ref{section: free lie} we apply the spectral sequence of Theorem~\ref{theorem: intro, E} to a free simplicial restricted Lie algebra $L_\bullet=\free(V_\bullet)$ generated by a simplicial vector space $V_\bullet\in \sVect_\kk$. Finally, we derive Theorem~\ref{theorem: intro, F} from the Curtis theorem~\cite{Curtis_lower} and previous computations.

\subsection{Acknowledgments} %The author would like to thank Dmitry Kaledin for posing the problem and Mark Behrens for helpful discussions during the preparation of the work. The author is also grateful to Kirill Magidson and Artem Prikhodko for pointing out a crucial misconception in an early version of the paper.
The author is grateful to Mark Behrens, Dmitry Kaledin, Kirill Magidson, and Artem Prikhodko for many helpful conversations. The author also thanks the Max Planck Institute for Mathematics in Bonn for its hospitality and financial support.

\subsection{Notation}\label{section: notation} Here we describe some notation that will be used throughout the paper. As was said, $p$ is a fixed prime number and $\kk$ is a fixed perfect field of characteristic $p$. We denote by $\Vect_{\kk}$ the category of vector spaces over $\kk$; $\Vect^{gr}_{\kk}$ (resp. $\Vect^{>0}_{\kk}$) is the category of non-negatively (resp. positively) graded vector spaces over $\kk$. We usually denote by $V_*=\oplus_{q\geq 0} V_q$ an object of $\Vect^{gr}_{\kk}$. 

Throughout most of the paper, except Section~\ref{section: serre spectral sequence}, we write $\Sigma V_*$ for the shift of $V_*\in \Vect^{gr}_{\kk}$, i.e. $(\Sigma V_*)_{q} = V_{q-1}$, $q\geq 0. $ Moreover, $\Sigma^tV_*=\Sigma(\Sigma^{t-1}V_*)$, $t\geq 0$; we extend this notation to all integers in an usual way. In Section~\ref{section: serre spectral sequence}, for better readability of formulas, we denote the shift $\Sigma^t V_*$ by $V_*[t]$.

We denote by $\CoAlg$ the category of cocommutative coalgebras over $\kk$ and we denote by $\Alg$ the category of associative algebras over $\kk$. If it is not said otherwise, all algebras are unital and all coalgebras are counital. We denote by $\CoAlg^{aug}$ the category of coaugmented cocommutative coalgebras; an object of $\CoAlg^{aug}$ is a pair $(C,\eta\colon \kk \to C)$, where $C\in \CoAlg$ and $\eta$ is a map of coalgebras. 

In this paper, all Hopf algebras are cocommutative, but not necessary commutative; we denote by $\Hopf$ the category of (cocommutative) Hopf algebras over~$\kk$. We notice that $\Hopf$ is equivalent to the category $\Grp(\CoAlg)$ of group objects in $\CoAlg$, since the direct product of cocommutative coalgebras $C$ and $D$ in the category $\CoAlg$ is the tensor product $C\otimes_\kk D$, see e.g.~\cite[Theorem~6.4.5]{Sweedler69}.

%In this paper, we include the alternating condition $[x,x]=0$ in the definition of a Lie algebra. %$$L$ over $\kk$ is called \emph{isotropic} if $[x,x]=0$ for all $x\in L$.. Note that if $p\neq 2$, then any Lie algebra is isotropic.

We define Lie algebras so that they satisfy the alternating condition $[x,x]=0$, which always implies the antisymmetry $[x,y]=-[y,x]$, but is equivalent to it only over a field of characteristic $p\neq 2$. We write $\Lie$ for the category of Lie algebras over the field $\kk$. 

We denote by $\Delta$ the simplex category of finite nonempty linearly ordered sets. We denote by $[n], n\geq 0$ the object of $\Delta$ with $(n+1)$ elements. Let $\mathsf{C}$ be a category, then  $\mathsf{sC}$ (resp. $\mathsf{cC}$) is the category of simplicial (resp. cosimplicial) objects in $\mathsf{C}$, i.e. $\mathsf{sC}$ (resp. $\mathsf{cC}$) is the category of contravariant (resp. covariant) functors from $\Delta$ to~$\mathsf{C}$: $$\mathsf{sC}=\Fun(\Delta^{op},\mathsf{C}), \;\; \text{and} \;\; \mathsf{cC}=\Fun(\Delta, \mathsf{C}).$$
We usually denote by $X_\bullet$ (resp. $X^\bullet$) an object of $\mathsf{sC}$ (resp. $\mathsf{cC}$), where $X_n=X_\bullet([n])$ (resp. $X^n=X^\bullet([n])$) for $n\geq 0$. If the category $\mathsf{C}$ has a terminal object $*\in C$, then $\mathsf{s}_0\mathsf{C}\subset \mathsf{sC}$ is the full subcategory of reduced simplicial objects, $X_\bullet \in \mathsf{s}_0\mathsf{C}$ if and only if $X_0\cong *$. This notation can be nested, e.g. $\mathsf{csC}$ is the category of cosimplicial simplicial objects in $\mathsf{C}$.

Moreover, if the category $\mathsf{C}$ is complete and cocomplete, then the category $\mathsf{sC}$ is enriched, tensored, and cotensored over $\sSet$. Therefore there is a canonical notion of homotopy between maps in $\mathsf{sC}$, see~\cite[Definition~4, Section~II.1]{Quillen67}. In this case, we define (strong) deformation retracts in $\mathsf{sC}$ in a usual way. If $\mathsf{sC}$ is a simplicial model category, then deformation retracts are weak equivalences.

Let us denote by $\mathsf{Ch}_{\geq 0}$ (resp. $\mathsf{Ch}^{\geq 0}$) the category of connective chain (resp. cochain) complexes over $\kk$. An object of $\mathsf{Ch}_{\geq 0}$ (resp. $\mathsf{Ch}^{\geq 0}$) is a pair $C_\bullet=(C_*,d)$ (resp. $C^\bullet = (C^*, d)$), where $C_\bullet \in \Vect^{gr}_{\kk}$ and $d$ is a differential 
$$d\colon C_{n+1}\to C_{n} \;\; \text{(resp. $d\colon C^{n} \to C^{n+1}$)}, \; n\geq 0$$ such that $d^2=0$. We denote by $H_*(C_\bullet)$ (resp. $H^*(C^\bullet)$) the homology (resp. cohomology) groups of $C_\bullet$ (resp. $C^\bullet$).

Let $V_\bullet \in \sVect_{\kk}$ be a simplicial vector space. We denote by $\pi_*(V_\bullet)$ the homotopy groups of the underlying simplicial set. The normalized chain complex functor 
$$N\colon \sVect_{\kk} \to \mathsf{Ch}_{\geq 0}$$
is given by $$(NV_\bullet)_n = \bigcap_{i=1}^n \ker(d_i), \;\; d=d_0\colon (NV_\bullet)_n \to (NV_\bullet)_{n-1}.$$
By the Dold-Kan correspondence, $N$ is an equivalence and we denote by $\Gamma\colon \mathsf{Ch}_{\geq 0} \to \sVect_{\kk}$ its inverse. Moreover, we recall that there is an isomorphism $$\pi_*(V_\bullet) \cong H_*(NV_\bullet), \; V_\bullet \in \sVect_{\kk}. $$

We write $F\dashv G$ or 
\begin{equation*}
\begin{tikzcd}
F: \mathsf{C} \arrow[shift left=.6ex]{r}
&\mathsf{D} :G \arrow[shift left=.6ex,swap]{l}
\end{tikzcd}
\end{equation*}
if the functor $F$ is the left adjoint to the functor $G$.  We write $\oblv$ (abbrv. to obliviate) for various forgetful functors. We note that an underlying vector space is not always a result of applying $\oblv$. For example, suppose $(C,\eta)\in \CoAlg^{aug}$ is a coaugmented coalgebra, then the underlying vector space of $(C,\eta)$ is $C$, but $\oblv (C,\eta)$ is $\coker(\eta)$.

\section{Algebraic background}\label{section: algebraic background}
In this section we provide algebraic background for restricted Lie algebras, primitively generated Hopf algebras, and truncated coalgebras. At the end of Section~\ref{section: restricted lie algebras}, we recall that the category $\rLie$ of restricted Lie algebras is equivalent to the category $\prHopf$ of primitively generated Hopf algebras. In Section~\ref{section: truncated coalgebras}, for a cocommutative coalgebra $C$, we define the \emph{Verschiebung} operator $V\colon C \to C$ (Definition~\ref{definition: coalgebra verschiebung}); $C$ is called \emph{truncated} if $V$ is trivial (Definition~\ref{definition: truncated coalgebra}). The main results of Section~\ref{section: truncated coalgebras}
are Propositions~\ref{proposition: cofree truncated coalgebra} and~\ref{proposition:category property of trcoalg}, where we describe a cofree truncated coalgebra $\Sym^{tr}(W)$ and show that the category $\trcoalg$ of truncated coalgebras is locally presentable. In Section~\ref{section: primitively generated Hopf algebras} we derive from~\cite[Proposition~4.20]{MilnorMoore65} that the category $\prHopf$ is equivalent to the category $\Grp(\trcoalg)$ of group objects in $\trcoalg$ (Proposition~\ref{proposition: truncated and primitively generated}). Finally, in Proposition~\ref{proposition: free group objects}, we describe a \emph{free Hopf algebra} $H(C)$ generated by $C\in\trcoalg$.

\subsection{Restricted Lie algebras}\label{section: restricted lie algebras} 
Let $L$ be a Lie algebra and let $x\in L$. We denote by $\ad(x)\colon L\to L$ the map given by $y\mapsto \ad(x)(y)=[y,x]$.

\begin{dfn}\label{definition:p-operation}
Let $L$ be a Lie algebra over $\kk$. A \emph{$p$-operation} on $L$ is a map $\xi\colon L \to L$ such that 

\begin{itemize}
\item $\xi(ax)=a^p\xi(x)$, $a\in \kk, x\in L$;
\item $\ad(\xi(x))=\ad(x)^{\circ p}\colon L \to L$;
\item $\xi(x+y)=\xi(x)+\xi(y)+\sum_{i=1}^{p-1}\dfrac{s_i(x,y)}{i}$, for all $x,y\in L$, where $s_{i}(x,y)$ is the coefficient of $t^{i-1}$ in the formal expression $\ad(tx+y)^{\circ(p-1)}(x)$.
\end{itemize}

\end{dfn}

\begin{dfn}[{{\cite[Definition~V.4]{Jacobson79}}}]\label{definition:lier} A \emph{restricted Lie algebra} $(L,\xi)$ is a Lie algebra $L$ (over $\kk$) equipped with a $p$-operation $\xi\colon L \to L$. A linear map $$f\colon (L',\xi_{L'})\to (L,\xi_L)$$ is a homomorphism of restricted Lie algebras if
$$[f(x),f(y)]=f([x,y]), \; x,y\in L', $$
and $\xi_L(f(x))=f(\xi_{L'}(x))$, $x\in L'$.
We will denote by $\rLie$ the category of restricted Lie algebras.
\end{dfn}

%\begin{exmp}\label{example:abelian}
%We say that a restricted Lie algebra $L$ is \emph{abelian}, if $L$ has trivial bracket and $\xi=0$. We denote by $\Ab(V)$ a unique abelian restricted Lie algebra with the underlying vector space equal to $V\in \Vect_{\kk}$.
%\end{exmp}

%We generalize the previous example as follows. 
Recall that a Lie algebra $L$ is called \emph{abelian} if $L$ is equipped with zero bracket. First, we will describe abelian restricted Lie algebras.

\begin{dfn}\label{definition: twisted polynomial ring}
The \emph{twisted polynomial ring} $\kk\{\xi\}$ is defined as the set of polynomials in the variable $\xi$ and coefficients in $\kk$. It is endowed with a ring structure with the usual addition and with a non-commutative multiplication that can be summarized with the relation: $$\xi a = a^{p} \xi, \; \; a\in \kk.$$
We denote by $\Mod_{\kk\{\xi\}}$ (resp. $\Mod^{\kk\{\xi\}}$) the abelian category of \emph{left}  (resp. \emph{right}) $\kk\{\xi\}$-modules.
\end{dfn}

The full subcategory of abelian restricted Lie algebras is equivalent to $\Mod_{\kk\{\xi\}}$ because, if $L$ is an abelian restricted Lie algebra, then the $p$-operation $\xi\colon L\to L$ is additive. We denote by $\trivxi(M)$ a unique abelian restricted Lie algebra with the underlying left $\kk\{\xi\}$-module equal to $M\in \Mod_{\kk\{\xi\}}$. %This notation will be explained {\bf LATER}. 
Finally, we say that an abelian restricted Lie algebra is \emph{$p$-abelian} if the $p$-operation $\xi\colon L\to L$ is trivial, i.e. $\xi=0$.

We proceed with more examples of restricted Lie algebras.

\begin{exmp}\label{example:assislie}
Given an associative $\kk$-algebra $A$, we write $A^{\circ}$ for the restricted Lie algebra whose underlying vector space is $A$ equipped with the bracket $[x,y]=xy-yx$ and the $p$-operation $\xi(x)=x^p$.
\end{exmp}

\begin{dfn}\label{definition: primitive elements}
Let $(C, \eta\colon \kk \to C)$ be a coaugmented cocommutative coalgebra over $\kk$ with comultiplication $\Delta\colon C\to C\otimes C$. Recall that an element $x\in C$ is called \emph{primitive} if
$$\Delta(x)=1\otimes x + x\otimes 1. $$
Here  $1=\eta(1)\in C$. We denote by $P(C)$ the set of primitive elements in $C$.
\end{dfn}

\begin{exmp}\label{example:hopfislie}
Let $H$ be a cocommutative Hopf algebra over $\kk$. Then the set of primitive elements $P(H)$ is a restricted Lie subalgebra in $H^{\circ}$, see Example~\ref{example:assislie}. %with the bracket $[x,y]=xy-yx$ and the $p$-operation $\xi(x)=x^p$.
\end{exmp}

\begin{exmp}\label{example:freelie}
Let $V$ be a vector space over $\kk$ and let $T(V)$ be the tensor algebra generated by $V$. It is well-known that $T(V)$ has a unique structure of a cocommutative Hopf algebra such that generators $v\in V$ are primitive elements. Therefore $P(T(V))$ is a restricted Lie algebra, which we will denote by $\free(V)$.

We say that a restricted Lie algebra $L$ is \emph{free} if $L$ is isomorphic to $\free(V)$ for some $V\in \Vect_{\kk}$. This terminology is justified because the functor $$\free\colon \Vect_{\kk} \to \rLie$$ is the left adjoint to the forgetful functor $$\oblv\colon \rLie \to \Vect_{\kk}.$$
\end{exmp}

\begin{prop}\label{proposition: free lie algebras, fresse}
The underlying vector space of a free restricted Lie algebra $\free(V),V\in \Vect_{\kk}$ splits as follows
$$\oblv\circ\free(V)\cong \bigoplus_{n\geq 1}L^r_n(V) = \bigoplus_{n\geq 1} (\bfLie_n \otimes V^{\otimes n})^{\Sigma_n}.$$
Here $\bfLie_n \in \Vect_{\kk}$ is the $n$-th space of the Lie operad.
\end{prop}

\begin{proof}
\cite[Theorem~1.2.5]{Fresse00}. 
\end{proof}

The next proposition is standard, cf.~\cite[Chapter 6]{MilnorMoore65}.
\begin{prop}\label{proposition:category property of rLie} The category $\rLie$ is monadic over $\Vect_{\kk}$ via the adjunction $\free \dashv \oblv$. The category $\rLie$ is complete and cocomplete and the forgetful functor $\oblv$ creates limits and sifted colimits. Moreover, $\rLie$ is locally presentable. \qed
\end{prop}

For instance, the direct product $L_1\times L_2$ is the direct sum $L_1\oplus L_2$ as a vector space, with $[l_1,l_2]=0,l_1\in L_1, l_2\in L_2$ and $p$-operation acting componentwise. We fix the following observation for later purposes (Proposition~\ref{proposition: cofibrant in srlie with action}).

\begin{prop}\label{proposition: pushouts and products in rlie}
Let 
$$
\begin{tikzcd}
L \arrow{r} \arrow{d}
& L_1 \arrow{d} \\
L_2 \arrow{r}
& L_{12}
\end{tikzcd}
$$
be a pushout square in $\rLie$ and let $M\in \rLie$ be a restricted Lie algebra. Then the commutative diagram 
$$
\begin{tikzcd}
M\times L \arrow{r} \arrow{d}
& M\times L_1 \arrow{d} \\
M\times L_2 \arrow{r}
& M\times L_{12}
\end{tikzcd}
$$
is again a pushout square in $\rLie$.
\end{prop}

\begin{proof}
We show that the natural map
$$(M\times L_1) \coprod_{M\times L} (M\times L_2) \to M\times (L_1\coprod_L L_2) = M\times L_{12} $$
is an isomorphism. Let $$f\colon M\coprod_M M \to (M\times L_1) \coprod_{M\times L} (M\times L_2)$$ be the map of restricted Lie algebras induced by the maps from the zero Lie algebra to $L$, $L_1$, and $L_2$, and let 
$$g\colon M \to M\times L_{12}, \; g(m)=(m,0), \; m\in M$$ be the canonical embedding. Consider the following commutative diagram
$$
\begin{tikzcd}%[column sep=large]
	M \coprod_M M \arrow[shift left=.75ex]{r}{f}
\arrow[shift right=.75ex,swap]{r}{0} \arrow{d}{\cong}
&
(M\times L_1) \coprod_{M\times L} (M\times L_2) \arrow{r} \arrow{d}
&
L_1\coprod_L L_2 \arrow{d}{\cong}\\
	M \arrow[shift left=.75ex]{r}{g}
\arrow[shift right=.75ex,swap]{r}{0}
&
M\times L_{12} \arrow{r}
&
L_{12}.
\end{tikzcd}
$$
Since $g(M)$ is a (restricted) Lie ideal in $M\times L_{12}$, the bottom row is a coequalizer diagram both in~$\rLie$ and in~$\Vect_{\kk}$. Furthermore, the top row is also a coequalizer diagram in $\rLie$ because pushouts commute with coequalizers. Since the outer vertical arrows are isomorphisms and $f$ is a split monomorphism, it suffices to show that the top row is also a coequalizer diagram in~$\Vect_{\kk}$, i.e. the image $\im(f)$ is a (restricted) Lie ideal in the pushout $(M\times L_1) \coprod_{M\times L} (M\times L_2)$. However, this pushout is generated as a restricted Lie algebras by $\im(f)$, $L_1$, and $L_2$. Finally, we observe that $[m,l]=0$ for any $m \in \im(f)$ and $l\in L_i$, $i=1,2$, which imply the proposition.
\end{proof}

\begin{dfn}\label{definition:UEA}
The \emph{universal enveloping algebra} of $(L,\xi)\in \rLie$ is the quotient algebra $$U^r(L)=T(L)/\langle x\otimes y -y\otimes x -[x,y],\;\;z^p-\xi(z),\;\; x,y,z\in L\rangle.$$
\end{dfn}

This construction is natural and the functor $U^r\colon \rLie\to \Alg$ is the left adjoint to the functor $(-)^{\circ} \colon \Alg \to \rLie$ of Example~\ref{example:assislie}. In particular, $U^r\free(V)\cong T(V)$.

We also note that $U^r$ preserves colimits and takes direct products to the tensor product of algebras. In particular, $U^r(L)$ is an augmented cocommutative Hopf algebra via the diagonal map $\Delta\colon L\to L\times L$.

We write $\Hopf$ for the category of cocommutative Hopf algebras. By the previous paragraph, we have the functor $$U^r\colon \rLie \to \Hopf,$$ and by Example~\ref{example:hopfislie}, we have the functor
$$P\colon \Hopf \to \rLie, \; H\mapsto P(H) $$
in the opposite direction. 

\begin{dfn}\label{definition: primitevely generated}
We say that a Hopf algebra $H\in \Hopf$ is \emph{primitively generated} if the subset of primitive elements  $P(H)\subset H$ generates $H$, i.e. the natural map $T(P(H)) \to H$ is surjective. We denote by $\prHopf$ the full subcategory of $\Hopf$ spanned by primitively generated Hopf algebras.
\end{dfn}

Since $U^r(L)$ is a quotient of a tensor algebra by a Hopf ideal, the Hopf algebra $U^r(L)$ is primitively generated for all $L\in \rLie$. %Thus we obtain the adjoint pair:
%\begin{equation}\label{equation: primitive universal enveloping adjunction}
%\begin{tikzcd}
%P: \prHopf \arrow[shift left=.6ex]{r}
%&\rLie :U^r. \arrow[shift left=.6ex,swap]{l}
%\end{tikzcd}
%\end{equation}

\begin{thm}\label{theorem: hopf are equivalent to lie}
The functors $U^r\colon \rLie \to \prHopf$ and $P\colon \prHopf \to \rLie$ are inverse to each other. In other words, they provide the following equivalence of categories
$$
\begin{tikzcd}
P: \prHopf \arrow[shift left=.6ex]{r}
&\rLie :U^r. \arrow[shift left=.6ex,swap]{l}
\end{tikzcd}
$$
%The adjoint pair~\eqref{equation: primitive universal enveloping adjunction} is an equivalence of categories.
\end{thm}

\begin{proof} \cite[Theorem~6.11]{MilnorMoore65}.
\end{proof}

\subsection{Truncated coalgebras}\label{section: truncated coalgebras}

Recall that $\mathrm{char}(\kk)=p$. 
\begin{dfn}\label{definition: frobenius twist}
Let $W\in \Vect_{\kk}$ be a vector space over $\kk$. We define the \emph{$(-1)$-th Frobenius twist} $W^{(-1)}\in \Vect_{\kk}$ of $W$ as follows. As an abelian group, $W^{(-1)}=W$ and we endow it with a new $\kk$-action 
$$-\cdot-\colon \kk \times W^{(-1)} \to W^{(-1)} $$
given by $a\cdot w = a^p w$, $a\in \kk$, $w\in W^{(-1)}=W$. Since the field $\kk$ is perfect, there also exists the inverse operation; namely, we define the \emph{Frobenius twist} of $W$ as a unique $\kk$-vector space $W^{(1)}$ such that $(W^{(1)})^{(-1)}= W$.
\end{dfn}

For any vector space $W\in \Vect_{\kk}$ we have a natural $\kk$-linear map 
\begin{equation}\label{equation: linear frobenius}
W^{(1)} \to \mathrm{Sym}^p(W)=(W^{\otimes p})_{\Sigma_p}
\end{equation}
sending $w$ to $w^{\otimes p}$. If $W$ is finite-dimensional, we have the dual map
\begin{equation}\label{equation: linear verschiebung}
\Gamma^p(W)=(W^{\otimes p})^{\Sigma_p} \to W^{(1)},
\end{equation}
and we extend it to all vector spaces to taking filtered colimits.

\begin{dfn}\label{definition: coalgebra verschiebung}
Let $C$ be a cocommutative coalgebra over $\kk$ with comultiplication $\Delta\colon C\to C\otimes C$ and counit $\e\colon C \to \kk$. We define the \emph{Verschiebung} operator $V\colon C\to C^{(1)}$ as follows:
$$V\colon C\xrightarrow{\Delta^{p-1}} (C^{\otimes p})^{\Sigma_p} \xrightarrow{\eqref{equation: linear verschiebung}} C^{(1)}. $$
We note that $V$ is a $\F_p$-linear coalgebra homomorphism.
\end{dfn}

\begin{dfn}\label{definition: truncated coalgebra}
A cocommutative coalgebra $C$ is called \emph{truncated} if $$\ker(V)=\ker(\e),$$ where $\e\colon C\to \kk$ is the counit. We write $\trcoalg$ for the full subcategory in the category of cocommutative coalgebras $\CoAlg$ spanned by \emph{non-zero} truncated ones.
\end{dfn}

\begin{rmk}\label{remark: algebras with zero frobenius}
Let $C$ be a finite-dimensional coalgebra, and let $C^{*}$ be its dual algebra. Then $C$ is truncated if and only if for any $x\in C^*$ the $p$-th power $x^p$ is some scalar multiple of the unit $1\in C^*$. 
\end{rmk}

\begin{exmp}\label{example: trivial coalgebra} 
Let $W$ be a vector space over $\kk$. We define the \emph{trivial} coalgebra $\triv(W)$ as follows. The underlying vector space of $\triv(W)$ is $W\oplus \kk$, where the second summand is spanned by the element $1\in \kk$. We define the comultiplication $\Delta\colon \triv(W)\to \triv(W)\otimes \triv(W)$ and the counit $\e\colon \triv(W)\to \kk$ as follows $$\Delta(1)=1, \;\; \Delta(w)=1\otimes w + w\otimes 1, w\in W;$$
$$\e(w)=0,w\in W, \;\; e(1)=1.$$ This is easy to see that the coalgebra $\triv(W)$ is truncated.
\end{exmp}

\begin{dfn}[Chapter~8, \cite{Sweedler69}]\label{definition: simple and pointed} A non-zero coalgebra $C\in \CoAlg$ is called \emph{simple} if $C$ has no non-zero proper subcoalgebras and $C\in \CoAlg$ is called \emph{pointed} if all simple subcoalgebras of $C$ are $1$-dimensional. 
\end{dfn}

\begin{prop}\label{proposition: truncated are pointed and irriducible}
$\quad$
\begin{enumerate}
\item Any simple truncated coalgebra is $1$-dimensional. In particular, any truncated coalgebra is pointed.
\item Any non-zero truncated coalgebra has a \emph{unique} simple subcoalgebra.
\end{enumerate}
\end{prop}

\begin{proof}
If $C$ is simple, then $C$ is finite-dimensional, and the dual algebra $C^*$ is a finite field extension of $\kk$, see~\cite[Lemma~8.0.1]{Sweedler69}. In particular, the Frobenius map $$(-)^p \colon C^* \to C^*$$ is injective. However, since the coalgebra $C$ is truncated, the image of this map is $\kk\subseteq C^*$ by Remark~\ref{remark: algebras with zero frobenius}. Therefore $C^*\cong \kk$. This proves the first part.

For the second part, we show that $C\in \trcoalg$ has at most one grouplike element. Recall that $0\neq x \in C$ is called \emph{grouplike} if 
$$\Delta(x)=x\otimes x; $$
and the set $G(C)\subset C$ of grouplike elements in $C$ one-to-one corresponds to simple subcoalgebras of $C$, see~\cite[Lemma~8.0.1]{Sweedler69}. If $x\in C$ is grouplike, then $V(x)=x$ and $\e(x)=1$. Therefore, if $x,y\in C$ are grouplike, then $x-y\in \ker(\e)$. Since $C$ is truncated, this implies $V(x-y)=0$. So, we obtain
$$x-y=V(x)-V(y)=V(x-y)=0. $$
Since any non-zero coalgebra contains a simple subcoalgebra, the proposition follows.
\end{proof}

In particular, there is an exactly one coalgebra map $\kk \to C$ for any $C\in \trcoalg$; that is any non-zero truncated coalgebra is \emph{canonically} coaugmented. Therefore we will consider the category $\trcoalg$ as a full subcategory of coaugmented cocommutative coalgebras $\CoAlg^{aug}$.

\begin{prop}\label{proposition: left and right adjoints, truncated}
The fully faithful embedding
$$\trcoalg \subset \CoAlg^{aug} $$
has both left and right adjoints. In particular, the category $\trcoalg$ has all limits, which can be computed in $\CoAlg^{aug}$. Similarly, for all colimits in $\trcoalg$.
\end{prop}

\begin{proof}
The left adjoint $l\colon \CoAlg^{aug} \to \trcoalg$ can be given as the coequalizer of the following diagram:
$$
\begin{tikzcd}%[column sep=large]
C^{(-1)} \arrow[shift left=.75ex]{r}{V^{(-1)}}
  \arrow[shift right=.75ex,swap]{r}
&
C \arrow{r}
&
l(C),
\end{tikzcd}
$$
where $C\in \CoAlg^{aug}$, and the lower arrow is the composition the counit map and the coaugmentation map. Similarly, the right adjoint $r\colon \CoAlg^{aug} \to \trcoalg$ can be defined as the equalizer of the same diagram:
$$
\begin{gathered}[b]
\begin{tikzcd}%[column sep=large]
r(C) \arrow{r}
&
C \arrow[shift left=.75ex]{r}{V}
  \arrow[shift right=.75ex,swap]{r}
&
C^{(1)}.
\end{tikzcd}\\[-\dp\strutbox]
\end{gathered}\eqno\qedhere
$$
\end{proof}

\begin{dfn}\label{definition: conilpotent coalgebra}
A non-counital coalgebra $(C,\Delta\colon C\to C\otimes C)$ is called \emph{conilpotent} if for any $x\in C$ there exists $n$ such that $$\Delta^{n-1}(x)=0,$$ where $\Delta^{n-1}\colon C \to C^{\otimes n}$ is the $(n-1)$-fold composition of the comultiplication. A coaugmented coalgebra $(C,\eta\colon \kk \to C)$ is called conilpotent if the non-counital coalgebra $\coker(\eta)$ is conilpotent.
\end{dfn}

\begin{lmm}\label{lemma: criteria for conilpotent}
A finite-dimensional coaugmented coalgebra $(C,\eta)$ is conilpotent if and only if the dual augmented algebra $(C^*,\eta^*)$ has the \emph{nilpotent} augmentation ideal $I=\ker(\eta^*)\subset C^*$, i.e. $I^n=0$ for some $n\in \N$. \qed
\end{lmm}

\begin{prop}\label{proposition: truncated coalgebra is conilpotent}
Any truncated coalgebra $C$ is conilpotent (as a coaugmented coalgebra).
\end{prop}

\begin{proof}
Let $x\in C$, we have to show that $\Delta^{n-1}(x)=0 \in (C/\kk)^{\otimes n}$ for some $n\in \N$. We can assume that $C$ is generated by $x$. Then by~\cite[Theorem~2.2.1]{Sweedler69}, the coalgebra $C$ is finite-dimensional. Therefore, by Lemma~\ref{lemma: criteria for conilpotent}, it suffices to show that the augmentation ideal $I$ in the dual algebra $C^*$ is nilpotent. Since $C$ is truncated, $x^p=0$ for any $x\in I$. Then $I^{pm}=0$, where $m=\dim(I)$.
\end{proof}

Let $W$ be a vector space over $\kk$. Then the symmetric algebra $\Sym(W)$ is a (cocommutative) Hopf algebra, which is truncated as a coalgebra. Since $w^p\in \Sym(W)$ is a primitive element for any $w\in W$, the ideal
$$I =(w^p\;|\; w\in W)\subset \Sym(W) $$
generated by all $p$-th powers is a Hopf ideal. Therefore $\Sym^{tr}(W)=\Sym(W)/I$ is a Hopf algebra, which is truncated as a coalgebra.

\begin{prop}\label{proposition: cofree truncated coalgebra}
The functor $$\Sym^{tr}\colon \Vect_{\kk} \to \trcoalg $$ is the right adjoint to the forgetful functor $$\oblv\colon \trcoalg \hookrightarrow \CoAlg^{aug} \to \Vect_{\kk} $$ given by $(C,\eta\colon\kk \to C) \mapsto \coker(\eta)$.
\end{prop}

\begin{proof}
By Propositions~\ref{proposition: left and right adjoints, truncated} and~\ref{proposition: truncated coalgebra is conilpotent}, the right adjoint $R\colon \Vect_{\kk} \to \trcoalg$ for the functor $\oblv$ is the equalizer of the diagram
$$ 
\begin{tikzcd}%[column sep=large]
R(W) \arrow{r}
&
\Gamma(W) \arrow[shift left=.75ex]{r}{V}
  \arrow[shift right=.75ex,swap]{r}
&
\Gamma(W)^{(1)},
\end{tikzcd}
$$
where $W\in \Vect_{\kk}$ and 
$$\Gamma(W)=\bigoplus_{i=0}^{\infty} (W^{\otimes i})^{\Sigma_i}$$ is a cofree conilpotent cocommutative coalgebra. We recall that $\Gamma(W)$ is a commutative Hopf algebra also known as the divided power algebra generated by $W$. Explicitly, $\Gamma(W)$ is generated by elements $\gamma_{n}(w)$, $w\in W$, $n\geq0$ subjects to relations:
\begin{itemize}
\item $\gamma_0(w)=1$
\item $\gamma_n(aw)=a^n \gamma_n(w)$, $w\in W$, $a\in \kk$;
\item $\gamma_n(w+w')=\sum_{i=0}^n\gamma_i(w)\gamma_{n-i}(w')$, $w,w'\in W$;
\item $\gamma_m(w)\gamma_n(w)=\dfrac{(m+n)!}{m!n!}\gamma_{m+n}(w)$, $w\in W$;
\item $\Delta(\gamma_n(w))=\sum_{i=0}^{n}\gamma_i(w)\otimes \gamma_{n-i}(w)$, $w\in W$.
\end{itemize}

In particular, $V(\gamma_{pn}(w))=\gamma_n(w)$ and $R(W)$ is the sub-Hopf algebra generated by $\gamma_0(w),\ldots,\gamma_{p-1}(w), w\in W$. The latter is isomorphic to $\Sym^{tr}(W)$.
\end{proof}

\begin{prop}\label{proposition:category property of trcoalg} The category $\trcoalg$ is comonadic over $\Vect_{\kk}$ via the adjunction $\oblv \dashv \Sym^{tr}$. The category $\trcoalg$ is complete and cocomplete and the forgetful functor $\oblv$ creates colimits. Moreover, $\trcoalg$ is locally presentable. \qed
\end{prop}

\begin{exmp}\label{example: coalgebras product and coproduct}
If $C,D\in \trcoalg$, then their Cartesian product $C\times D$ is the tensor product $C\otimes_{\kk} D$ equipped with the obvious comultiplication. The coCartesian coproduct $C\sqcup D$ is the wedge sum of augmented vector spaces $(C\oplus D)/\kk$ equipped with the obvious comultiplication.
\end{exmp}

\subsection{Primitively generated Hopf algebras}\label{section: primitively generated Hopf algebras}
Let $L$ be a restricted Lie algebra, then the universal enveloping algebra $U^r(L)$ is naturally equipped with the increasing \emph{Lie filtration} $F_\bullet$, which is inductively defined as follows:
\begin{itemize}
\item $F_sU^r(L)=0$ if $s<0$;
\item $F_0U^r(L)=\kk$;
\item $F_{s+1}U^r(L)=F_sU^r(L)+L\cdot F_sU^r(L) \subset U^r(L)$.
\end{itemize}
We denote by $E_0U^r(L)$ the associated graded Hopf algebra. Since $E_0U^r(L)$ is commutative, there is a natural map
\begin{equation}\label{equation: associated graded of universal enveloping}
\Sym^{tr}(L) \to E_0U^r(L).
\end{equation}

\begin{thm}[Poincar\'{e}-Birkhoff-Witt]\label{theorem: poincare-birkhoff-witt}
The homomorphism~\eqref{equation: associated graded of universal enveloping} is an isomorphism of Hopf algebras.
\end{thm}

\begin{proof}
\cite[Proposition~6.12]{MilnorMoore65}.
\end{proof}

%Moreover, since both $U^r(L)$ and $E_0U^r(L)$ are generated by primitive elements $L\subset F_1U^r(L)$, we obtain that $U^r(L)$ is isomorphic to $E_0U^r(L)$ as a coalgebra. Therefore, there exists a natural coalgebra isomorphism:
%\begin{equation}\label{equation: hopf algebra is cofree}
%\Sym^{tr}(L) \xrightarrow{\cong} U^r(L). 
%\end{equation}

%\begin{cor}\label{corollary: primitively generated are cofree}
%Any primitively generated Hopf algebra is a cofree truncated coalgebra. \qed
%\end{cor}

\begin{prop}\label{proposition: truncated and primitively generated}
A Hopf algebra $H$ is primitively generated if and only if $H$ is truncated as a cocommutative coalgebra.
\end{prop}

\begin{proof} 
%Dual to Proposition~4.20 in~\cite{MilnorMoore65}. 
%If $H$ is primitively generated, then $H=U^r(P(H))$ by Theorem~\ref{theorem: hopf are equivalent to lie}. By Corollary~\ref{corollary: primitively generated are cofree}, $U^r(P(H))$ is a truncated coalgebra.
Let $H$ be a primitively generated Hopf algebra, and let $I\subset H$ be the augmentation ideal. We observe that $V(I)=0$ because $I$ is generated by the set of primitive elements $P(H)\subset I$, the Verschiebung operator $V\colon H \to H^{(1)}$ is a Hopf algebra homomorphism, and $V(P(H))=0$. %, we get $V(I)=0$, 
Therefore $H$ is a truncated coalgebra.

Suppose now that $H$ is a truncated coalgebra. It suffices to show that the natural morphism
\begin{equation}\label{equation: proof, equation 1}
U^r(P(H)) \to H 
\end{equation}
is surjective. %By Theorem~\ref{theorem: poincare-birkhoff-witt}, the map~\eqref{equation: proof, equation 1} is an injection, so it suffices to show that~\eqref{equation: proof, equation 1} is surjective. 
For a coaugmented cocommutative coalgebra $C$, one can consider a natural (increasing) \emph{conilpotent} filtration on $C$:
$$
\begin{tikzcd}%[column sep=large]
F_{(n)}C=\mathrm{eq}(C \arrow[shift left=.75ex]{r}{\Delta^{n}}
  \arrow[shift right=.75ex,swap]{r}
&
C^{\otimes n+1}).
\end{tikzcd}
$$
A coalgebra $C$ is conilpotent if and only if the conilpotent filtration $F_{(\bullet)}C$ is exhaustive. Moreover, for a Hopf algebra $H$ the filtration $F_{(\bullet)}H$ is both multiplicative and comultiplicative. Therefore the associated graded Hopf algebra $E_{(\bullet)} H$ is a connected graded Hopf algebra, $P(E_{(\bullet)} H)=P(H)$.

By Proposition~\ref{proposition: truncated coalgebra is conilpotent}, both Hopf algebras $U^rP(H)$ and $H$ are conilpotent, hence the map~\eqref{equation: proof, equation 1} is surjective if and only if the map
$$
E_{(\bullet)}U^rP(H) \to E_{(\bullet)}H
$$
is surjective; or equivalently, if and only if the induced map on the module of indecomposable elements
\begin{equation}\label{equation: proof, equation 2}
Q(E_{(\bullet)}U^rP(H)) \to Q(E_{(\bullet)}H)
\end{equation}
is surjective. 

Since $H$ is truncated as a coalgebra, the Hopf algebra $E_{(\bullet)} H$ is truncated as well. The linear dual of~\cite[Proposition~4.20]{MilnorMoore65} implies that the algebras $E_{(\bullet)} H$, $E_{(\bullet)}U^rP(H)$ are primitively generated. Finally, the commutative diagram
$$
\begin{tikzcd}
P(E_{(\bullet)}U^rP(H)) \arrow[equal]{r} \arrow[two heads]{d}
& P(E_{(\bullet)}H) \arrow[two heads]{d} \\
Q(E_{(\bullet)}U^rP(H)) \arrow{r}
& Q(E_{(\bullet)}H)
\end{tikzcd}
$$
implies that the map~\eqref{equation: proof, equation 2} is surjective.
\end{proof}

Let $\mathsf{C}$ be a 1-category with finite products. Recall that a \emph{group object} of $\mathsf{C}$ is an object $X\in \mathsf{C}$ equipped with a multiplication map $m\colon X\times X \to X$, a unit map $e\colon *\to X$, and with an inverse map $i\colon X\to X$, such that $m$ is associative and unital. We denote by $\Grp(\mathsf{C})$ the category of group objects in $\mathsf{C}$.

\begin{cor}\label{corollary: group objects in trcoalg}
The category $\Grp(\trcoalg)$ is equivalent to the category $\prHopf$ of primitively generated Hopf algebras. \qed
\end{cor}

Finally, we will describe free group objects in $\Grp(\trcoalg)$.

\begin{dfn}\label{definition: smash product coalgebras}
Let $C,D\in \trcoalg$ be truncated coalgebras. We define the \emph{smash product} $C\wedge D\in \trcoalg$ as the following coequalizer
$$
\begin{tikzcd}%[column sep=large]
C\wedge D = \mathrm{coeq}(C \sqcup D \arrow[shift left=.75ex]{r}{\mathrm{can}}
  \arrow[shift right=.75ex,swap]{r}
&
C \times D),
\end{tikzcd}
$$
where the upper arrow is the canonical map from coproduct to product and the lower arrow is the composition the counit map and the coaugmentation map. In particular, $\oblv(C\wedge D)\cong \oblv(C)\otimes \oblv(D)$
\end{dfn}

\begin{prop}\label{proposition: free group objects}
The forgetful functor $\prHopf \to \trcoalg$ has a left adjoint $$H\colon \trcoalg \to \prHopf.$$ Moreover, as an algebra $H(C), C\in\trcoalg$ is the tensor algebra $T(\oblv(C))$ generated by the quotient vector space $\oblv(C)=C/\kk$.
\end{prop}

\begin{proof}
The tensor algebra $T(\oblv(C))$ has a unique bialgebra structure 
%$$\Delta \colon T(\ker(\e)) \to T(\ker(\e))\otimes T(\ker(\e)) $$
%$T(\ker(\e))$ 
such that the map $C \to T(\oblv(C))$ is a coalgebra homomorphism. We denote the resulting bialgebra by $H(C)$. More precisely, as a coaugmented coalgebra $H(C)$ is isomorphic to
$$\coprod_{n=0}^{\infty} C^{\wedge n} = \coprod_{n=0}^{\infty} C^{\otimes n} / \sim. $$ In particular, $H(C)$ is truncated, and so $H(C)$ is the free bialgebra generated by $C$. Therefore it is enough to show that $H(C)$ admits an antipode (which is unique if it exists).

By Proposition~\ref{proposition: truncated are pointed and irriducible}, the bialgebra $H(C)$ is pointed and has a unique grouplike element. By~\cite[Proposition~9.2.5]{Sweedler69}, this implies that $H(C)$ is a Hopf algebra.
\end{proof}

\section{Koszul duality}\label{section: koszul duality}

In this section we state and prove our main results summarized in Theorems~\ref{theorem: intro, A},~\ref{theorem: intro, B}, and~\ref{theorem: intro, C} from the introduction. 

Section~\ref{section: kan loop functor} is devoted to the proof of Theorem~\ref{theorem: intro, A}. By~\cite{Priddy70long}, the functor $\barW U^r$ is the composite of two functors
$$\srLie \xrightarrow{U^r} \sprHopf  \xrightarrow{\barW} \sotrcoalg, $$
and by~\cite[Theorem~6.11]{MilnorMoore65}, the leftmost arrow here is an equivalence. Inspired by~\cite{Stevenson12}, we present the functor $\barW$ as the composite of simpler functors:
$$\overline{W} \colon \sprHopf \xrightarrow{N} \sostrcoalg \xhookrightarrow{\iota} \sstrcoalg \xrightarrow{T} \strcoalg, $$
see Definition~\ref{definition: twisted bar construction} and Remark~\ref{remark: barW is Eilenberg-MacLane}. Here $N$ is the \emph{nerve functor}~\eqref{equation: nerve functor}, $\iota$ is the embedding~\eqref{equation: horizontally reduced} of the full subcategory of \emph{horizontally reduced} bisimplicial coalgebras into all, and $T$ is the \emph{Artin-Mazur codiagonal}~\eqref{equation: artin-mazur codiagonal}.
Therefore the left adjoint $G$ to $\barW$ is the composite
$$G\colon \strcoalg \xrightarrow{\Dec} \sstrcoalg \xrightarrow{R} \sostrcoalg \xrightarrow{GZ} \sprHopf,$$
where $\Dec$ is the \emph{total d\'{e}calage}~\eqref{equation: total decalage}, $R$ is given in~\eqref{equation: left adjoint to iota}, and $GZ$ is the functor from~Proposition~\ref{proposition: GZ fundamental group}. 

By~\cite{OR20}, the unit map $C_\bullet \to T\Dec(C_\bullet)$ is a weak equivalence for any $C_\bullet\in \strcoalg$ (Theorem~\ref{theorem: artin-mazur codiagonal}). We examine the adjunction $GZ\dashv N$ in Lemma~\ref{lemma: GZ of a suspension} and Corollary~\ref{corollary: decalage is EM}, and we show that the unit map for this adjunction is a weak equivalence for certain simplicial coalgebras. These observations and formal properties of functors $\Dec$ and $T$ suffice to prove the main theorem of this section, Theorem~\ref{theorem: kan loop}.

The main results of Section~\ref{section: model structures} are Theorems~\ref{theorem:modelsrlie},~\ref{theorem:modelsotrcoalg}, and~\ref{theorem:model structure srlie, barW-equivalence}, where we construct simplicial model structures on the categories $\srLie$ and $\sotrcoalg$ equipped with various notions of weak equivalences. The model structure of Theorem~\ref{theorem:modelsrlie} (resp. of Theorem~\ref{theorem:modelsotrcoalg}) is a right (resp. a left) transferred model structure from $\sVect_{\kk}$; for the proof, we apply~\cite[Theorem~11.3.2]{Hirschhorn03} (resp.~\cite[Theorem~2.2.1]{HKRS17}).

In Definition~\ref{definition: barW-equivalence}, we introduce the notion of an \emph{$\kk$-equivalence}. Following \cite{Priddy70long}, we demonstrate in Corollary~\ref{corollary: barW Ur preserves weak equivalences} and Proposition~\ref{proposition: chain coalgebra and chain complex} that the class of $\kk$-equivalences is well-behaved. Using~\cite[Proposition~A.2.6.15]{HTT} we construct a simplicial model category $\srLie_\xi$ as a Bousfield localization of $\srLie$. At the end of the section, we prove Theorem~\ref{theorem: coalgebras and lie algebras} and Proposition~\ref{proposition: F-complete is an accessible localization}, which constitute together Theorem~\ref{theorem: intro, C} from the introduction.

\subsection{Kan loop group functor}\label{section: kan loop functor}

\begin{dfn}\label{definition: simplicial coalgebras}
Let $\strcoalg=\Fun(\Delta^{op}, \trcoalg)$ denote the category of \emph{simplicial objects} in the category $\trcoalg$ of non-zero truncated coalgebras over the field~$\kk$. A simplicial coalgebra $C_\bullet\in \strcoalg$ is \emph{reduced} if $C_0\cong \kk$. Let $\sotrcoalg$ denote the full subcategory of $\strcoalg$ spanned by reduced objects.
\end{dfn}

\begin{dfn}\label{definition: nerve}
	Let $H\in \prHopf$ be a primitively generated Hopf algebra and let $\e\colon H\to \kk$ be the counit. The \emph{nerve} of $H$  is a reduced simplicial coalgebra $N_\bullet H$ defined as follows:
$$N_0H=\kk,\;\; N_{q}H=H^{\otimes q},\;\; q> 0; $$
with face and degeneracy maps given by
$$d_0(h_1\otimes\ldots \otimes h_q)=\e(h_1)h_{2}\otimes \ldots \otimes h_q, \;\; q>0; $$
$$d_{i}(h_1\otimes\ldots\otimes h_q)=h_1\otimes \ldots \otimes h_{i-1} \otimes h_{i}h_{i+1}\otimes \ldots \otimes h_q, \;\; 1 \leq i\leq q-1,q>0; $$
$$d_q(h_1\otimes\ldots \otimes h_q)=\e(h_q)h_{1}\otimes \ldots \otimes h_{q-1}, \;\; q>0; $$
$$s_0(h_1\otimes\ldots\otimes h_q) = 1 \otimes h_1\otimes \ldots \otimes h_q,\;\; q\geq 0;$$
$$s_{i+1}(h_1\otimes\ldots\otimes h_q) = h_1 \otimes \ldots \otimes h_{i}\otimes 1 \otimes h_{i+1}\otimes \ldots \otimes h_q,\;\; i,q\geq 0. $$
\end{dfn}

By Proposition~\ref{proposition: truncated and primitively generated}, all coalgebras $N_qH,q\geq 0$ are truncated, and so we obtain the \emph{nerve functor} 
\begin{equation}\label{equation: nerve functor}
N\colon \prHopf \to \sotrcoalg.
\end{equation}
As usual, we observe that the simplicial set $N_\bullet H$ is $2$-coskeletal, see e.g.~\cite[Proposition~2.2.3]{Illusie72}. 

\begin{prop}\label{proposition: GZ fundamental group}
The nerve functor $N$ has a left adjoint 
$$GZ\colon \sotrcoalg \to \prHopf. $$
Namely, $GZ(C_\bullet), C_\bullet\in \sotrcoalg$ is the quotient of free Hopf algebra $H(C_1)$ of Proposition~\ref{proposition: free group objects} by the two-sided Hopf ideal $I\subset H(C_1)$ generated by elements
$$r(c)=d_1(c)-\sum_{i}d_0(c^{(1)}_i)d_2(c^{(2)}_i), $$
where $c\in C_2$, and $\Delta(c)=\sum_{i}c^{(1)}_i \otimes c^{(2)}_i$.
\end{prop}

\begin{proof}
The straightforward computation shows that the set of simplicial coalgebra maps $\Hom_{\sotrcoalg}(C_\bullet,N_\bullet H)$ is a subset of $\Hom_{\trcoalg}(C_1,H)$, and a map $f\in \Hom_{\trcoalg}(C_1,H)$ belongs to $\Hom_{\sotrcoalg}(C_\bullet,N_\bullet H)$ if and only if 
$$f(d_1(c))=\sum_if(d_0(c^{(1)}_i))f(d_2(c^{(2)}_i)) $$
for each $c\in C_2$, $\Delta(c)=\sum_{i}c^{(1)}_i\otimes c^{(2)}_i$. Now, the proposition follows by Proposition~\ref{proposition: free group objects}, which guarantees the existence of a free Hopf algebra.
\end{proof}

\begin{rmk}\label{remark: GZ set of relations}
Note that $r(c_1+c_2)=r(c_1)+r(c_2),c_1,c_2\in C_2$, and $r(c)=0$ if $c\in C_2$ is degenerate, i.e. $c=s_i(c'), c'\in C_1, i=0,1$.
\end{rmk}

\begin{lmm}\label{lemma: GZ indecomposables}
The module of indecomposable elements $Q(GZ(C_\bullet))$ for the Hopf algebra $GZ(C_\bullet)$, $C_\bullet \in \strcoalg$ is naturally isomorphic to $\pi_1(C_\bullet)$.
\end{lmm}

\begin{rmk}\label{remark: GZ is a fundamental group}
For the category $\mathsf{Grp}$ of groups the nerve construction is right adjoint to the fundamental group functor. Likewise, we regard $GZ$ as a fundamental group construction.
\end{rmk}

\begin{exmp}\label{example: GZ of suspension}
Let $C\in \trcoalg$ be a truncated coalgebra. We will abbreviate $\oblv(C)\in \Vect_\kk$ by $\widetilde{C}$, $C=\widetilde{C}\oplus \kk$. %the kernel $\ker(\e)$ of the counit $\e\colon C\to \kk.$ 
Recall from~\cite[Section~III.5]{GoerssJardine} that the \emph{Kan suspension} $\Sigma_\bullet C$ of $C$ is a reduced simplicial coalgebra defined as follows:
$$\Sigma_0C=\kk,\;\; \Sigma_{q}C=\coprod_{i=1}^{q} C\cong \kk\oplus \bigoplus_{i=1}^q \widetilde{C},\;\; q> 0; $$
with face and degeneracy maps constant on the first summand and given by the following formulas on the second summand:
$$d_0(c_1,\ldots , c_q)= d_q(c_1,\ldots, c_q) =0, \;\; q>0; $$
$$d_{i}(c_1,\ldots, c_q)=(c_1,\ldots, c_{i-1}, c_{i}+c_{i+1}, \ldots, c_q), \;\; 1 \leq i\leq q-1,q>0; $$
$$s_0(c_1,\ldots, c_q) = (0, c_1, \ldots, c_q),\;\; q\geq 0;$$
$$s_{i+1}(c_1,\ldots, c_q) = (c_1, \ldots, c_{i}, 0, c_{i+1}, \ldots, c_q),\;\; i,q\geq 0. $$
Notice that the homotopy groups $\pi_*(\Sigma_\bullet C)$ of the underlying simplicial vector space are equal to:
$$\pi_*(\Sigma_\bullet C)=\left\{
\begin{array}{ll}
\kk & \mbox{if $\ast=0$,}\\
\widetilde{C} & \mbox{if $\ast=1$,}\\
0 & \mbox{if $\ast>1$.}
\end{array}
\right.
$$
By Remark~\ref{remark: GZ set of relations}, we obtain that $GZ(\Sigma_\bullet C)\cong H(C)$, the free Hopf algebra generated by $C$.
\end{exmp}

\begin{dfn}\label{definition: weak equivalence of simplicial coalgebras}
A map $f\colon C_\bullet \to D_\bullet $ of simplicial coalgebras $C_\bullet,D_\bullet \in \scoalg$ is a \emph{weak equivalence} if $f$ is a weak equivalence of the underlying simplicial vector spaces, i.e. $\pi_*(f)$ is an isomorphism.
\end{dfn}

\begin{lmm}\label{lemma: GZ of a suspension}
Let $C\in \trcoalg$, then the natural map
$$\Sigma_\bullet C \to N_\bullet GZ(\Sigma_\bullet C) $$
is a weak equivalence of simplicial coalgebras.
\end{lmm}

\begin{proof}
First, note that the underlying simplicial vector space of $N_\bullet H, H\in \prHopf$ only depends on the underlying augmented algebra structure of $H$. Second, by Example~\ref{example: GZ of suspension}, the Hopf algebra $GZ(\Sigma_\bullet C)$ is the free Hopf algebra $H(C)$, and by Proposition~\ref{proposition: free group objects}, there is an algebra isomorphism $H(C)\cong T(\oblv(C)).$ Finally, for the free associative algebra $T(\oblv(C))$, we have $$\pi_i(N_\bullet T(\oblv(C)))=0, \; i>1,$$ and a natural isomorphism $\pi_1(N_\bullet T(\oblv(C)))\cong \oblv(C)$. %This implies the lemma.
\end{proof}

\begin{dfn}\label{definition: bisimplicial coalgebras}
Let $\sstrcoalg$ denote the category of \emph{bisimplicial} truncated coalgebras over~$\kk$, i.e. $\sstrcoalg=\Fun(\Delta^{op}\times\Delta^{op}, \trcoalg)$. A bisimplicial map $f\colon C_{\bullet,\bullet} \to D_{\bullet,\bullet}$ is a \emph{vertical} weak equivalence if the maps $$f_{n,\bullet}\colon C_{n,\bullet} \to D_{n,\bullet} $$ are weak equivalences for each $n\geq 0$. Similarly, $f$ is a \emph{horizontal} weak equivalence if the maps $$f_{\bullet, m}\colon C_{\bullet, m} \to D_{\bullet, m} $$ are weak equivalences for each $m\geq 0$. Finally, $f$ is a \emph{levelwise} weak equivalence if $f$ is either a vertical or a horizontal weak equivalence.
\end{dfn}

Let $\sigma\colon \Delta \times \Delta \to \Delta$ denote the ordinal sum functor on the category $\Delta$, described on objects via $\sigma([n],[m])=[n+1+m]$. The induced functor 
\begin{equation}\label{equation: total decalage}
\Dec = \sigma^*\colon \strcoalg \to \sstrcoalg 
\end{equation}
is called \emph{total d\'{e}calage}, see~\cite[Chapitre~VI.1.5]{Illusie72}. In other words, the coalgebra of $(n,m)$-bisimplices of $\Dec(C_\bullet)$ is $\Dec(C_\bullet)_{n,m}=C_{n+1+m}$ and the degeneracy and face maps in $\Dec(C_\bullet)_{\bullet,\bullet}$ are recombinations of those for $C_\bullet\in \strcoalg$.

We denote by $\Dec_m\colon \strcoalg \to \strcoalg$ a functor given by $$C_\bullet \mapsto \Dec(C)_{\bullet, m}.$$ Note that
\begin{equation}\label{equation: partial decalage, 1}
\Dec_{m+1}(C_\bullet)=\Dec_0(\Dec_m(C_\bullet)).
\end{equation}
Moreover, the simplicial coalgebra $\Dec_0(C_\bullet)$ is an augmented simplicial object via the map $d_0\colon \Dec_0(C_\bullet) \to C_0$, e.g. see~\cite[Section~2]{Stevenson12}. In particular, $d_0$ is a (strong) deformation retract in simplicial coalgebras.

Let $C_\bullet\in \strcoalg$ be a simplicial coalgebra and let $\sk_1 C_\bullet \in \strcoalg$ be the $1$-skeleton of $C_\bullet$. The natural map $\sk_1 C_\bullet \to C_\bullet$ induces the map 
$$\Dec_0(\sk_1 C_\bullet) \to \Dec_0(C_\bullet) $$
such that $\Dec_0(\sk_1 C_\bullet)_0 =\Dec_0(C_\bullet)_0 = C_1$ and both $\Dec_0(\sk_1 C_\bullet)$ and $\Dec_0(C_\bullet)$ are deformation retracts of $C_0$. 

\begin{lmm}\label{lemma: zero decalage of 1-skeleta}
The reduced simplicial coalgebra $\Dec_0(\sk_1 C_\bullet)/C_1$ is isomorphic to the Kan suspension $\Sigma_\bullet (C_1/s_0C_0)$ of the quotient coalgebra $C_1/s_0 C_0$.
\end{lmm}

\begin{proof}
Straightforward computation.
\end{proof}

\begin{lmm}\label{lemma: decalage and skeleton}
The map $$f\colon \Dec_0(\sk_1 C_\bullet)/C_1 \to \Dec_0(C_\bullet)/C_1 $$ is a weak equivalence of reduced simplicial coalgebras. 
\end{lmm}

\begin{proof}
We denote by $\widetilde{C}_\bullet\subset C_\bullet$ the kernel of the counit $\e\colon C_\bullet \to \kk$. Then the underlying simplicial vector space of the quotient coalgebra $\Dec_0(C_\bullet)/C_1$ is a direct sum $(\Dec_0(\widetilde{C}_\bullet)/\widetilde{C}_1) \oplus \kk$. We compute the homotopy groups $\pi_*(\Dec_0(\widetilde{C}_\bullet)/\widetilde{C}_1)$ via the long exact sequence:
$$\ldots \to \pi_{*+1}(\Dec_0(\widetilde{C}_\bullet)/\widetilde{C}_1) \to \pi_*(\widetilde{C}_1) \to \pi_*(\Dec_0(\widetilde{C}_\bullet)) \to\pi_*(\Dec_0(\widetilde{C}_\bullet)/\widetilde{C}_1) \to \ldots$$
Since $\Dec_0(\widetilde{C}_\bullet)$ is a deformation retract of the constant simplicial vector space $\widetilde{C}_0$, we have $\pi_0(\Dec_0(\widetilde{C}_\bullet))=\widetilde{C}_0$, $\pi_0(\widetilde{C}_1)=\widetilde{C}_1$, and $\pi_*(\Dec_0(\widetilde{C}_\bullet))=\pi_*(\widetilde{C}_1)=0$ if $\ast>0$. Therefore, $$\pi_1(\Dec_0(\widetilde{C}_\bullet)/\widetilde{C}_1)=\ker{d_0}=\coker{s_0}=\widetilde{C}_1/s_0\widetilde{C}_0$$ and $\pi_i(\Dec_0(\widetilde{C}_\bullet)/\widetilde{C}_1)=0$ if $i\neq 1$.
This implies that $f$ is a weak equivalence.
\end{proof}

\begin{prop}\label{proposition: GZ of decalage}
The induced map $$GZ(f)\colon GZ(\Dec_0(\sk_1 C_\bullet)/C_1) \to GZ(\Dec_0(C_\bullet)/C_1)$$ 
is an isomorphism of Hopf algebras.
\end{prop}

\begin{proof}
By Lemma~\ref{lemma: GZ indecomposables} and Lemma~\ref{lemma: decalage and skeleton}, the induced map 
$$QGZ(f)\colon QGZ(\Dec_0(\sk_1 C_\bullet)/C_1) \to QGZ(\Dec_0(C_\bullet)/C_1) $$
of the modules of indecomposable elements is an isomorphism. Therefore $GZ(f)$ is surjective. We show that $GZ(f)$ is injective as well.

By Lemma~\ref{lemma: zero decalage of 1-skeleta} and Example~\ref{example: GZ of suspension}, the left hand side $GZ(\Dec_0(\sk_1 C_\bullet)/C_1)$ is the free Hopf algebra $H(C_1/s_0C_0)$ generated by the quotient coalgebra $C_1/s_0C_0$. Whereas, the right hand side $GZ(\Dec_0(C_\bullet)/C_1)$ is the quotient of the free Hopf algebra $H(C_2/s_0 C_1)$ subject to relations $$r(c)=d_1(c)-\sum_{i}d_0(c^{(1)}_i)d_2(c^{(2)}_i), $$
where $c\in C_3/s^2_0 C_1$, and $\Delta(c)=\sum_{i}c^{(1)}_i \otimes c^{(2)}_i$. Finally, the Hopf algebra homomorphism $GZ(f)$ maps a generator $c\in C_1/s_0 C_0$ to the generator $s_1(c)\in C_2/s_0 C_1$. We will construct a left inverse $$d\colon GZ(\Dec_0(C_\bullet)/C_1) \to GZ(\Dec_0(\sk_1 C_\bullet)/C_1)$$ to the map $GZ(f)$.

Note that a tensor algebra $T(V), V\in \Vect_\kk$ has a canonical anti-automorphism $T(V) \xrightarrow{\cong} T(V)^{op} $ given on monomials by $$v_1 v_2\cdot\ldots \cdot v_n \mapsto v_nv_{n-1}\cdot\ldots \cdot v_1.$$ Since as an algebra a free Hopf algebra $H(C), C\in \trcoalg$ is a tensor algebra $T(\oblv(C))$, there exists the similar anti-automorphism of Hopf algebras: 
\begin{equation}\label{equation: free hopf algebra op = id}
H(C) \xrightarrow{\cong} H(C)^{op}.
\end{equation}

Next, the face operators $d_0,d_1\colon C_2 \to C_1$ together with the map~\eqref{equation: free hopf algebra op = id} produce maps
$$d_0\colon H(C_2) \to H(C_1) \cong H(C_1)^{op}, $$
$$d_1\colon H(C_2) \to H(C_1) $$
of Hopf algebras. We define a map $d'\colon H(C_2) \to H(C_1)$ as the following composite:
\begin{align*}
d'\colon H(C_2) \xrightarrow{\Delta} H(C_2)\otimes H(C_2) &\xrightarrow{d_0 \otimes d_1} H(C_1)^{op} \otimes H(C_1) \\
&\xrightarrow{S\otimes \id} H(C_1)\otimes H(C_1) \\
&\xrightarrow{\nabla} H(C_1).
\end{align*}
Here $S\colon H(C_1)^{op} \to H(C_1)$ is the antipode map for the free Hopf algebra $H(C_2)$, and $\nabla\colon H(C_1)\otimes H(C_1) \to H(C_1)$ is the multiplication map.

A straightforward computation with simplicial relations shows that
$$d'(s_0(c)) = \e(c), \;\; c\in C_1 \;\; \text{and} \;\; d'(r(c)) = 0, \;\;c\in C_3. $$
Therefore the map $d'\colon H(C_2) \to H(C_1)$ factors through a unique map
$$d\colon H(C_2/s_0 C_1)/I \to H(C_1/s_0 C_0),  $$
where $I \subset H(C_2/s_0 C_1)$ is the two-sided ideal generated by relations $r(c)=0, c\in C_3$. We show that $d$ is a left inverse for the map $GZ(f)=s_1$, i.e $$ds_1(x)=x, \; x\in H(C_1/s_0 C_0).$$ We can assume that $x\in H(C_1/s_0 C_0)$ is a generator, i.e $x\in C_1/s_0 C_0$. Let $\Delta(x)=\sum_i x^{(1)}_i\otimes x^{(2)}_i$, then
\begin{align*}
ds_1(x)&=\sum_i S(d_0s_1(x^{(1)}_i))d_1s_1(x^{(2)}_i)\\
&=\sum_i S(s_0d_1(x^{(1)}_i))(x^{(2)}_i)\\
&=\sum_i \e(x^{(1)}_i)(x^{(2)}_i) = x. \qedhere
\end{align*}
\end{proof}

\begin{cor}\label{corollary: decalage is EM} Let $C_\bullet\in \strcoalg$, then the induced map
$$\Dec_0(C_\bullet)/C_1 \to NGZ(\Dec_0(C_\bullet)/C_1) $$
is a weak equivalence.
\end{cor}

\begin{proof}
In the commutative diagram
$$
\begin{tikzcd}
\Dec_0(\sk_1 C_\bullet)/C_1) \arrow{d} \arrow{r}
&NGZ(\Dec_0(\sk_1 C_\bullet)/C_1) \arrow{d} \\
\Dec_0(C_\bullet)/C_1 \arrow{r}
& NGZ(\Dec_0(C_\bullet)/C_1),
\end{tikzcd}
$$
the top arrow is a weak equivalence by Lemma~\ref{lemma: GZ of a suspension}, the left vertical arrow is a weak equivalence by Lemma~\ref{lemma: decalage and skeleton}, and the right vertical arrow is an isomorphism by Proposition~\ref{proposition: GZ of decalage}.
\end{proof}

Since the category $\trcoalg$ is complete, the total d\'{e}calage functor $\Dec$ has a right adjoint
\begin{equation}\label{equation: artin-mazur codiagonal}
T\colon \sstrcoalg \to \strcoalg, 
\end{equation}
which is called the \emph{Artin-Mazur codiagonal}, see~\cite{Stevenson12}.

\begin{thm}\label{theorem: artin-mazur codiagonal}
%The functor $\Dec$ admits the right adjoint
%$$T\colon \sstrcoalg \to \strcoalg.$$
The Artin-Mazur codiagonal functor $T$ satisfies following properties:
\begin{enumerate}
\item the unit map $C_\bullet \to T\Dec(C_\bullet)$ is a weak equivalence for any $C_\bullet\in \strcoalg$.
\item there is a natural weak equivalence $d\to T$, where $d\colon \sstrcoalg \to \strcoalg$ is the diagonal functor, i.e. $d(C_{\bullet,\bullet})_n=C_{n,n}$, $C_{\bullet,\bullet}\in \sstrcoalg$.
\item $T$ maps levelwise weak equivalences of bisimplicial coalgebras to weak equivalences of simplicial coalgebras.
\end{enumerate}
\end{thm}

\begin{proof}
By~\cite[Corollary~2.3]{OR20}, the unit map $C_\bullet \to T\Dec(C_\bullet)$ is a deformation retract, and so it is a weak equivalence.

The natural weak equivalence $d \to T $ of simplicial coalgebras was constructed in the proof of~\cite[Theorem~1.1]{Stevenson12}. Finally, $T$ preserves levelwise weak equivalence because $d$ does so.
\end{proof}

Let us denote by $\sprHopf$ the category of simplicial objects in the category $\prHopf$ of primitively generated Hopf algebras. Furthermore, let 
\begin{equation}\label{equation: horizontally reduced}
\sostrcoalg \subset \sstrcoalg
\end{equation}
be the full subcategory of bisimplicial coalgebras spanned by \emph{horizontally} reduced ones, i.e. $C_{\bullet,\bullet}\in \sstrcoalg$ lies in $\sostrcoalg$ if and only if $C_{0,m}=\kk$ for all $m\geq 0$. Finally, we extend the nerve functor~\eqref{equation: nerve functor} on simplicial Hopf algebras as follows
\begin{align*}
N\colon \sprHopf &\to \sostrcoalg, \\
N(H_\bullet)_{n, m} &= N_n(H_m).
\end{align*} 

\begin{dfn}\label{definition: twisted bar construction}
The \emph{twisted bar construction} $\overline{W}\colon \sprHopf \to \strcoalg$ is given by the following composite:
$$\overline{W} \colon \sprHopf \xrightarrow{N} \sostrcoalg \xhookrightarrow{\iota} \sstrcoalg \xrightarrow{T} \strcoalg. $$
\end{dfn}

Let $\sAlg^{aug}$ be the category of simplicial (augmented) associative algebras over a field $\kk$. Similar to Definition~\ref{definition: twisted bar construction}, we define the twisted bar construction for augmented associative algebras
\begin{equation}\label{equation: twisted bar associative algebras}
\overline{W}\colon \sAlg^{aug} \to \sVect_{\kk}.
\end{equation}
Note that the following diagram
$$
\begin{tikzcd}
	\sprHopf \arrow{d}{\oblv} \arrow{r}{\overline{W}}
	&\strcoalg \arrow{d}{\oblv} \\
	\sAlg^{aug} \arrow{r}{\overline{W}}
	& \sVect_{\kk}
\end{tikzcd}
$$
commutes.

\begin{rmk}\label{remark: barW is Eilenberg-MacLane}
By~\cite{Duskin75} or~\cite[Lemma~5.2]{Stevenson12}, the definition of the twisted bar construction $\barW$ given here coincides with the one given in~\cite[Section~2.3]{Priddy70long} (see also~\cite[p.12-05]{Moore54}).
\end{rmk}

\begin{prop}\label{proposition: twisted bar preserves weak equivalences}
The twisted bar construction $\barW\colon \sAlg^{aug} \to \sVect_{\kk}$ (resp. $\overline{W} \colon \sprHopf \to \strcoalg$) preserves weak equivalences. 
\end{prop}
\begin{proof}
Let $f\colon A_\bullet \to A'_\bullet$ be a weak equivalence of simplicial augmented associative algebras. Then $N(f)$ is a vertical weak equivalence of bisimplicial vector spaces, and so is $\barW(f)=TN(f)$ by Theorem~\ref{theorem: artin-mazur codiagonal} (see also~\cite[Theorem~1.1]{Stevenson12}.
\end{proof}

\iffalse
\begin{rmk}\label{remark: spectral sequence for barW}
Using part~(2) of Theorem~\ref{theorem: artin-mazur codiagonal} and the spectral sequence of a bisimplicial set one can get a strongly convergent spectral sequence
$$E^2_{s,t} = \pi_t(\Tor^{H_\bullet}_s(\kk,\kk)) \Rightarrow \pi_{s+t}(dN_\bullet H_\bullet) \cong \pi_{s+t}(\barW H_\bullet),$$
where $H_\bullet\in \sprHopf$, see~\cite{Quillen66} or~\cite[Section~IV.2.2]{GoerssJardine}.
\end{rmk}
\fi

\begin{dfn}\label{definition: indecomposables}
Let $Q\colon \sAlg^{aug} \to \sVect_{\kk}$ denote the functor of \emph{indecomposable elements}, i.e. $Q(A_\bullet)=I_\bullet/I^2_\bullet$, where $I_\bullet \subset A_\bullet$ is the augmentation ideal.
\end{dfn}

\begin{prop}\label{proposition: twisted bar and indecomposables}
There is a natural transformation
$$\eta_{A_\bullet}\colon \barW(A_\bullet)/\kk \to \Sigma_{\bullet} Q(A_\bullet),\;\; A_\bullet\in \sAlg^{\aug}, $$
where $\kk \subset \barW(A_\bullet)$ is a constant simplicial vector space spanned by the unit elements in $A_\bullet$ and $\Sigma_{\bullet} Q(A_\bullet) \in \sVect_{\kk}$ is the Kan suspension of the simplicial vector space $Q(A_\bullet)$. Moreover, $\eta_{A_\bullet}$ is a weak equivalence if $A_\bullet\in \sAlg^{\aug}$ is a degreewise free associative algebra.
\end{prop}
 
\begin{proof}
There is a natural map of vector spaces
$$A/\kk \cong I  \to I/I^2 =Q(A), \;\; A\in \sAlg^{aug},$$
where $I\subset A$ is the augmentation ideal. This map induces the following map of simplicial vector spaces
\begin{equation}\label{equation: indecomposables coincide, eq1}
\barW_\bullet(A_\bullet)/\kk \to \Sigma_\bullet Q(A_\bullet).
\end{equation}

We show now that the map~\eqref{equation: indecomposables coincide, eq1} is a weak equivalence if $A_t\cong T(V_t)$ is a tensor algebra, $V_t \in \Vect_\kk$, $t\geq 0$. By part~(2) of Theorem~\ref{theorem: artin-mazur codiagonal} and the spectral sequence of a bisimplicial set, we have a strongly convergent spectral sequence
$$E^1_{s,t} = \Tor^{A_t}_s(\kk,\kk) \Rightarrow \pi_{s+t}(dN_\bullet A_\bullet) \cong \pi_{s+t}(\barW A_\bullet),$$
where $A_\bullet\in \sAlg^{aug}$, see e.g.~\cite{Quillen66} or~\cite[Section~IV.2.2]{GoerssJardine}.
%We have the spectral sequence of Remark~\ref{remark: spectral sequence for barW}:
%$$E^2_{s,t}=\Tor^{U^r(L_t)}_s(\kk,\kk)\Rightarrow \pi_{s+t}(\barW U^r(L_\bullet)). $$
%We observe that $E^2_{s,t}=0$ for all $s>1$. Indeed, $U^r(L_t) =U^r\free(V_t)\cong T(V_t)$ is a free associative algebra, so 
%$$\Tor_s^{T(V_t)}(\kk,\kk)=V_t, $$
%if $s=1$, and vanishes otherwise. This implies the proposition. %Therefore, $$E^2_{s,t}=\pi_t(\Tor^{T(V_\bullet)}_s(\kk,\kk))=0$$
%if $s>1$. The proposition now follows.
Since $A_t\cong T(V_t)$ is a free associative algebra, we have $E^1_{s,t}=0$ for $s>1$ and $E^1_{1,t}\cong V_t$, see e.g.~\cite[Example~2.2(1)]{Priddy70} or~\cite[\S1.2, 1.5]{PP05}. This implies the proposition.
\end{proof}

Note that the fully faithful embedding $\iota\colon \sostrcoalg \hookrightarrow \sstrcoalg$ has a left adjoint $R\colon \sstrcoalg \to \sostrcoalg$ given by
\begin{equation}\label{equation: left adjoint to iota}
R(C_{\bullet,\bullet})_{n,m}=C_{n,m}/C_{0,m} 
\end{equation}
for all $C_{\bullet,\bullet} \in \sstrcoalg$. Therefore the twisted bar construction $\overline{W}$ has a left adjoint
\begin{equation}\label{equation: Kan loop group}
G\colon \strcoalg \to \sprHopf, 
\end{equation}
which is given by $G(C_\bullet)_m=GZ(\Dec_m(C_\bullet)/C_{m+1})$, $C_\bullet \in \strcoalg$.

\begin{dfn}\label{definition: Hopf-Kan loop algebra}
Let $C_\bullet$ be a simplicial truncated coalgebra, then $G(C_\bullet)\in \sprHopf$ is called the \emph{Hopf-Kan loop algebra} of $C_\bullet$.
\end{dfn}

\begin{thm}\label{theorem: kan loop}
Let $C_\bullet \in \sotrcoalg$ be a reduced simplicial truncated coalgebra. Then the unit map
$$\eta\colon C_\bullet \to \overline{W}G(C_\bullet) $$
is a weak equivalence.
\end{thm}

\begin{proof}
The units of the adjunctions $\Dec\dashv T$, $R\dashv \iota$, and $N\dashv GZ$ give a factorization of $\eta$   
$$
C_\bullet\to T\Dec C_\bullet\to T \iota R \Dec(C_\bullet)\to T \iota N GZ R \Dec(C_\bullet)  
$$
in $\strcoalg$.  The map $C_\bullet\to T \Dec C_\bullet$ is a weak equivalence by Theorem~\ref{theorem: artin-mazur codiagonal}. 
The maps $T \Dec(C_\bullet)\to T \iota R\Dec(C_\bullet)$ and 
$T \iota R\Dec(C_\bullet)\to T \iota N GZ R \Dec(C_\bullet)$ are induced by the maps 
$$ 
\Dec(C_\bullet)\to \iota R \Dec(C_\bullet)\;\; \text{and}\;\; R\Dec(C_\bullet)\to  N GZ R \Dec(C_\bullet) 
$$ 
in $\sstrcoalg$. We will show that both of these maps are levelwise weak equivalences.  

The first map is a vertical weak equivalence, since $$(R\Dec(C_\bullet))_{n,\bullet}=\Dec(C_\bullet)_{n,\bullet}/\Dec(C_\bullet)_{0,\bullet}$$ and $\Dec(C_\bullet)_{0,\bullet}$ is a deformation retract of $C_0\cong \kk$.

The second map is a horizontal weak equivalence due to Corollary~\ref{corollary: decalage is EM} and the observation~\eqref{equation: partial decalage, 1}. Since $T$ maps levelwise weak equivalences of bisimplicial coalgebras to weak equivalences, the theorem follows.
\end{proof}

Using Example~\ref{example: GZ of suspension} one can check the following.
\begin{prop}\label{proposition: kan loop group is almost free}
The Hopf-Kan loop functor $G\colon \strcoalg \to \sprHopf $ is given as follows. 
\begin{enumerate}
\item The Hopf algebra $G(C_\bullet)_m$ is the free Hopf algebra $H(C_{m+1}/s_0C_m)$ generated by the quotient coalgebra $C_{m+1}/s_0C_m$. 
\item The face and degeneracy operators are defined on the generators of $G(C_\bullet)$ by the following formulas
$$d_i[x] = [d_{i+1}x] \; \; \text{if} \;\; i>0, \;\; s_i[x] = [s_{i+1}x] \; \; \text{if} \;\; i\geq 0, $$
$$d_0[x] = \sum_i S([d_0(x^{(1)}_i)])[d_1(x^{(2)}_i)];$$
where $x\in C_{m+1}$, $\Delta(x)=\sum_i x^{(1)}_i\otimes x^{(2)}_i$, $[y]$ is the class of $y\in C_{n+1}$ in $C_{n+1}/s_0C_n$, and $$S\colon H(C_{m}/s_0C_{m-1})^{op} \to H(C_{m}/s_0C_{m-1})$$ is the antipode. \qed
\end{enumerate}
\end{prop}

\subsection{Model structures}\label{section: model structures} In this section we will construct model structures for categories of simplicial restricted Lie algebras $\srLie$ and reduced simplicial truncated coalgebras $\sotrcoalg$. We refer the reader to~\cite{Quillen67}, \cite{Hirschhorn03}, and~\cite[Appendices~A.2-3]{HTT} for most of definitions in this section.

\begin{dfn}[J.~H.~Smith]\label{definition: combinatorial} A model category $\mathsf{C}$ is \emph{combinatorial} if $\mathsf{C}$ is cofibrantly generated (\cite[Definition~11.1.2]{Hirschhorn03}) and $\mathsf{C}$ is locally presentable. 
\end{dfn}
We recall generating sets for simplicial model categories $\sSet$ and $\sVect_{\kk}$. Define the following sets of morphisms in $\sSet$ $$I_\Delta=\{\partial \Delta^n \hookrightarrow \Delta^n\;|\; n\geq 0 \}, \;\;\; J_{\Delta}=\{\Lambda^n_i \hookrightarrow \Delta^n\;|\; n>0,\; 0 \leq i\leq n\},$$ 
where $\Lambda^n_i\subset \Delta^n$ (resp. $\partial \Delta^n \subset \Delta^n$) is the $i$-th horn (resp. the boundary) of $\Delta^n$. Then the model category $\sSet$ is cofibrantly generated by $I_\Delta, J_\Delta$. Similarly, define $$I_\Vect=\{\kk(\partial \Delta^n) \hookrightarrow \kk(\Delta^n)\;|\; n\geq 0 \},$$ 
$$J_{\Vect}=\{\kk(\Lambda^n_i) \hookrightarrow \kk(\Delta^n)\;|\; n>0,\; 0 \leq i\leq n\},$$ 
where $\kk(X_\bullet), X_\bullet\in\sSet$ is the simplicial vector space spanned by $X_\bullet$. Then the model category $\sVect$ is cofibrantly generated by $I_\Vect, J_\Vect$.

\begin{rmk}\label{remark: surjective on components}
By~\cite[Proposition~1, pII.3.8]{Quillen67}, a map $f\colon U_\bullet \to W_\bullet \in \sVect_{\kk}$ of simplicial vector spaces is a \emph{fibration} in the model structure above if and only if 
\begin{enumerate}
\item the induced map $U_\bullet \to W_\bullet \times_{\pi_0(W_\bullet)} \pi_0(U_\bullet)$ is degreewise surjective, where $\pi_0(U_\bullet)$ and $\pi_0(W_\bullet)$ are constant simplicial vector spaces;
\item or equivalently, the induced map $N(f)\colon N(U_\bullet) \to N(W_\bullet)$ of normalized chain complexes (see Section~\ref{section: notation}) is surjective in positive degrees.
\end{enumerate}
\end{rmk}

\begin{thm}\label{theorem:modelsrlie} There exists a simplicial combinatorial right proper model structure on the category $\srLie$ such that a map $f\colon L'_\bullet \to L_\bullet$ is

\begin{itemize}
\item a \emph{weak equivalence} if and only if $\pi_*(f)$ is an isomorphism;
\item a \emph{fibration} if and only if $\oblv(f)\colon \oblv L'_\bullet \to \oblv L_\bullet$ is a fibration in $\sVect_{\kk}$ (see Remark~\ref{remark: surjective on components});
\item a \emph{cofibration} if and only if $f$ has the left lifting property with respect to all acyclic fibrations.
\end{itemize}
Moreover, $$I_\Lie=\free(I_\Vect)=\{\free(u)\;|\; u\in I_\Vect\}$$ is a set of generating cofibrations and $$J_\Lie=\free(J_\Vect)=\{\free(v)\;|\; v\in J_\Vect\}$$ is a set of generating trivial cofibrations for the model category $\srLie$.
\end{thm}

\begin{proof}
Since the category $\rLie$ is complete and cocomplete, the category $\srLie$ of simplicial objects has the canonical simplicial enrichment. By~\cite[Theorem~4, pII.4.1]{Quillen67}, the category $\srLie$ is a simplicial model category equipped with notions of weak equivalences, fibrations and cofibrations as defined above. Furthermore, by an implicit argument in ibid. the category $\srLie$ is cofibrantly generated; more explicitly, one can use~\cite[Theorem~11.3.2]{Hirschhorn03}. Indeed, the sets $I_{\Lie}, J_\Lie$ clearly permit the small object argument (\cite[Definition~10.5.15]{Hirschhorn03}) and the forgetful functor $\oblv$ takes relative $J_\Lie$-cell complexes to weak equivalences because any map in $J_\Vect$ is a homotopy equivalence.

Since any object in $\srLie$ is fibrant, the model category $\srLie$ is \emph{right proper} by~\cite[Corollary~13.1.3(2)]{Hirschhorn03}. Finally, the category $\srLie$ is locally presentable by Proposition~\ref{proposition:category property of rLie}.
\end{proof}

\begin{rmk}\label{remark: left proper}
At the time of writing, we are not aware whether or not the model structure of Theorem~\ref{theorem:modelsrlie} on $\srLie$ is \emph{left proper}, cf. \cite[Section~2.3]{Rezk02}.
\end{rmk}

We will discuss cofibrations in $\srLie$; we encourage the reader to compare the next definitions with the definition of an almost-free morphism in the category of simplicial commutative algebras, e.g. given in~\cite[p.~23]{Goerss90} or in~\cite[Definition~3.3']{MillerCorrection}.

\begin{dfn}\label{definition: almost-simplex category}
The \emph{almost-simplex category} $\widetilde{\Delta}$ is the category of finite ordered sets $[n]=\{0,\ldots, n\}, n\geq 0$ together with order-preserving maps which send $0$ to $0$ (cf. the first definition of Section~2 in~\cite{MillerCorrection}). 
\end{dfn}

Let us denote by $\tilde{\mathsf{s}}\mathsf{C}$ the category of \emph{almost-simplicial objects} in a category $\mathsf{C}$, i.e. $\tilde{\mathsf{s}}\mathsf{C}$ is  the category of contravariant functors from $\widetilde{\Delta}$ to $\mathsf{C}$. 

\begin{rmk}\label{remark: almost-simplicial vector spaces}
By an analog of the Dold-Kan correspondence, the category $\tilde{\mathsf{s}}\Vect_{\kk}$ of almost-simplicial vector spaces is equivalent to the category of \emph{graded} vector spaces, see~\cite[pp.607-608]{MillerCorrection}. Thus, a map $f\colon U_\bullet \to W_\bullet$ in $\sVect$ is a monomorphism if and only if there is an almost-simplicial vector subspace $V_\bullet$ of $W_\bullet$ such that the natural map $U_n\oplus V_n \to W_n$ is an isomorphism for each $n\geq 0$.
\end{rmk}

The last paragraph inspires the next definition.

\begin{dfn}\label{definition: almost-free morphism}
A morphism $f\colon L'_\bullet \to L_\bullet$ in $\srLie$ is called \emph{almost-free} if there is an almost-simplicial vector subspace $V_\bullet$ of $L_\bullet$ such that the natural map of almost-simplicial restricted Lie algebras $L'_\bullet\sqcup \free(V_\bullet) \to L_\bullet$ is an isomorphism in $\tilde{\mathsf{s}}\rLie$.
\end{dfn}

%We call a morphism $f\colon L_{\bullet} \to M_{\bullet}$ in $\srLie$ \emph{almost-free} if, for every $n\geq 0$, there is a sub-vector space $V_n \subset M_n$ and maps of vector spaces
%$$\delta_i \colon V_n \to V_{n-1}, \;\; 1 \leq i\leq n,$$
%$$\sigma_i \colon V_{n} \to V_{n+1}, \;\; 0 \leq i \leq n, $$
%such that the evident extension 
%$$L_n \sqcup \free(V_n) \to M_n $$
%is an isomorphism for every $n$ and there are commutative diagrams with the horizontal maps isomorphisms:
%$$
%\begin{tikzcd}
%L_n \sqcup \free(V_n) \arrow{d}{d_i \sqcup \free\delta_i} \arrow{r}{\cong}
%&M_n \arrow{d}{d_i} \\
%L_{n-1}\sqcup \free(V_{n-1}) \arrow{r}{\cong}
%& M_{n-1},
%\end{tikzcd}
%$$
%for $i\geq 1$, and
%$$
%\begin{tikzcd}
%L_n \sqcup \free(V_n) \arrow{d}{s_i \sqcup \free\sigma_i} \arrow{r}{\cong}
%&M_n \arrow{d}{s_i} \\
%L_{n+1}\sqcup \free(V_{n+1}) \arrow{r}{\cong}
%& M_{n+1},
%\end{tikzcd}
%$$
%for $i\geq 0$. We notice the reader that only the face operator $d_0$ is \emph{not} induced up from $\sVect_{\kk}$. Finally, 

Finally, we say that a simplicial restricted Lie algebra $L_{\bullet}\in \srLie$ is \emph{almost-free} if the morphism $0\to L_{\bullet}$ is almost-free. We notice that an almost-free simplicial restricted Lie algebra is ``free'' in the sense of~\cite[Section~3.2]{Priddy70long}, but not vice versa.

The following proposition can be proved exactly as the similar result in~\cite[Theorem~3.4]{Miller84} (see also the correction~\cite{MillerCorrection}).

\begin{prop}\label{proposition: almost-free are cofibrations}
Any almost-free morphism $f\colon L'_\bullet \to L_\bullet$ is a cofibration. \qed
\end{prop}

\begin{rmk}\label{remark:almost free} Similar to the case of simplicial commutative algebras, one can show that for any map $f\colon L'_\bullet \to L_\bullet$ there is a functorial factorization
$$f\colon L'_\bullet \xrightarrow{i} Q_\bullet(f) \xrightarrow{p} L_\bullet, $$
where $i$ is almost-free and $p$ is an acyclic fibration, cf.~\cite[Theorem~1.3 and Proposition~1.4]{Goerss90}. In particular, any cofibration in $\srLie$ is a retract of an almost-free morphism.
\end{rmk}

Next, we show that the category $\sotrcoalg$ also can be equipped with a model structure.

\begin{thm}\label{theorem:modelsotrcoalg} There exists a simplicial combinatorial left proper model structure on $\sotrcoalg$ such that $f\colon C_\bullet \to D_\bullet$ is

\begin{itemize}
\item a \emph{weak equivalence} if and only if $\pi_*(f)$ is an isomorphism;
\item a \emph{cofibration} if and only if $f\colon C_\bullet \to D_\bullet$ is degreewise injective;
\item a \emph{fibration} if and only if $f$ has the right lifting property with respect to all acyclic cofibrations.
\end{itemize}
\end{thm}

\begin{rmk}\label{remark: goerss, model structure on coalgebras}
P.~Goerss in~\cite[Section~3]{Goerss95} showed that the category $\scoalg$ of simplicial coalgebras over a field $\kk$ endowed with the same notions of weak equivalences and cofibrations is a simplicial model category. It seems likely that one can straightforwardly adapt his argument for the category $\sotrcoalg$ as well. Nevertheless, we will prove Theorem~\ref{theorem:modelsotrcoalg} by applying more general technique.
\end{rmk}

\begin{proof}
Since the category $\trcoalg$ of truncated coalgebras is complete and cocomplete, the category $\sotrcoalg$ of reduced simplicial objects has the canonical simplicial enrichment. Recall that finite coproducts in $\trcoalg$ are wedge sums (see Example~\ref{example: coalgebras product and coproduct}), therefore a coproduct of reduced simplicial truncated coalgebras (computed in $\strcoalg$) is still reduced. Moreover, for any $C_\bullet \in \sotrcoalg$, we have a factorization 
\begin{equation}\label{equation: model structure on coalg, eq1}
C_\bullet \sqcup C_\bullet = C_\bullet \times \partial \Delta^1 \xhookrightarrow{j} C_\bullet \times \Delta^1 \xrightarrow{p} C_\bullet 
\end{equation}
of the fold map $\nabla\colon C_\bullet \sqcup C_\bullet \to C_\bullet$ such that $j$ is a cofibration and $p$ is a weak equivalence.

We now apply Theorem~2.2.1 from~\cite{HKRS17} to the adjoint pair 
\begin{equation*}
\begin{tikzcd}
\oblv: \sotrcoalg \arrow[shift left=.6ex]{r}
&\soVect_{\kk} :\Sym^{tr} \arrow[shift left=.6ex,swap]{l}
\end{tikzcd}
\end{equation*}
in order to obtain the desired model structure on $\sotrcoalg$. Indeed, the model category $\soVect_{\kk}$ is cofibrantly generated; the category $\sotrcoalg$ is locally presentable by Proposition~\ref{proposition:category property of trcoalg}; $\oblv(C_\bullet), C_\bullet\in \sotrcoalg$ is a cofibrant object in $\soVect_{\kk}$ because any object in $\soVect_{\kk}$ is cofibrant; and finally, the factorization~\eqref{equation: model structure on coalg, eq1} fulfills the third condition of~\cite[Theorem~2.2.1]{HKRS17}. The obtained model structure on $\sotrcoalg$ is cofibrantly generated by sets due to~\cite[Theorem~2.2.3]{BHKKRS15}.

It is clear that the axiom SM7b from~\cite[pII.2.3]{Quillen67} holds for the constructed model structure on $\sotrcoalg$, and so $\sotrcoalg$ is a simplicial model category. Finally, $\sotrcoalg$ is left proper because any object in $\sotrcoalg$ is cofibrant.
\iffalse
P.~Goerss in~\cite[Section~3]{Goerss95} showed that the category $\scoalg$ of simplicial coalgebras over a field $\kk$ endowed with the same notions of weak equivalences and cofibrations is a simplicial model category. The proof of this statement for $\sotrcoalg$ goes almost identically, and we refer the reader to \emph{op.~cit.} for details. Although, we warn that in the proof of the counterpart of Lemma~3.5 in ibid. one has to use cofree truncated coalgebras $\Sym^{tr}(-)$ (see Proposition~\ref{proposition: cofree truncated coalgebra}) instead of an abstact cofree cocommutative coalgebras (called $S(-)$ in ibid.)

Furthermore, the counterpart of Lemma~3.7 in ibid. implies that the obtained model structure is cofibrantly generated. Indeed, let $\kappa$ be an infinite regular cardinal greater than the cardinality of $\kk$. Then the set of generating cofibrations $I$ consists of inclusions 
$$C_\bullet \hookrightarrow D_\bullet, C_\bullet,D_\bullet \in \sotrcoalg,$$
where the cardinality of a basis in $D_\bullet$ is at most $\kappa$. Similarly, the set of generating trivial cofibrations $J$ consists of inclusions as above which are weak equivalences at the same time. Both $I$ and $J$ are indeed sets if we take a one representative of an inclusion for each isomorphism class. 

Finally, the model category $\sotrcoalg$ is left proper because any object is cofibrant, and $\sotrcoalg$ is presentable by Proposition~\ref{proposition:category property of trcoalg}.
\fi
\end{proof}

\begin{rmk}\label{remark: sets of generating cofibrations in sotrcoalg}
Let $\kappa$ be an infinite regular cardinal greater than the cardinality of $\kk$, and let $I_{\CoAlg}$ be the set of isomorphism classes of inclusions 
$$C_\bullet \hookrightarrow D_\bullet, \; C_\bullet,D_\bullet \in \sotrcoalg$$
such that the cardinality of a basis in $D_\bullet$ is at most $\kappa$. Similarly, let $J_{\CoAlg}\subset I_{\CoAlg}$ be the set of isomorphism classes of inclusions as above which are weak equivalences. Then one can show that $I_{\CoAlg}$ (resp. $J_{\CoAlg}$) is a set of generating (resp. trivial) cofibrations for the model structure of Theorem~\ref{theorem:modelsotrcoalg} on the category $\sotrcoalg$. We are not aware if it is possible to choose more practical generating sets.
\end{rmk}

The next proposition was proven by S.~Priddy in~\cite[Proposition~2.8]{Priddy70long}. Here we repeat the argument for the reader's convenience.

\begin{prop}\label{proposition: Ur preserves weak equivalences}
The functor $U^r\colon \srLie \to \sprHopf$ preserves weak equivalences.
\end{prop}

\begin{proof}
%By Remark~\ref{remark: twisted bar preserves weak equivalences}, it suffices to show that $U^r\colon \srLie \to \sprHopf$ preserves weak equivalences. 
Let $f\colon L'_\bullet \to L_\bullet$ be a weak equivalence. Filter both $U^r(L'_\bullet)$ and $U^r(L_\bullet)$ by their Lie filtrations, see Section~\ref{section: primitively generated Hopf algebras}. Since $U^r(f)$ preserves Lie filtrations, it induces a map of associated spectral sequences:
\begin{equation*}
\begin{tikzcd}[column sep=large]
\pi_*E_0U^r(L'_\bullet) \arrow[r, "E_0U^r(f)"] \arrow[d, Rightarrow]
&\pi_*E_0U^r(L_\bullet) \arrow[d, Rightarrow] \\
\pi_*U^r(L'_\bullet) \arrow[r, "\pi_*U^r(f)"]
&\pi_* U^r(L_\bullet).
\end{tikzcd}
\end{equation*}

According to the Poincar\'{e}-Birkhoff-Witt theorem (Theorem~\ref{theorem: poincare-birkhoff-witt}), we have isomorphisms $E_0U^r(L'_\bullet)\cong \Sym^{tr}(L'_\bullet)$ and $E_0U^r(L_\bullet)\cong \Sym^{tr}(L_\bullet)$, and so the induced map $E_0U^r(f)$ is an isomorphism as well. The Lie filtrations on $U^r(L'_\bullet)$ and $U^r(L_\bullet)$ are complete and bounded below, and so both spectral sequences converge strongly. Therefore $E^{\infty}U^r(f)$ and $\pi_*U^r(f)$ are also isomorphisms.
\end{proof}

\begin{cor}\label{corollary: barW Ur preserves weak equivalences}
The functor $\barW U^r\colon \srLie \to \sotrcoalg$ preserves weak equivalences.
\end{cor}

\begin{proof}
See Propositions~\ref{proposition: Ur preserves weak equivalences} and~\ref{proposition: twisted bar preserves weak equivalences}.
\end{proof}

Next, following~\cite[Proposition~3.5]{Priddy70long}, we calculate the composite $$\oblv\circ \barW U^r\colon \srLie \to \sVect_{\kk}.$$
For that we consider a functor
$$\triv \colon \Vect_{\kk} \to \rLie $$
which maps a vector space $V$ to the restricted $p$-abelian Lie algebra $\triv(V)$ with the underlying vector spaces equal to $V$ equipped with identically zero Lie bracket and $p$-operation.

We observe that the functor $\triv$ has a left adjoint
\begin{equation}\label{equation: abxi}
\Abxi\colon \rLie \to \Vect_{\kk} 
\end{equation}
given by $L\mapsto L/([L,L]+\xi(L))$, where $[L,L]\subset L$ is the (restricted) Lie ideal generated by all elements of the form $[x,y]$, $x,y\in L$.

We extend the adjoint pair $\Abxi \dashv \triv$ degreewise to the adjunction
\begin{equation}\label{equation: quillen adjunction, abxi triv}
\begin{tikzcd}
\Abxi: \srLie \arrow[shift left=.6ex]{r}
&\sVect_{\kk} :\triv \arrow[shift left=.6ex,swap]{l}
\end{tikzcd}
\end{equation}
between categories of simplicial objects. Note that the adjunction~\eqref{equation: quillen adjunction, abxi triv} is a Quillen adjunction because the composite $\oblv\circ \triv = \id $, and so the functor $\triv$ preserves weak equivalences and fibrations. %see also~\cite[Propositions~3.3,~3.9]{Frank15}.

\begin{prop}\label{proposition: chain coalgebra and chain complex}
There is a natural transformation
$$\eta_{L_\bullet}\colon \oblv \circ \barW U^r(L_\bullet) \xrightarrow{} \Sigma_{\bullet} \Abxi(L_\bullet),\;\; L_\bullet\in \srLie. $$
Moreover, $\eta_{L_\bullet}$ is a weak equivalence if $L_\bullet\in \srLie$ is almost-free.
\end{prop}

Here $\Sigma_{\bullet} \Abxi(L_\bullet) \in \sVect_{\kk}$ is the Kan suspension of the simplicial vector space $\Abxi(L_\bullet)$, see~\cite[Section~III.5]{GoerssJardine}.

\begin{proof}
\iffalse
Let us denote by $\widetilde{U}^r(L) \subset U^r(L)$ the augmentation ideal in the universal enveloping algebra $U^r(L), L\in \rLie$. Note that the module of indecomposable elements $$QU^r(L)=\widetilde{U}^r(L)/\widetilde{U}^r(L)^2$$ is naturally isomorphic to $\Abxi(L)$. In particular, there is a natural map of vector spaces
$$\oblv \circ U^r(L)= \widetilde{U}^r(L) \to \widetilde{U}^r(L)/\widetilde{U}^r(L)^2 =\Abxi(L), $$
which induces the map of simplicial vector spaces
\begin{equation}\label{equation: chain coincide, eq1}
\oblv \circ\barW_\bullet U^r(L_\bullet) \to \Sigma_\bullet\Abxi(L_\bullet).
\end{equation}

We show now that the map~\eqref{equation: chain coincide, eq1} is a weak equivalence if $L_\bullet=\free(V_\bullet)$ is an almost-free simplicial restricted Lie algebra, $V_\bullet\in \tilde{s}\Vect_\kk$. We have the spectral sequence of Remark~\ref{remark: spectral sequence for barW}:
$$E^2_{s,t}=\Tor^{U^r(L_t)}_s(\kk,\kk)\Rightarrow \pi_{s+t}(\barW U^r(L_\bullet)). $$
%We observe that $E^2_{s,t}=0$ for all $s>1$. Indeed, $U^r(L_t) =U^r\free(V_t)\cong T(V_t)$ is a free associative algebra, so 
%$$\Tor_s^{T(V_t)}(\kk,\kk)=V_t, $$
%if $s=1$, and vanishes otherwise. This implies the proposition. %Therefore, $$E^2_{s,t}=\pi_t(\Tor^{T(V_\bullet)}_s(\kk,\kk))=0$$
%if $s>1$. The proposition now follows.
Since $U^r(L_t) =U^r\free(V_t)\cong T(V_t)$ is a free associative algebra, we observe $E^2_{s,t}=0$ for $s>1$ and $E^2_{1,t}\cong V_t$, see e.g.~\cite[Example~2.2(1)]{Priddy70} or~\cite[\S1.2, 1.5]{PP05}. This implies the proposition.
\fi
Note that there is a natural isomorphism
$$QU^r(L) \cong \Abxi(L), \;\; L\in\rLie, $$
where $U^r(L)\in \prHopf$ is the universal enveloping algebra and $QU^r(L)$ is the module of indecomposable elements, see Definition~\ref{definition: indecomposables}. Finally, Proposition~\ref{proposition: twisted bar and indecomposables} implies the assertion.
\end{proof}

\begin{exmp}\label{example: PG of triv}
Let $V_\bullet \in \sVect_\kk$ be a simplicial vector space, let $\Sigma_\bullet V$ be the Kan suspension of $V_\bullet$ (see~\cite[Section~III.5]{GoerssJardine}), and let $\triv(\Sigma_\bullet V)$ be a trivial simplicial coalgebra (Example~\ref{example: trivial coalgebra}). Using Proposition~\ref{proposition: kan loop group is almost free} one can show that $PG(\triv \Sigma_{\bullet} V) \cong \free(V_\bullet)$. Moreover, 
the adjoint map
\begin{equation}\label{equation: triv to free}
\triv(\Sigma_{\bullet} V) \to \barW U^r(\free(V_\bullet))
\end{equation}
is a weak equivalence by Theorem~\ref{theorem: kan loop}.
\end{exmp}

\begin{dfn}\label{definition: barW-equivalence}
A map $f\colon L'_\bullet \to L_\bullet$ in $\srLie$ is an \emph{$\kk$-equivalence} if and only if $\barW U^r(f)$ is a weak equivalence in $\sotrcoalg$.
\end{dfn}

\begin{exmp}\label{example: barW-equivalences}
By Corollary~\ref{corollary: barW Ur preserves weak equivalences}, any weak equivalence in $\srLie$ is an $\kk$-equivalence. Furthermore, by Theorem~\ref{theorem: kan loop}, if $f\colon C_\bullet \to D_\bullet$ is a weak equivalence in $\sotrcoalg$, then $PG(f)$ is an $\kk$-equivalence.
\end{exmp}

\begin{lmm}\label{lemma: barW-equivalences of algebras closed under pushouts}
Let
$$
\begin{tikzcd}
X_\bullet \arrow{d}{g} \arrow{r}{f}
&X'_\bullet \arrow{d}{g'} \\
Y_\bullet \arrow{r}
& Y'_\bullet
\end{tikzcd}
$$
be a pushout square in the category $\sAlg^{aug}$ of simplicial augmented associative algebras over the field $\kk$. Suppose that $\barW(g)$ is a weak equivalence of simplicial vector spaces and $f$ is an almost-free morphism. Then $\barW(g')$ is also a weak equivalence.
\end{lmm}

\begin{proof}
Since the model structure on the category $\sAlg^{aug}$ is left proper (see~\cite[Section~2.3 and Example~2.7]{Rezk02}) and the functor $\barW$ preserves weak equivalences (see Proposition~\ref{proposition: twisted bar preserves weak equivalences}), we can assume that $X_\bullet, Y_\bullet$ are almost-free objects in $\sAlg^{aug}$ and $g\colon X_\bullet \to Y_\bullet$ is an almost-free morphism. Then both $X'_\bullet, Y'_\bullet$ are almost-free, and so by Proposition~\ref{proposition: twisted bar and indecomposables}, it suffices to show that 
$$\Sigma_{\bullet}Q(g')\colon \Sigma_{\bullet}Q(X'_\bullet) \to \Sigma_{\bullet}Q(Y'_\bullet) $$
is a weak equivalence in $\sVect_{\kk}$. However, by the assumption, the morphism $Q(g)$ is a weak equivalence and the functor $Q\colon \sAlg^{aug} \to \sVect_{\kk}$ preserves cofibrations and colimits. Since the model category $\sVect_{\kk}$ is left proper, the lemma follows.
\end{proof}

\begin{thm}\label{theorem:model structure srlie, barW-equivalence} There exists a simplicial combinatorial left proper model structure on the category $\srLie$ such that a map $f\colon L'_\bullet \to L_\bullet$ is

\begin{itemize}
\item a \emph{weak equivalence} if and only if $f$ is an $\kk$-equivalence;
\item a \emph{cofibration} if and only if $f$ is a cofibration in the model structure of Theorem~\ref{theorem:modelsrlie}.
\item a \emph{fibration} if and only if $f$ has the right lifting property with respect to all $\kk$-acyclic cofibrations.
\end{itemize}
\end{thm}

\begin{proof}

We construct the desired model structure by using Proposition~A.2.6.15 from~\cite{HTT}. Let $I_\Lie$ be the generating set of cofibrations from Theorem~\ref{theorem:modelsrlie} and let $\mathcal{W}_{\kk}$ be the class of morphisms in $\srLie$ spanned by $\kk$-equivalences. It is enough to prove that
\begin{enumerate}
\item the class $\mathcal{W}_\kk$ is perfect, see~\cite[Definition~A.2.6.12]{HTT}. We recall that a class of morphisms $\mathcal{W}$ is called \emph{perfect} if all isomorphisms are in $\mathcal{W}$, $\mathcal{W}$ has the two-out-of-three property, $\mathcal{W}$ is closed under filtered colimits, and there is a (small) subset $\mathcal{W}_0 \subset \mathcal{W}$ which generates $\mathcal{W}$ by filtered colimits.
\item a pushout of $g\in \mathcal{W}_{\kk}$ along an almost-free morphisms is again in $\mathcal{W}_{\kk}$;
\item if $g\colon L'_\bullet \to L_\bullet$ is a morphism in $\srLie$ which has the right lifting property with respect to all morphisms in $I_\Lie$, then $g$ is in $\mathcal{W}_{\kk}$.
\end{enumerate}

By the definition, the class $\mathcal{W}_{\kk}$ is the preimage $(\barW U^r)^{-1}(\mathcal{W}_{\CoAlg})$, where $\mathcal{W}_{\CoAlg}$ is the class of weak equivalences in $\sotrcoalg$. We claim that the class $\mathcal{W}_{\CoAlg}$ is perfect. Indeed, any isomorphism is in $\mathcal{W}_{\CoAlg}$; $\mathcal{W}_{\CoAlg}$ has the two-out-of-three-property; $\mathcal{W}_{\CoAlg}$ is closed under filtered colimits; and finally, $\mathcal{W}_{\CoAlg}$ is an accessible subcategory in the arrow category $\mathrm{Ar}(\sotrcoalg)$ by Theorem~\ref{theorem:modelsotrcoalg} and~\cite[Corollary~A.2.6.8]{HTT}. Since the functor $\barW U^r$ preserves filtered colimits, the class $\mathcal{W}_{\kk}$ is also perfect by~\cite[Corollary~A.2.6.14]{HTT}.

For the second part, let 
$$
\begin{tikzcd}
X_\bullet \arrow{d}{g} \arrow{r}{f}
&X'_\bullet \arrow{d}{g'} \\
Y_\bullet \arrow{r}
& Y'_\bullet
\end{tikzcd}
$$
be a pushout square in $\srLie$ such that $g$ is an $\kk$-equivalence and $f$ is almost-free. We will show that $g'$ is again an $\kk$-equivalence. Since the functor $U^r\colon \rLie \to \Alg$ is left adjoint, the diagram
$$
\begin{tikzcd}
U^r(X_\bullet) \arrow{d}{U^r(g)} \arrow{r}{U^r(f)}
&U^r(X'_\bullet) \arrow{d}{U^r(g')} \\
U^r(Y_\bullet) \arrow{r}
& U^r(Y'_\bullet)
\end{tikzcd}
$$
is a pushout square in $\sAlg$, and so is in $\sAlg^{aug}$. Moreover, $U^r(f)$ is an almost-free morphism of augmented associative algebras and $\barW(U^r(g))$ is a weak equivalence of simplicial vector spaces. The claim follows now by Lemma~\ref{lemma: barW-equivalences of algebras closed under pushouts}.

Finally, if $g\colon L'_\bullet \to L_\bullet$ is a morphism in $\srLie$ which has the right lifting property with respect to all morphisms in $I_\Lie$, then $g$ is a weak equivalence by Theorem~\ref{theorem:modelsrlie}. By Corollary~\ref{corollary: barW Ur preserves weak equivalences}, any weak equivalence is in $\mathcal{W}_{\kk}$.
\end{proof}

%Let us denote by $\mathsf{sLie}^{r}_{\xi}$ the category $\srLie$ equipped with the model structure of Theorem~\ref{theorem:model structure srlie, barW-equivalence}. The choice of notation is explained by the following observation. We will see in Section~\ref{section: F-complete} that a connected simplicial restricted Lie algebra $L_\bullet$ is a fibrant object in $\mathsf{sLie}^r_{\xi}$ if and only if homotopy groups $\pi_i(L_\bullet),i>0$ are derived $\xi$-adic complete modules.

Using the notation from the proof of Theorem~\ref{theorem:model structure srlie, barW-equivalence} we write $\mathcal{W}_{\CoAlg}$ for the class of weak equivalences in the category $\sotrcoalg$ and we write $\mathcal{W}_{\kk}$ for the class of $\kk$-equivalences in $\srLie$. Moreover, we use $\mathcal{W}_{\Lie}$ for the class of usual weak equivalences in $\srLie$ from Theorem~\ref{theorem:modelsrlie}.

\begin{dfn}\label{definition: infinite-categories}
We denote by $\sCA_0=\sotrcoalg[\mathcal{W}^{-1}_{\CoAlg}]$ the $\infty$-category obtained from the (ordinary) category $\sotrcoalg$ by inverting the class $\mathcal{W}_{\CoAlg}$, see \cite[Definition~1.3.4.1]{HigherAlgebra}. Similarly, $\sL=\srLie[\mathcal{W}^{-1}_{\Lie}]$ is the $\infty$-category obtained from $\srLie$ by inverting $\mathcal{W}_{\Lie}$, and $\sLxi=\srLie[\mathcal{W}^{-1}_{\kk}]$ is the $\infty$-category obtained from $\srLie$ by inverting $\mathcal{W}_{\kk}$.
\end{dfn}

\begin{thm}\label{theorem: coalgebras and lie algebras}
The adjoint pair $PG \dashv \barW U^r$ induces an equivalence of $\infty$-categories
\begin{equation}\label{equation: equivalence of infty-categories}
%\begin{tikzcd}
%PG: \sotrcoalg \arrow[shift left=.6ex]{r}
%&\mathsf{sLie}^r_{\xi} :\barW U^r. \arrow[shift left=.6ex,swap]{l}
%\end{tikzcd}
PG: \sCA_0 \simeq \sLxi : \barW U^r.
\end{equation}
\end{thm}

\begin{proof} By Definition~\ref{definition: barW-equivalence}, we have $\barW U^r(\mathcal{W}_{\kk})= \mathcal{W}_{\CoAlg}$, and by Theorem~\ref{theorem: kan loop}, we have $PG(\mathcal{W}_{\CoAlg})\subset \mathcal{W}_{\kk}$. Therefore $PG$ and $\barW U^r$ induce functors 
$$PG\colon \sCA_0 \to \sLxi,\;\; \barW U^r\colon \sLxi \to \sCA_0  $$
between localized $\infty$-categories. Similarly, natural transformations
$$\id_{\sotrcoalg} \to \barW U^r \circ PG,  \;\; PG\circ \barW U^r \to \id_{\srLie}$$
induces natural transformations
$$\id_{\sCA_0}\to \barW U^r \circ PG, \;\; PG\circ \barW U^r \to \id_{\sLxi} $$
for functors between obtained $\infty$-categories. By Theorem~\ref{theorem: kan loop}, these natural transformations are natural equivalences. This implies the theorem.
\end{proof}

\begin{rmk}\label{remark: not quillen adjunction}
We are not aware whether or not the adjoint pair
\begin{equation*}
\begin{tikzcd}
PG: \sotrcoalg \arrow[shift left=.6ex]{r}
&\srLie :\barW U^r. \arrow[shift left=.6ex,swap]{l}
\end{tikzcd}
\end{equation*}
is a Quillen adjunction for any of two model structures on $\srLie$.
\end{rmk}

Let $\mathsf{C}$ be a simplicial model category. Then the (simplicial) nerve $N(\mathsf{C}^{\circ})$ forms an $\infty$-category, where $\mathsf{C}^{\circ}\subset \mathsf{C}$ is the full simplicial subcategory of fibrant-cofibrant objects. The $\infty$-category $N(\mathsf{C}^{\circ})$ is called the \emph{underlying $\infty$-category} of $\mathsf{C}$, see~\cite[Section~A.2]{HTT}.

\begin{prop}\label{proposition: coalgebras and lie, infty-categorical, part1}
%$\quad$
%\begin{enumerate}
%\item The $\infty$-categories $\sL$ and $\sCA_0$ are complete and cocomplete.
%\item 
The $\infty$-categories $\sL$, $\sLxi$, and $\sCA_0$ are presentable. In particular, $\sL$, $\sLxi$, and $\sCA_0$ are complete and cocomplete.
%\item The adjoint pair of derived functors
%$$
%\begin{tikzcd}
%\free: D_{\geq 0}(\Vect_\kk) \arrow[shift left=.6ex]{r}
%&\sL :\oblv \arrow[shift left=.6ex,swap]{l}
%\end{tikzcd}
%$$
%is monadic, and the functor $\oblv$ creates (homotopy) limits and sifted (homotopy) colimits.
%\item The adjoint pair of derived functors
%$$
%\begin{tikzcd}
%\oblv: \sCA_0 \arrow[shift left=.6ex]{r}
%&D_{>0}(\Vect_\kk) :\Sym^{tr} \arrow[shift left=.6ex,swap]{l}
%\end{tikzcd}
%$$
%is comonadic, and the functor $\oblv$ creates (homotopy) colimits.
%\end{enumerate}
\end{prop}

\begin{proof}
Let us denote by $\mathsf{sLie}^{r}_{\xi}$ the category $\srLie$ equipped with the model structure of Theorem~\ref{theorem:model structure srlie, barW-equivalence}. By~\cite[Theorem~1.3.4.20]{HigherAlgebra}, there are equivalences of $\infty$-categories:
$$\sL=\srLie[\mathcal{W}^{-1}_{\Lie}]\simeq N((\srLie)^o), \;\;\;\sLxi=\srLie[\mathcal{W}^{-1}_{\kk}]\simeq N((\mathsf{sLie}^{r}_{\xi})^o),$$ 
and 
$$\sCA_0=\sotrcoalg[\mathcal{W}^{-1}_{\CoAlg}] \simeq N((\sotrcoalg)^o). $$
By \cite[Proposition~A.3.7.6]{HTT} and Theorems~\ref{theorem:modelsrlie},~\ref{theorem:modelsotrcoalg}, \ref{theorem:model structure srlie, barW-equivalence}, we get the proposition.
\end{proof}

Recall that $\mathsf{sLie}^{r}_{\xi}$ is the model category from Theorem~\ref{theorem:model structure srlie, barW-equivalence}. Note that the identity functor $\id_{\srLie}$ produces a Bousfield localization
\begin{equation}\label{equation: quillen adjunction, bousfield}
\begin{tikzcd}
\srLie \arrow[shift left=.6ex]{r}
&\mathsf{sLie}^{r}_{\xi}, \arrow[shift left=.6ex,swap]{l}
\end{tikzcd}
\end{equation}
which by~\cite[Proposition~5.2.4.6]{HTT} induces the adjoint pair
\begin{equation}\label{equation: infty-adjunction, bousfield}
\begin{tikzcd}
L_\xi:\sL \arrow[shift left=.6ex]{r}
&\sLxi : \iota_\xi \arrow[shift left=.6ex,swap]{l}
\end{tikzcd}
\end{equation}
between underlying $\infty$-categories.

\begin{prop}\label{proposition: F-complete is an accessible localization}
The functor $\iota_{\xi}$ is fully faithful. In particular, the full subcategory $\sLxi \subset \sL$ is a localization.
\end{prop}

\begin{proof}
See~\cite[Appendix~A.3.7]{HTT}.
\end{proof}

\section{Homotopy theory of simplicial restricted Lie algebras}\label{section:homotopy theory}
This section is a technical heart of the paper. In our proof of Theorem~\ref{theorem: intro, D}, we mimic the classical proof (given e.g in~\cite[Theorem~11.1.1]{MayPonto}) that a simply-connected space $X$ is $p$-complete if and only if its homotopy groups $\pi_n(X),n\geq 2$ are derived $p$-complete. In order to transfer this proof to our context, we need to show that the category $\srLie$ shares a lot of common properties and features with the category of (pointed and connected) spaces; in this section we carefully check all required properties. We state our results in a form minimal enough for the proof of Theorem~\ref{theorem: intro, D}, although many of them can be easily generalized.

In Section~\ref{section:homotopy excision theorem} we prove that the homotopy excision theorem (Theorem~\ref{theorem: homotopy excision}) holds in the category $\srLie$ of simplicial restricted Lie algebras. 

Inspired by Proposition~\ref{proposition: chain coalgebra and chain complex}, we define \emph{homology} $H_*(L_\bullet;M)$ (Definition~\ref{definition: chain complex}) and \emph{cohomology groups} $H^*(L_\bullet;M)$ (Definition~\ref{definition: cochain complex}) of $L_\bullet \in \srLie$ with coefficients in any module $M$ over the ring $\kk\{\xi\}$. We reformulate Proposition~\ref{proposition: chain coalgebra and chain complex} as follows: $\pi_*(\barW U^r (L_\bullet))$ is the homology groups $H_*(L_\bullet;\kk)$ with coefficients in the trivial module $\kk$. Following~\cite{Ellis93}, we prove the Hurewicz theorem (Theorem~\ref{theorem: hurewicz theorem}), which combined with the homotopy excision theorem gives the relative Hurewicz theorem (Corollary~\ref{corollary: relative hurewicz theorem}). Finally, in Proposition~\ref{proposition: homology of abelian Lie algebra}, we compute the homology groups $H_*(\trivxi(M);\kk)$ of an abelian restricted Lie algebra $\trivxi(M)$ provided $M$ is a torsion-free $\kk\{\xi\}$-module.

In Section~\ref{section: postnikov system} we show that any $L_\bullet\in \srLie$ has a \emph{Postnikov tower} and we prove in Proposition~\ref{proposition: fibration with EM as fiber} that this Postnikov tower has \emph{$k$-invariants} provided $\pi_0(L_\bullet)=0$. In Section~\ref{section: principal fibration} we study group objects and \emph{principal fibrations} in $\srLie$. We define principal fibrations in Definition~\ref{definition: lie principal fibration} and we show that they are a homotopy invariant of the base in Lemma~\ref{lemma: principal fibrations are homotopy invariant}. Moreover, in Theorem~\ref{theorem: principal fibrations, classifying}, we construct a \emph{classifying object} $BM_\bullet$ (Definition~\ref{definition: classifying space for principal fibrations}) for principal $M_\bullet$-fibrations. Finally, we observe in Corollary~\ref{corollary: fibration with KM is principal} that each stage in the Postnikov tower of $L_\bullet$, $\pi_0(L_\bullet)=0$ is weakly equivalent to a principal fibration.

In Section~\ref{section: serre spectral sequence} we adapt classical approaches of~\cite[Chapter~9.4]{Spanier94} and~\cite[Chapter~15]{Switzer} to obtain an analog of the \emph{Serre spectral sequence} for principal fibrations in the category $\srLie$ (Theorem~\ref{theorem: SSS}). At the time of writing, we are not aware how to generalize our analog of the Serre spectral sequence to arbitrary fibrations, see Remark~\ref{remark: any fibration}.

\subsection{Homotopy excision theorem}\label{section:homotopy excision theorem} %In this section, we will see that a weak version of the Blakers-Massey theorem still holds in the category of simplicial restricted Lie algebras $\srLie$.
We fix some notation. 
\begin{dfn}\label{definition: connectivity, objects}
A simplicial restricted Lie algebra $L_\bullet\in \srLie$ is \emph{$n$-connected} if $\pi_i(L_\bullet)=0$ for all $i\leq n$. A simplicial restricted Lie algebra $L_\bullet$ is \emph{connected} if it is $0$-connected, i.e. $\pi_0(L_\bullet)=0$. 
\end{dfn}

\begin{dfn}\label{definition: connectivity, morphisms}
A morphism $f\colon L'_\bullet \to L_\bullet$ in $\srLie$ is \emph{$n$-connected} if the induced map on homotopy groups 
$$\pi_{i}(f)\colon \pi_i(L'_\bullet) \to \pi_i(L_\bullet)$$
is an isomorphism for $i< n$ and a surjection for $i=n$. 
\end{dfn}

Let $\fib(f)\in \srLie$ denote the \emph{homotopy fiber} of a morphism $f\colon L'_\bullet \to L_\bullet$. Then $f$ is $n$-connected if and only if $\fib(f)$ is $(n-1)$-connected and $\pi_0(f)$ is a surjection. Similarly, we write $\cofib(f)\in \srLie$ for the \emph{homotopy cofiber} of the morphism~$f$.

\begin{thm}[Homotopy excision theorem]\label{theorem: homotopy excision}
Let $f\colon L'_\bullet \to L_\bullet$ be a $n$-connected morphism in $\srLie$, $L'_\bullet$ is connected, and $n\geq 0$. Then the natural map (in the homotopy category $\mathrm{Ho}(\srLie)$)
\begin{equation}\label{equation: fiber to cofiber}
\fib(f) \to \Omega \cofib(f)
\end{equation}
is $(n+1)$-connected.
\end{thm}

\begin{rmk}\label{remark: sl is not a topos}
Since the underlying $\infty$-category $\sL$ of $\srLie$ is not an $\infty$-topos, we can not apply (at least directly) the generalized Blakers-Massey theorem, see e.g.~\cite{ABFJ20}. Instead, we will prove the homotopy excision theorem using model-theoretic approach, and only in the particular case as above.
\end{rmk}

\begin{dfn}\label{definition: reduced}
A map $f\colon L'_\bullet \to L_\bullet$ in $\srLie$ is \emph{$n$-reduced} if the maps $f_i\colon L'_i \to L_i$  are isomorphisms for $i\leq n$. 
\end{dfn}
Note that an almost-free $n$-reduced map in $\srLie$ is $n$-connected. The opposite is true up to weak equivalences and the proof is standard.

\begin{lmm}\label{lemma: reduced replacement}
Let $f\colon L'_\bullet\to L_\bullet$ be a $n$-connected morphism in $\srLie$, $n\geq 0$. Then there exists a commutative diagram
$$
\begin{tikzcd}
\widetilde{L'}_\bullet \arrow{d} \arrow{r}{\tilde{f}}
&\widetilde{L}_\bullet \arrow{d} \\
L'_\bullet \arrow{r}{f}
& L_\bullet
\end{tikzcd}
$$
such that the vertical arrows are weak equivalences in $\srLie$ and the map $\tilde{f}$ is almost-free and $n$-reduced. \qed %i.e.$\tilde{f}_i\colon \widetilde{L}_i \to \widetilde{M}_i$ is an isomorphism for $i\leq n$. \qed
\end{lmm}

Let $f \colon L'_\bullet \to L_\bullet$ be an almost-free $n$-reduced map. Notice that $f$ is an inclusion of the underlying simplicial vector spaces. Let us denote by $L_\bullet//L'_\bullet$ the quotient simplicial vector space
$$L_\bullet//L'_\bullet = \oblv(L_\bullet)/\oblv(L'_\bullet).$$
Similarly, we denote by $L_\bullet / L'_\bullet$ the underlying simplicial vector space of the quotient Lie algebra, i.e. $L_\bullet / L'_\bullet$ is the underlying simplicial vector of the coequalizer
$$
\begin{tikzcd}%[column sep=large]
\mathrm{coeq}( L'_\bullet \arrow[shift left=.75ex]{r}{f}
  \arrow[shift right=.75ex,swap]{r}{0}
&
L_\bullet).
\end{tikzcd}
$$
There are equivalences 
$$\oblv(\fib(f))\simeq \Sigma^{-1} L_\bullet // L'_\bullet, \;\;\; \oblv(\cofib(f)) \simeq L_\bullet/L'_\bullet, $$
and the map~\eqref{equation: fiber to cofiber} is equivalent to the desuspension of the canonical map
\begin{equation}\label{equation: quotient to quotient}
L_\bullet // L'_\bullet \to L_\bullet / L'_\bullet
\end{equation}
in $\sVect_\kk$. Thus, Theorem~\ref{theorem: homotopy excision} is equivalent to the following statement
\begin{itemize}
\item[$(\ast)$]\label{homotopy excision, reduced} Let $f\colon L'_\bullet \to L_\bullet$ be an almost-free $n$-reduced map in $\srLie$, $L'_0=0$, and $n\geq 0$, then the map~\eqref{equation: quotient to quotient} is $(n+2)$-connected.
\end{itemize}

We first show $(\ast)$ for free maps.
\begin{lmm}\label{lemma: homotopy excision, free}
Let $i\colon U_\bullet \to W_\bullet$ be $n$-reduced inclusion of simplicial vector spaces, and let $f=\free(i)\colon \free(U_\bullet) \to \free(W_\bullet)$ be the induced map between free simplicial restricted Lie algebras. Then $(\ast)$ holds for $f$ provided $U_0=0$ and $n\geq 0$.
\end{lmm}

\begin{proof}
Set $V_\bullet = W_\bullet/U_\bullet$ the quotient simplicial vector space, $V_i=0$ if $i\leq n$. By Proposition~\ref{proposition: free lie algebras, fresse}, there is a splitting
$$\oblv\circ\free(X_\bullet) \cong \bigoplus_{m\geq 1} L^r_m(X_\bullet) =\bigoplus_{m\geq 1}(\bfLie_m \otimes X_\bullet^{\otimes m})^{\Sigma_m}, $$
where $X_\bullet \in \sVect_\kk$. Therefore the map~\eqref{equation: quotient to quotient} is the direct sum of surjective maps
\begin{equation}\label{equation: surjective maps eq1}
p_m\colon L^r_m(W_\bullet)/L^r_m(U_\bullet) \to L^r_m(V_\bullet), \;\; m\geq 1, 
\end{equation}
and it suffices to show that each map $p_m$ is $(n+2)$-connected. 

The simplicial vector space $L^r_m(W_\bullet)$ has an increasing filtration
$$0=F_{-1}L^r_m(W_\bullet)\subset F_0L^r_m(W_\bullet) \subset F_1L^r_m(W_\bullet) \subset \ldots \subset F_mL^r_m(W_\bullet)=L^r_m(W_\bullet) $$
such that the quotient vector space $F_j/F_{j-1} \in \sVect_\kk$ is isomorphic to
$$F_jL^r_m(W_\bullet)/F_{j-1}L^r_m(W_\bullet)\cong(\bfLie_m\otimes U_\bullet^{\otimes (m-j)}\otimes V_\bullet^{\otimes j})^{\Sigma_{m-j}\times \Sigma_j}, $$
where $0\leq j\leq m$. Since $U_0=0$ and $V_i=0$ for $i\leq n$, we obtain that
$$\pi_i(F_jL^r_m(W_\bullet)/F_{j-1}L^r_m(W_\bullet))=0 $$
for $i\leq n+1$ and $0<j<m$. Therefore, 
$$\pi_i(F_{m-1}L^r_m(W_\bullet)/F_{0}L^r_m(W_\bullet))=0 $$
if $i\leq n+1$, and so each map $p_m$ is $(n+2)$-connected, $m\geq 1$.
\end{proof}

Next, we will resolve any almost-free $n$-reduced map in $\srLie$ by free maps. Let $L\in \rLie$ be a restricted Lie algebra. Recall that the \emph{bar construction} $B_\bullet(L)$ of $L$ is the simplicial restricted Lie algebra defined as follows
$$B_s(L)=(\free\circ\oblv)^{\circ(s+1)}(L), \;\; s\geq 0, $$
where the face operators are induced by the counit map $\free\circ \oblv \to \id$, and the degeneracy operators are induced by the unit map $\id \to \oblv \circ \free$. Notice that there is a canonical map
$$B_\bullet(L) \to L \in \srLie$$ 
to the constant simplicial restricted Lie algebra $L$, and this map is a weak equivalence. We are now ready to prove Theorem~\ref{theorem: homotopy excision}.

\begin{proof}[Proof of Theorem~\ref{theorem: homotopy excision}]
Let $f\colon L'_\bullet \to L_\bullet$ be an almost-free $n$-reduced map in $\srLie$, $L'_0=0$, $n\geq 0$. Consider the induced map of bisimplicial objects
$$B_\bullet(f)\colon B_\bullet(L'_\bullet) \to B_\bullet(L_\bullet). $$
Note that each $B_s(f), s\geq 0$ is the map of free simplicial restricted Lie algebras induced by the $n$-reduced inclusion of simplicial vector spaces
$$\oblv \circ (\free\circ \oblv)^{\circ s}(L'_\bullet) \to  \oblv \circ (\free\circ \oblv)^{\circ s}(L_\bullet).$$
Thus, by Lemma~\ref{lemma: homotopy excision, free}, the map
$$B_s(L_\bullet)//B_s(L'_\bullet) \to B_s(L_\bullet)/B_s(L'_\bullet) $$
is $(n+2)$-connected for each $s\geq 0$. Moreover, 
$$\pi_i(B_s(L_\bullet)//B_s(L'_\bullet))=0$$
for $s\geq 0$ and $i\leq n$.

There are (chains of) weak equivalences
$$
L_\bullet//L'_\bullet \simeq d(B_\bullet(L_\bullet)//B_\bullet(L'_\bullet)),
$$
$$
L_\bullet/L'_\bullet \simeq d(B_\bullet(L_\bullet)/B_\bullet(L'_\bullet)),
$$
where $d\colon \ssVect_\kk \to \sVect_\kk$, $d(V_{\bullet,\bullet})_s=V_{s,s}$, $V_{\bullet,\bullet} \in \ssVect_{\kk}$ is the diagonal simplicial vector space. Thus, there is a map of strongly convergent spectral sequences
$$
\begin{tikzcd}
E^1_{s,t}=\pi_t(B_s(L_\bullet)//B_s(L'_\bullet)) \arrow[d] \arrow[r, Rightarrow]
&\pi_{s+t}(L_\bullet//L'_\bullet) \arrow[d] \\
\widetilde{E}^1_{s,t}=\pi_t(B_s(L_\bullet)/B_s(L'_\bullet)) \arrow[r, Rightarrow]
& \pi_{s+t}(L_\bullet/L'_\bullet),
\end{tikzcd}
$$
where the differentials act as follows $d^r\colon E^r_{s,t}\to E^r_{s-r,t+r-1}$, $\tilde{d}^r\colon \widetilde{E}^r_{s,t}\to \widetilde{E}^r_{s-r,t+r-1}$.

Note that $E^1_{s,t}=\widetilde{E}^1_{s,t}=0$ if $t\leq n$, $E^1_{s,n+1} \to \widetilde{E}^1_{s,n+1}$ is an isomorphism for all $s\geq 0$, and $E^1_{s,n+2} \twoheadrightarrow \widetilde{E}^1_{s,n+2}$ is a surjection for all $s\geq 0$. Therefore the map $E^2_{s,n+1} \to \widetilde{E}^2_{s,n+1}$ is an isomorphism for all $s\geq 0$, and $E^2_{0,n+2} \twoheadrightarrow \widetilde{E}^2_{0,n+2}$ is a surjection. Hence, $E^{\infty}_{0,n+1} \cong \widetilde{E}^{\infty}_{0,n+1}$, $E^{\infty}_{1,n+1} \cong \widetilde{E}^{\infty}_{1,n+1}$, and $E^{\infty}_{0,n+2} \twoheadrightarrow \widetilde{E}^{\infty}_{0,n+2}$ is a surjection. This implies the theorem.
\end{proof}

\subsection{Homology and cohomology}\label{section: homology and cohomology}

Recall that $\kk\{\xi\}$ is the twisted polynomial ring (Definition~\ref{definition: twisted polynomial ring}), $\Mod_{\kk\{\xi\}}$ is the abelian category of left $\kk\{\xi\}$-modules, and 
$$\trivxi \colon \Mod_{\kk\{\xi\}} \to \rLie $$
is the functor which maps a left module $M$ to the restricted Lie algebra $\trivxi(M)$ with the underlying vector spaces equal to $M$, $p$-operation given by $\xi$, and with identically zero bracket.

We observe that the functor $\trivxi$ has a left adjoint
$$\Ab\colon \rLie \to \Mod_{\kk\{\xi\}} $$
given by $L\mapsto L/[L,L]$, where $[L,L]\subset L$ is the (restricted) Lie ideal generated by all elements of the form $[x,y]$, $x,y\in L$.

We extend the adjoint pair $\Ab \dashv \trivxi$ degreewise to the adjunction
\begin{equation}\label{equation: quillen adjunction, ab trivxi}
\begin{tikzcd}
\Ab: \srLie \arrow[shift left=.6ex]{r}
&\sMod_{\kk\{\xi\}} :\trivxi \arrow[shift left=.6ex,swap]{l}
\end{tikzcd}
\end{equation}
between categories of simplicial objects. Note that the adjunction~\eqref{equation: quillen adjunction, ab trivxi} is a Quillen adjunction because the composite $$\oblv\circ \trivxi\colon \sMod_{\kk\{\xi\}} \to \sVect_{\kk}$$ maps a left $\kk\{\xi\}$-module to the underlying vector space, and so the functor $\trivxi$ preserves weak equivalences and fibrations.%, see also~\cite[Propositions~3.3,~3.9]{Frank15}.

\begin{rmk}\label{remark: ab maps almost-free to projective}
Note that the simplicial left $\kk\{\xi\}$-module $\Ab(L_\bullet)$ is degreewise projective (and even free) provided a simplicial restricted Lie algebra $L_\bullet\in \srLie$ is almost-free.
\end{rmk}

The pair~\eqref{equation: quillen adjunction, ab trivxi} induces the adjoint pair of derived functors
\begin{equation}\label{equation: derived adjunction, ab trivxi}
\begin{tikzcd}
\mathbb{L}\Ab: \sL \arrow[shift left=.6ex]{r}
&D_{\geq 0}(\Mod_{\kk\{\xi\}}) :\trivxi \arrow[shift left=.6ex,swap]{l}
\end{tikzcd}
\end{equation}
between underlying $\infty$-categories.

Recall that we denote by ${\Mod}^{\kk\{\xi\}}$ the abelian category of \emph{right} $\kk\{\xi\}$-modules.
\begin{dfn}\label{definition: chain complex}
Let $M\in \Mod^{\kk\{\xi\}}$ be a right $\kk\{\xi\}$-module. We define the \emph{chain complex} $\widetilde{C}_*(L_\bullet;M)\in D(\Vect_\kk)$ of $L_\bullet\in \srLie$ with coefficients in $M$ by the formula
$$\widetilde{C}_*(L_\bullet;M) = \Sigma M\otimes_{\kk\{\xi\}} \mathbb{L}\Ab(L_\bullet). $$
Here $-\otimes_{\kk\{\xi\}} - $ is the derived tensor product
$$-\otimes_{\kk\{\xi\}}-\colon D({\Mod}^{\kk\{\xi\}}) \times D(\Mod_{\kk\{\xi\}}) \to D(\Vect_{\kk}). $$
Furthermore, we define the \emph{$s$-th homology group} $\widetilde{H}_s(L_\bullet;M)$ of $L_\bullet\in \srLie$ with coefficients in $M$ by the next rule
$$ \widetilde{H}_s(L_\bullet;M) = \pi_s(\widetilde{C}_*(L_\bullet;M)), \;\; s\geq 0.$$
\end{dfn}

%Since the functor $\widetilde{C}_*(-,M)\colon \sL\to D_{\geq 0}(\Vect_{\kk})$ commutes with (homotopy) colimits, we have the Bousfield-Kan spectral sequence which computes the homology groups of any simplicial restricted Lie algebra from the homology groups of constant ones.

%\begin{prop}\label{proposition: spectral sequence, Lie homology}
%Let $L_\bullet \in \srLie$ be a simplicial restricted Lie algebra and let $M\in \Mod^{\kk\{\xi\}}$ be a right $\Mod^{\kk\{\xi\}}$-module. Then exists a strongly convergent spectral sequence
%$$E^2_{n,m}= \pi_n \widetilde{H}_m(L_n,M) \to \widetilde{H}_{n+m}(L_\bullet,M), $$
%where $\widetilde{H}_*(L_n,M)$ is the homology groups of the constant simplicial restricted Lie algebra $L_n\in \rLie\subset \srLie$. In the spectral sequence, the differential $d^r$ acts as follows $d^r\colon E^r_{n,m} \to E^r_{n-r,m+r-1}$. \qed
%\end{prop}

Consider the field $\kk$ as a (diagonal) $\kk\{\xi\}$-bimodule with $\xi$ acting by zero, then we have $\kk\otimes_{\kk\{\xi\}} \Ab(L)=\Abxi(L)$, see formula~\eqref{equation: abxi}. Thus, Proposition~\ref{proposition: chain coalgebra and chain complex} together with the Eilenberg-Zilber theorem imply that the homology groups 
$$\widetilde{H}_*(L_\bullet;\kk)=\bigoplus_{n>0}\widetilde{H}_n(L_\bullet;\kk), \;\; L_\bullet\in \srLie $$ form naturally a graded non-unital cocommutative coalgebra. Let us denote by $H_*(L_\bullet;\kk)$ the graded \emph{coaugmented} coalgebra associated with $\widetilde{H}_*(L_\bullet;\kk)$, i.e. $$H_*(L_\bullet;\kk)=\kk\oplus \widetilde{H}_*(L_\bullet;\kk),$$
where the first summand is in degree 0.

\begin{cor}[K\"{u}nneth formula]\label{corollary: kunneth formula}
There is a natural isomorphism
$$H_*(L_\bullet\times L'_\bullet;\kk)\cong H_*(L_\bullet;\kk)\otimes H_*(L'_\bullet;\kk),$$
where $L_\bullet,L'_\bullet\in\srLie$. \qed
\end{cor}

\begin{rmk}\label{remark: kunneth formula any coefficient}
We are not aware if the K\"{u}nneth formula is true for homology groups with any other coefficients.
\end{rmk}

Note that the functor $\Ab\colon \srLie \to \sMod_{\kk\{\xi\}}$ comes with the natural transformation 
\begin{equation}\label{equation: natural transformation, hurewicz}
\id \to \Ab 
\end{equation}
given by the quotient map $L_\bullet \to L_\bullet/[L_\bullet,L_\bullet]=\Ab(L_\bullet)$. This natural transformation induces the \emph{Hurewicz homomorphism}
\begin{equation}\label{equation: hurewicz homomorphism}
h\colon \pi_s(L_\bullet) \to \widetilde{H}_{s+1}(L_\bullet;\kk\{\xi\}), \;\; L_\bullet\in \srLie, \;\; s\geq 0. 
\end{equation}
We notice that both sides of~\eqref{equation: hurewicz homomorphism} are naturally endowed with an action of $\xi$. Indeed, the homology groups $\widetilde{H}_{s}(L_\bullet;\kk\{\xi\}), s\geq 1$ are left $\kk\{\xi\}$-modules by Definition~\ref{definition: chain complex}; and the $p$-operation $\xi \colon L_\bullet \to L_\bullet$, $L_\bullet\in \srLie$ is a map of \emph{simplicial sets}, so it induces a (in general, non-linear) map
\begin{equation}\label{equation: xi-action on homotopy groups}
\xi_* \colon \pi_*(L_\bullet) \to \pi_*(L_\bullet).
\end{equation}
We point out here that the Hurewicz homomorphism~\eqref{equation: hurewicz homomorphism} is compatible with these $\xi$-actions on both sides. Finally, we notice that $\pi_0(L_\bullet), L_\bullet \in \srLie$ is itself a restricted Lie algebra.

\begin{thm}[Hurewicz theorem]\label{theorem: hurewicz theorem}
Let $L_\bullet\in \srLie$ be a simplicial restricted Lie algebra. Then the Hurewicz homomorphism $h\colon \pi_0(L_\bullet) \to \widetilde{H}_1(L_\bullet;\kk\{\xi\})$ induces an isomorphism
$$\Ab(\pi_0(L_\bullet)) \cong \widetilde{H}_1(L_\bullet;\kk\{\xi\}).$$
If $\pi_i(L_\bullet)=0$ for $0\leq i \leq n$, then
$$h\colon \pi_{n+1}(L_\bullet) \xrightarrow{\cong} \widetilde{H}_{n+2}(L_\bullet; \kk\{\xi\}) $$
is an isomorphism, and 
$$h\colon \pi_{n+2}(L_\bullet) \twoheadrightarrow \widetilde{H}_{n+3}(L_\bullet; \kk\{\xi\}) $$
is a surjection.
\end{thm}

\begin{proof}
The first statement is clear, since the functor $\Ab\colon \rLie \to \Mod_{\kk\{\xi\}}$ is a left adjoint, and so it commutes with colimits. 

For the second part, we use Lemma~\ref{lemma: reduced replacement}, so we can assume that $L_\bullet \in \srLie$ is an almost-free object $L_\bullet=\free(V_\bullet)$, $V_\bullet \in \tilde{\mathsf{s}}\Vect_{\kk}$ and $V_i=0$ for $0\leq i \leq n$. %Indeed, if $L \in \srLie$ has $\pi_i(L)=0$ for $0\leq i\leq k$, then the $(k+1)$-coskeleton $\cosk_{k+1}(L)$ of $L$ has trivial homotopy groups, and the kernel $$\widetilde{L} = \ker(L \to \cosk_{k+1}L)$$ is weakly equivalent to $L$ and $k$-reduced. Finally, any $k$-reduced simplicial restricted Lie algebra has a $k$-reduced almost-free replacement.

We show that $\pi_{n+1}[L_\bullet,L_\bullet]=0$, where $[L_\bullet,L_\bullet]\subset L_\bullet$ is the Lie ideal generated by all elements of the form $[x,y]$, $x,y\in L_\bullet$. By the previous paragraph, it suffices to construct an element $\{x,y\}\in [L_{n+2},L_{n+2}]$, $x,y\in L_{n+1}$ such that 
$$\partial\{x,y\}=\sum_{i=0}^{n+3}(-1)^id_i\{x,y\}=[x,y]. $$
We set $\{x,y\}=[s_1 y,s_0 x -s_1 x]$, then the straightforward computation with the simplicial relations shows that $\partial\{x,y\}=[x,y]$.

Finally, there is a short exact sequence of simplicial vector spaces
$$0\to \oblv([L_\bullet,L_\bullet]) \to  \oblv L_\bullet \to \Ab(L_\bullet) \to 0 $$
which induces the long exact sequence of homotopy groups
\begin{align*}
\ldots &\to \pi_{n+2}[L_\bullet,L_\bullet] \to \pi_{n+2}(L_\bullet) \to \pi_{n+2}\Ab(L_\bullet) \\
&\to \pi_{n+1}[L_\bullet,L_\bullet] \to \pi_{n+1}(L_\bullet) \to \pi_{n+1}\Ab(L_\bullet) \to 0.
\end{align*}
Since $\pi_{n+1}[L_\bullet,L_\bullet]=0$, the theorem follows.
\end{proof}

\begin{rmk}\label{remark: ellis}
Our proof is almost identical to the proof of the Hurewicz theorem in the category of simplicial Lie algebras (non-necessary equipped with a $p$-operation), see~\cite[Theorem~8]{Ellis93}.
\end{rmk}

We say that $(L_\bullet,A_\bullet)$ is a \emph{pair} in $\srLie$ if $A_\bullet$ is a simplicial restricted Lie subalgebra of $L_\bullet$. A map of pairs $$f\colon (L'_\bullet, A'_\bullet) \to (L_\bullet, A_\bullet)$$
is a map $f\colon L'_\bullet \to L_\bullet$ in $\srLie$ such that $f(A'_\bullet)\subset A_\bullet$. Notice that we do not ask $A_\bullet$ be a Lie ideal in $L_\bullet$ in opposite to~\cite[\S 8]{May70operations}.

\begin{dfn}\label{definition: relative homotopy and homology}
Let $(L_\bullet,A_\bullet)$ be a pair $\srLie$. We define the \emph{$s$-th relative homotopy group} $\pi_s(L_\bullet, A_\bullet)$ by the formula
$$\pi_s(L_\bullet, A_\bullet) = \pi_{s-1}(\fib(\iota)), \;\; s\geq 1,$$
where $\fib(\iota) \in \srLie$ is the homotopy fiber of $\iota\colon A_\bullet \hookrightarrow L_\bullet$. Similarly, we define the \emph{$s$-th relative homology group} $H_s(L_\bullet, A_\bullet; M)$, $M\in \Mod_{\kk\{\xi\}}$ as follows
$$H_s(L_\bullet, A_\bullet;M) = \widetilde{H}_s(\cofib(\iota);M), \;\; s\geq 0, $$
where $\cofib(\iota)\in \srLie$ is the homotopy cofiber of $\iota$.
\end{dfn}

Let $(L_\bullet,A_\bullet)$ be a pair in $\srLie$, then there are long exact sequences of homotopy and homology groups:
\begin{equation}\label{equation: les, homotopy pair}
\ldots \to \pi_{s}(A_\bullet) \to \pi_{s}(L_\bullet) \to \pi_s(L_\bullet, A_\bullet) \xrightarrow{\delta} \pi_{s-1}(A_\bullet) \to \ldots
\end{equation}
\begin{align}\label{equation: les, homology pair}
\ldots &\to \widetilde{H}_{s}(A_\bullet;M) \to \widetilde{H}_{s}(L_\bullet;M) \to H_s(L_\bullet, A_\bullet;M) \\
&\xrightarrow{\delta} \widetilde{H}_{s-1}(A_\bullet;M) \to \widetilde{H}_{s-1}(L_\bullet;M) \to H_{s-1}(L_\bullet, A_\bullet;M) \to \ldots \nonumber
\end{align}
Furthermore, there are similar long exact sequences for \emph{triples} in $\srLie$. Namely, if $(L_\bullet, A_\bullet, B_\bullet)$ is a triple in $\srLie$, that is $B_\bullet \subset A_\bullet \subset L_\bullet$, then there are long sequences
\begin{equation}\label{equation: les, homotopy triple}
\ldots \to \pi_{s}(A_\bullet,B_\bullet) \to \pi_{s}(L_\bullet,B_\bullet) \to \pi_s(L_\bullet, A_\bullet) \xrightarrow{\delta} \pi_{s-1}(A_\bullet,B_\bullet) \to \ldots
\end{equation}
\begin{align}\label{equation: les, homology triple}
\ldots &\to H_{s}(A_\bullet,B_\bullet;M) \to H_{s}(L_\bullet,B_\bullet;M) \to H_s(L_\bullet, A_\bullet;M) \\ &\xrightarrow{\delta} H_{s-1}(A_\bullet, B_\bullet;M) \to H_{s-1}(L_\bullet,B_\bullet;M) \to H_{s-1}(L_\bullet, A_\bullet;M) \to \ldots \nonumber
\end{align}

Note that the natural map $\fib(\iota) \to \Omega \cofib(\iota)$ induces the \emph{relative Hurewicz homomorphism}:
$$h\colon \pi_{s}(L_\bullet, A_\bullet) \to H_{s+1}(L_\bullet, A_\bullet;\kk\{\xi\}), $$ 
compatible with the exact sequences above. We say that a pair $(L_\bullet,A_\bullet)\in \srLie$ is \emph{$n$-connected} if $\pi_i(L_\bullet,A_\bullet)=0$ for each $i\leq n$. The homotopy excision theorem~\ref{theorem: homotopy excision} immediately implies the following corollary.

\begin{cor}[Relative Hurewicz theorem]\label{corollary: relative hurewicz theorem}
Suppose that $(L_\bullet,A_\bullet)$ is a $n$-connected pair in $\srLie$ and $\pi_0(A_\bullet)=0$, $n\geq 1$. Then the relative Hurewicz homomorphism
$$\pi_{n+1}(L_\bullet, A_\bullet) \to H_{n+2}(L_\bullet, A_\bullet; \kk\{\xi\}) $$
is an isomorphism, and 
$$\pi_{n+2}(L_\bullet, A_\bullet) \to H_{n+3}(L_\bullet, A_\bullet; \kk\{\xi\}) $$
is a surjection. \qed
\end{cor}

As usual, the relative Hurewicz theorem implies the homological Whitehead theorem.

\begin{cor}[Homological Whitehead theorem]\label{corollary: whitehead theorem}
Let $f\colon L'_\bullet \to L_\bullet$ be a map between connected simplicial restricted Lie algebras. Then $f$ is a weak equivalence if and only if the induced map 
$$f_*\colon \widetilde{H}_i(L'_\bullet; \kk\{\xi\}) \to \widetilde{H}_i(L_\bullet; \kk\{\xi\})$$ 
is an isomorphism for all $i\geq 1$. \qed
\end{cor}

Let $M\in \Mod_{\kk\{\xi\}}$ be a left $\kk\{\xi\}$-module and consider the abelian Lie algebra $\trivxi(M)\in \rLie \subset \srLie$. We will compute $H_*(\trivxi(M);\kk)$ provided $M$ is torsion-free. First, by the Hurewicz theorem, we have 
\begin{equation}\label{equation: homology of K(M,0), eq1}
H_1(\trivxi(M);\kk)\cong \kk\otimes_{\kk\{\xi\}} M = M/\xi(M), \; M\in \Mod_{\kk\{\xi\}}
\end{equation}
%for any left $\kk\{\xi\}$-module $M$. 

Next, using the addition operation $M\times M \to M$, $(m_1,m_2)\mapsto m_1+m_2$ we observe that $\trivxi(M)$ is canonically a commutative group object of $\srLie$. Therefore, by the K\"{u}nneth formula, $H_*(\trivxi(M);\kk)$ is a graded commutative and cocommutative Hopf algebra.  Using the isomorphism~\eqref{equation: homology of K(M,0), eq1} we get a natural map of graded Hopf algebras
\begin{equation}\label{equation: homology of K(M,0), eq2}
\gamma\colon\Sym^*(M/\xi(M)) \to H_*(\trivxi(M);\kk),
\end{equation}
where $\Sym^*(M/\xi(M))$ is the free graded commutative algebra generated by the vector space $M/\xi(M)$. 

Recall that a left $\kk\{\xi\}$-module $M$ is called \emph{torsion-free} if for any non-zero $a\in \kk\{\xi\}$ the equation $ax=0, x\in M$ implies $x=0$.

\begin{prop}\label{proposition: homology of abelian Lie algebra}
Let $M\in \Mod_{\kk\{\xi\}}$ be a torsion-free left $\kk\{\xi\}$-module. Then the map $\gamma$ naturally factorizes via an isomorphism
$$\gamma'\colon \Lambda^{*}(M/\xi(M)) \xrightarrow{\cong} H_*(\trivxi(M);\kk), $$
where $\Lambda^{*}(M/\xi(M))$ is the exterior algebra generated by $M/\xi(M)$.
\end{prop}

\begin{proof}
First, assume that $M\cong\kk\{\xi\}$ is a free $\kk\{\xi\}$-module of rank $1$. Then the abelian Lie algebra $\trivxi(M)$ is canonically isomorphic to the free restricted Lie algebra $\free(M/\xi(M))$. Therefore,
$$H_i(\trivxi(M),\kk)\cong H_i(\free(M/\xi(M));\kk)\cong\left\{
\begin{array}{ll}
\kk & \mbox{if $ i=0$,}\\
M/\xi(M) & \mbox{if $i=1$,}\\
0, & \mbox{otherwise;}
\end{array}
\right.
$$
and the proposition holds for free modules of rank $1$.

Next, using an induction on the rank $r$ and the K\"{u}nneth formula we obtain the desired isomorphism for all free $\kk\{\xi\}$-modules of finite rank, i.e. for $M\cong \kk\{\xi\}^{\oplus r}$. By Corollary~\ref{corollary: torsion-free is a colimit of free}, any torsion-free $\kk\{\xi\}$-module is a filtered colimit of finitely generated free modules, which implies the proposition.
\end{proof}

\begin{exmp}\label{example: lie algebra with trivial coeffients}
Let $M=\kk\{\xi^{\pm}\}=\kk\{\xi\}[1/\xi]$ be the ring of twisted Laurent polynomials, see Section~\ref{section: derived xi-complete modules}. By Proposition~\ref{proposition: homology of abelian Lie algebra}, we have
$$\widetilde{H}_*(\trivxi(M);\kk)\cong 0, \;\; \text{but} \;\; \widetilde{H}_1(\trivxi(M);\kk\{\xi\})\cong M\neq 0.$$
Therefore, 
$$\widetilde{H}_*(\trivxi(\Sigma^n M);\kk)\cong 0, \;\; \text{but} \;\; \widetilde{H}_{n+1}(\trivxi(\Sigma^n M);\kk\{\xi\})\cong M\neq 0 $$
for all $n\geq 0$. Here $\Sigma^n M\in D_{\geq 0}(\Mod_{\kk{\xi}})$ is the shift of $M$.
\end{exmp}

In a similar way, one can also define the cohomology groups of $L_\bullet\in \srLie$.
\begin{dfn}\label{definition: cochain complex}
Let $M\in \Mod_{\kk\{\xi\}}$ be a \emph{left} $\kk\{\xi\}$-module. We define the \emph{cochain complex} $\widetilde{C}^*(L_\bullet;M)\in D(\Vect_\kk)$ of $L_\bullet\in \srLie$ with coefficients in $M$ as follows
$$\widetilde{C}^*(L_\bullet;M) = \RHom_{\kk\{\xi\}} (\Sigma \mathbb{L}\Ab(L_\bullet), M). $$
Here $\RHom_{\kk\{\xi\}}(-,-) $ is the derived $\Hom$-functor
$$\RHom_{\kk\{\xi\}}(-,-)\colon D({\Mod}_{\kk\{\xi\}}) \times D(\Mod_{\kk\{\xi\}}) \to D(\Vect_{\kk}). $$
Furthermore, we define the \emph{$s$-th cohomology group} $\widetilde{H}^s(L_\bullet;M)$ of $L_\bullet\in \srLie$ with coefficients in $M$ by the rule
$$ \widetilde{H}^s(L_\bullet;M) = \pi_{-s}(\widetilde{C}^*(L_\bullet;M)), \;\; s\geq 0.$$
\end{dfn}

\begin{rmk}\label{remark: universal coefficient formula}
Similar to the case of singular cohomology of spaces, we have the \emph{universal coefficient formula}. Namely, there is a natural exact sequence 
\begin{align*}
0 \to \Ext^1_{\kk\{\xi\}}(\widetilde{H}_{s-1}(L_\bullet;\kk\{\xi\}),M) &\to \widetilde{H}^s(L_\bullet;M)\\
&\to \Hom_{\kk\{\xi\}}(\widetilde{H}_s(L_\bullet;\kk\{\xi\}),M) \to 0 
\end{align*}
for any $L_\bullet\in \srLie$, $M\in \Mod_{\kk\{\xi\}}$, and $s\geq 1$. Moreover, this exact sequence splits, but the splitting may not be natural.
\end{rmk}

\begin{rmk}\label{remark: cohomology are representable}
The adjoint pair~\eqref{equation: derived adjunction, ab trivxi} implies that the homotopy functor $$\widetilde{H}^s(-;M) \colon \mathrm{Ho}(\srLie)^{op} \to \Vect_{\kk}$$
is representable by the abelian Lie algebra $\trivxi\Sigma^{s-1}M \in \srLie$, cf.~\cite[Proposition~3.4]{Priddy70long}. Here $\mathrm{Ho}(\srLie)$ is the homotopy category of the model category $\srLie$, see Theorem~\ref{theorem:modelsrlie}.
\end{rmk}

\subsection{Postnikov system}\label{section: postnikov system}
Let us denote by $\Delta_{\leq (n+1)} \subset \Delta$ the full subcategory in the simplex category $\Delta$ spanned by $[i]$, $i\leq n+1$. If $\mathsf{C}$ is a category, we write $\mathsf{s}_{\leq (n+1)}\mathsf{C}$ for the category of contravariant functors from $\Delta_{\leq (n+1)}$ to $\mathsf{C}$. Assume that the category $\mathsf{C}$ is complete, then the restriction functor
$$\mathrm{tr}_{(n+1)}^*\colon \sC \to \mathsf{s}_{\leq (n+1)} \mathsf{C} $$
has a right adjoint
$$\mathrm{tr}_{(n+1)*}\colon \mathsf{s}_{\leq (n+1)} \mathsf{C} \to \sC. $$
We write $$\cosk_{n+1}\colon \sC \to \sC$$ for the composite $\mathrm{tr}_{(n+1)*}\circ \mathrm{tr}_{(n+1)}^*$ and $\alpha^n\colon \id \to \cosk_{n+1}$ for the unit map.

Assume that $\mathsf{C}=\Vect_{\kk}$ is the category of vector spaces and let $V_\bullet \in \sVect_{\kk}$. Then $\pi_i(\cosk_{n+1}V_\bullet)=0$ for $i>n$, and the induced map $\pi_{i}(\alpha^n)$ is an isomorphism for $i\leq n$, see~\cite[Section~II.8]{MaySimplicial}.

Since the functor $\oblv\colon \srLie \to \sVect_\kk$ is a right adjoint, it preserves limits, and so we have a natural isomorphism:
$$\oblv(\cosk_{n+1}(L_\bullet)) \cong \cosk_{n+1}(\oblv(L_\bullet)), \; L_\bullet\in \srLie,\; n\geq 0. $$
Therefore the natural map
\begin{equation}\label{equation: postnikov tower, eq1}
\alpha^n \colon L_\bullet \to \cosk_{n+1}L_\bullet 
\end{equation}
again induces an isomorphism on $\pi_i$ for $i\leq n$, and $\pi_i\cosk_{n+1}L_\bullet=0$ for $i>n$. 

Here we slightly change the notation: for the rest of the paper, we will write $L_\bullet^{\leq n}$ for $\cosk_{n+1}L_\bullet$, $L_\bullet\in \srLie$, $n\geq 0$. Finally, we note that $\alpha^n\colon L_\bullet \to L_\bullet^{\leq n}$ is a fibration in $\srLie$ and we write $L_{\bullet}^{>n}$ for its fiber.
%Therefore, the functor $\cosk_{k+1} \colon \srLie \to \srLie$ preserves weak equivalences, and we denote by 
%$$(-)^{\leq k} \colon \sL \to \sL $$
%its derived functor. Furthermore, let us denote by $(-)^{>k}\colon \sL \to \sL$ the (homotopy) fiber $\fib(\alpha^k)$ of the natural transformation $\alpha^k\colon \id \to (-)^{\leq k}$.

%\begin{cor}\label{corollary: xi-action is linear}
%The map $\xi_*\colon \pi_n(L_\bullet) \to \pi_n(L_\bullet)$ of~\eqref{equation: xi-action on homotopy groups} is semi-linear for any $L_\bullet\in \srLie$ and $n>0$.
%\end{cor}

%\begin{proof}
%By the Hurewicz theorem~\ref{theorem: hurewicz theorem}, we have a chain of isomorphisms which are compatible with $\xi$-action
%\begin{equation*}
%\pi_n(L_\bullet)\cong \pi_n(L^{>(n-1)}_\bullet) \cong \widetilde{H}_{n+1}(L_\bullet^{>(n-1)};\kk\{\xi\}), \;\; n > 0. \qedhere
%\end{equation*}
%\end{proof}

\begin{dfn}\label{definition: EM}
Let $M$ be a left $\kk\{\xi \}$-module and $n\geq 0$. A simplicial restricted Lie algebra $L_\bullet\in \srLie$ is an \emph{Eilenberg–MacLane Lie algebra} of type $K(M,n)$ if it has the $n$-th homotopy group $\pi_{n}(L_\bullet)$ isomorphic to $M$ (as a left $\kk\{\xi\}$-module) and all other homotopy groups are trivial. 
\end{dfn}
\begin{exmp}\label{example: EM}
An abelian Lie algebra $\trivxi \Sigma^n M \in \srLie, M\in \Mod_{\kk\{\xi\}}$ is the Eilenberg-MacLane Lie algebra of type $K(M,n)$.
\end{exmp}

The Hurewicz theorem together with Remarks~\ref{remark: universal coefficient formula} and~\ref{remark: cohomology are representable}  immediately implies that there is a unique (up to a weak equivalence) Eilenberg-MacLane Lie algebra in $\srLie$ of a given type $K(M,n)$. Therefore we will abuse notation and call any such Lie algebra by $K(M,n)$.

\begin{prop}\label{proposition: fibration with EM as fiber}
Let $f\colon L'_\bullet \to L_\bullet$ be a map in $\srLie$ such that $\pi_0(L'_\bullet)=\pi_0(L_\bullet)=0$ and the homotopy fiber $\fib(f)$ is an Eilenberg-MacLane Lie algebra $K(M,n)$, $M\in \Mod_{\kk\{\xi\}}$, $n\geq 0$. Then there exists a map $$k\colon L_\bullet \to K(M,n+1)$$ such that the sequence
\begin{equation}\label{equation: k-invariant}
L'_\bullet \xrightarrow{f} L_\bullet \xrightarrow{k} K(M,n+1) 
\end{equation}
is a fiber sequence in $\srLie$.
\end{prop}

\begin{proof}
Note that the map $f\colon L'_\bullet\to L_\bullet$ is $n$-connected. Therefore, by the homotopy excision theorem~\ref{theorem: homotopy excision}, we have 
$$\pi_{i}(\cofib(f)) =0, \;\; i\leq n,$$ 
$$\pi_{n+1}(\cofib(f)) \cong \pi_{n}(\fib(f)) =M, $$
where $\cofib(f)$ is the cofiber of the map $f$. Hence the simplicial restricted Lie algebra $\cofib(f)^{\leq (n+1)}\in \srLie$ is an Eilenberg-MacLane Lie algebra of type $K(M,n+1)$, and we define the desired map $k$ as the composite 
$$L_\bullet \to \cofib(f) \to \cofib(f)^{\leq (n+1)} \simeq K(M,n+1).$$ 
The straightforward diagram chase shows that the sequence~\eqref{equation: k-invariant} is a fiber sequence.
\end{proof}

We summarize the results of this section in the next corollary.

\begin{cor}[Postnikov tower]\label{corollary: postnikov tower}
Let $L_\bullet\in \srLie$ be a simplicial restricted Lie algebra. Then there are a natural tower of fibrations
$$ \ldots \xrightarrow{\beta^{n+1}} L_\bullet^{\leq (n+1)} \xrightarrow{\beta^{n}} L_\bullet^{\leq n} \xrightarrow{\beta^{n-1}} \ldots $$
and compatible maps $\alpha^n\colon L_\bullet \to L_\bullet^{\leq n}$ such that
\begin{enumerate}
\item $\pi_i(L_\bullet^{\leq n})=0$ if $i>n$;
\item the induced map $\pi_i(\alpha^n)$ is an isomorphism for $i\leq n$;
\item $L_\bullet\simeq \holim_n L_\bullet^{\leq n}$.
\end{enumerate}
Moreover, if $\pi_0(L_\bullet)=0$, then %each map $\beta_n\colon L^{\leq (n+1)}_\bullet \to L_\bullet^{\leq n}$ is weakly equivalent to a twisted central extension with fiber $M_\bullet$, such that each $M_n, n\geq 0$ is a free left $\kk\{\xi\}$-module. In particular, 
there exist \emph{$k$-invariants}, i.e. there are maps $$k^n\colon L_\bullet^{\leq n} \to K(\pi_{n+1}(L_\bullet), n+2), \;\; n\geq 0$$ such that for each $n\geq 0$ the sequence
$$
L_\bullet^{\leq n+1} \xrightarrow{\beta^{n}} L_\bullet^{\leq n} \xrightarrow{k^n} K(\pi_{n+1}(L_\bullet),n+2) $$
is a fiber sequence in $\srLie$. \qed
\end{cor}

\subsection{Principal fibrations}\label{section: principal fibration}
In this section we will sketch the theory of principal fibrations in the category $\srLie$ of simplicial restricted Lie algebras. The corresponding theory in the category $\sSet$ of simplicial sets is well-known, and we will follow  along its lines. We will use Sections~V.2-V.3 from \cite{GoerssJardine} as our main reference.

Since the category $\Mod_{\kk\{\xi\}}$ of left $\kk\{\xi\}$-modules is abelian, we have the natural equivalence
\begin{equation*}
\Grp(\Mod_{\kk\{\xi\}}) \xrightarrow{\simeq} \Mod_{\kk\{\xi\}},
\end{equation*}
where the left hand side is the category of group objects in $\Mod_{\kk\{\xi\}}$. Moreover, it is not hard to see that the functor $\trivxi$ induces the equivalence
\begin{equation*}
\trivxi\colon \Grp(\Mod_{\kk\{\xi\}}) \xrightarrow{\simeq} \Grp(\rLie).
\end{equation*}
Indeed, if $\mu\colon L\times L \to L$ is a group multiplication in $\rLie$, then $\mu$ coincides with the usual vector addition, and the latter is a Lie algebra homomorphism if and only if the Lie bracket is trivial. Similarly, we have equivalences for the categories of simplicial objects:
$$\Grp(\sMod_{\kk\{\xi\}}) \xrightarrow{\simeq} \sMod_{\kk\{\xi\}},$$
$$\trivxi\colon \Grp(\sMod_{\kk\{\xi\}}) \xrightarrow{\simeq} \Grp(\srLie).$$

Let $M_\bullet \in \sMod_{\kk\{\xi\}}$ be a simplicial left $\kk\{\xi\}$-module. We say that $M_\bullet$ \emph{acts} on a simplicial restricted Lie algebra $L_\bullet\in \srLie$ if there exists a morphism in $\srLie$
$$\mu\colon \trivxi M_\bullet \times L_\bullet \to L_\bullet $$
which satisfies the associative and unit axioms. Note that $$\bar{\mu}\colon \trivxi M_\bullet \to L_\bullet, \; \bar{\mu}(-)=\mu(-,0)$$ is a map of simplicial restricted Lie algebras, $\bar\mu(M_\bullet)$ is a Lie ideal in $L_\bullet$, and $\mu(m,l)=\bar\mu(m)+l$.

Let $M_\bullet$ act on $L_\bullet \in \srLie$. We denote by $L_\bullet/M_\bullet$ the \emph{group action quotient}, i.e.
$$
\begin{tikzcd}%[column sep=large]
L_\bullet/M_\bullet=\mathrm{coeq}(\trivxi M_\bullet \times L_\bullet \arrow[shift left=.75ex]{r}{\mu}
  \arrow[shift right=.75ex,swap]{r}{pr_2}
&
L_\bullet).
\end{tikzcd}
$$
We point out that this notation is consistent with the notation for a \emph{categorical} quotient associated with a single morphism used before. Indeed, we have an isomorphism
$$
\begin{tikzcd}%[column sep=large]
\mathrm{coeq}(\trivxi M_\bullet \times L_\bullet \arrow[shift left=.75ex]{r}{\mu}
  \arrow[shift right=.75ex,swap]{r}{pr_2}
&
L_\bullet) \cong \mathrm{coeq}(\trivxi M_\bullet \arrow[shift left=.75ex]{r}{\bar{\mu}}
\arrow[shift right=.75ex,swap]{r}{0}
&
L_\bullet).
\end{tikzcd}
$$
Finally, we say that $M \in \Mod_{\kk\{\xi\}}$ acts \emph{freely} on $L\in \rLie$ if there is a $M$-equivariant isomorphism $$ L \cong \trivxi(M) \times X $$ in $\rLie$, where $M$ acts on the right hand side via $\trivxi(M)$. Note that the action is free if and only if $\bar{\mu}\colon \trivxi M \to L$ is a \emph{split} monomorphism in $\rLie$ and $\bar{\mu}(\trivxi M)$ is a \emph{Lie ideal} in $L$.

Let $\srLie_{M_\bullet}$ be the category of simplicial restricted Lie algebras with $M_\bullet$-action. Note that $\srLie_{M_\bullet}$ is a  simplicial category such that the forgetful functor
$$\oblv \colon \srLie_{M_\bullet} \to \srLie $$
is simplicial. Moreover, $\oblv$ has a left adjoint
$$\trivxi M_\bullet \times (-)\colon \srLie \to \srLie_{M_\bullet},$$
which is also simplicial. 

\begin{lmm}\label{lemma: forgetting the M-action preserves pushout}
The adjoint pair
$$
\begin{tikzcd}
\trivxi M_\bullet \times (-): \srLie \arrow[shift left=.6ex]{r}
	&\srLie_{M_\bullet} :\oblv \arrow[shift left=.6ex,swap]{l}
\end{tikzcd}
$$
is monadic and the forgetful functor $\oblv$ preserves pushout squares.
\end{lmm}

\begin{proof}
The adjoint pair is monadic by the definition of the category $\srLie_{M_\bullet}$. For the second part, it suffices to show that the endofunctor
$$\trivxi M_\bullet \times (-) \colon \srLie \to \srLie $$
preserves pushout squares, see~\cite[Proposition~4.3.2]{Borceux94}. Finally, the assertion follows by Proposition~\ref{proposition: pushouts and products in rlie}.
\end{proof}

Since the model category $\srLie$ is cofibrantly generated (Theorem~\ref{theorem:modelsrlie}), we can transfer this model structure to $\srLie_{M_\bullet}$. Namely, we obtain the following theorem.

\begin{thm}\label{theorem: model structure, srlie with action}
There exists a simplicial combinatorial model structure on $\srLie_{M_\bullet}$ such that $f\colon L'_\bullet \to L_\bullet$ is

\begin{itemize}
\item a \emph{weak equivalence} if and only if $f$ is a weak equivalence in $\srLie$;
\item a \emph{fibration} if and only if $f$ is a fibration in $\srLie$ (i.e. $\oblv(f)$ is a fibration in $\sVect_{\kk}$, see Remark~\ref{remark: surjective on components});
\item a \emph{cofibration} if and only if $f$ has the left lifting property with respect to all acyclic fibrations.
\end{itemize}
\end{thm}

\begin{proof}
We again use Theorem~11.3.2 from~\cite{Hirschhorn03}. Recall from Theorem~\ref{theorem:modelsrlie} that $I_{\Lie}$ is the set of generating cofibrations and $J_\Lie$ is the set of generating trivial cofibrations for the model structure on $\srLie$. Define the following sets of morphisms in $\srLie_{M_\bullet}$
$$I_{M_\bullet}= \{\trivxi(M_\bullet)\times u \; |\; u\in I_\Lie\}, \;\; J_{M_\bullet} = \{\trivxi(M_\bullet) \times v \; |\; v\in J_\Lie\}.$$
It suffices to show that $I_{M_\bullet}$, $J_{M_\bullet}$ permit the small object argument (see~\cite[Definition~10.5.15]{Hirschhorn03}) and the functor $$\oblv\colon \srLie_{M_\bullet} \to \srLie$$ takes $J_{M_\bullet}$-cell complexes to weak equivalences in $\srLie$. The first part is clear because $\oblv$ preserves filtered colimits; and the second part is clear because $\oblv$ preserves pushouts (see Lemma~\ref{lemma: forgetting the M-action preserves pushout}) and any map in $\oblv(J_{M_\bullet})$ is a homotopy equivalence.
\end{proof}

Similar to \cite[Corollary~V.2.10]{GoerssJardine}, we obtain a description of cofibrant objects in $\srLie_{M_\bullet}$.

\begin{prop}\label{proposition: cofibrant in srlie with action}
A simplicial restricted Lie $M_\bullet$-algebra $L_\bullet \in \srLie_{M_\bullet}$ is cofibrant if and only if 
\begin{enumerate}
\item $M_n$ acts freely on $L_n$ for each $n\geq 0$;
\item the quotient $L_\bullet/M_\bullet$ is a cofibrant object in $\srLie$.
\end{enumerate}
\end{prop}

We first prove an analog of \cite[Lemma~2.5]{GoerssJardine}.

\begin{lmm}\label{lemma: principal fibration over a simplex}
Let $L_\bullet\in \srLie_{M_\bullet}$ be a simplicial restricted Lie $M_\bullet$-algebra such that $M_i$ acts freely on $L_i$ for each $i\geq 0$ and the quotient $L_\bullet/M_\bullet=\free(\kk(\Delta^n))$, where $\kk(\Delta^n)$ is the simplicial vector space spanned by $\Delta^n\in \sSet$. Then there is an isomorphism $\psi\colon \trivxi(M_\bullet)\times\free(\kk(\Delta^n)) \xrightarrow{\cong} L_\bullet$ in $\srLie_{M_\bullet}$ such that the diagram
$$
\begin{tikzcd}
\trivxi(M_\bullet)\times \free(\kk(\Delta^n)) \arrow{rr}{\psi} \arrow[swap]{dr}{pr_2}
&& L_\bullet \arrow{dl}{q}  \\
& \free(\kk(\Delta^n)),
&
\end{tikzcd} 
$$
commutes.
\end{lmm}

\begin{proof}
By various adjunctions, there is a section $z\colon \free(\kk(\Delta^n)) \to L_\bullet$ of the quotient map $q\colon L_\bullet \to \free(\kk(\Delta^n))$. Then, we define $\psi$ as the following composite
$$\psi\colon \trivxi(M_\bullet)\times \free(\kk(\Delta^n)) \xrightarrow{\id\times z} \trivxi(M_\bullet)\times L_\bullet \xrightarrow{\mu} L_\bullet, $$
where $\mu\colon \trivxi(M_\bullet) \times L_\bullet \to L_\bullet$ is the action map. Since $M_i$ acts freely on $L_i$ for each $i\geq 0$, the map $\psi$ is an isomorphism.
\end{proof}

\begin{proof}[Proof of Proposition~\ref{proposition: cofibrant in srlie with action}]
Let $L_\bullet$ be a cofibrant object in $\srLie_{M_\bullet}$, then $L_\bullet$ is a retract of an $I_{M_\bullet}$-cell complex, see~\cite[Corollary~11.2.2]{Hirschhorn03}. By the cell induction, we obtain that $M_n$ acts freely on $L_n$ for each $n\geq 0$. Furthermore, since the quotient functor $$(-)/M_\bullet\colon \srLie_{M_\bullet} \to \srLie$$ preserves colimits and $(I_{M_\bullet})/M_\bullet=I_\Lie$, we get that the quotient $L_\bullet/M_\bullet$ is a retract of an $I_\Lie$-cell complex, and so $L_\bullet/M_\bullet$ is a cofibrant object in $\srLie$.

Suppose now that the quotient $X_\bullet = L_\bullet/M_\bullet \in \srLie$ is a cofibrant object. Then $X_\bullet$ is a retract of an $I_\Lie$-cell complex $Y_\bullet$; that is $Y_\bullet=\colim_n Y_\bullet^{(n)}$, $Y_{\bullet}^{(-1)}=0$, and for each $n\geq 0$, there is a pushout diagram
$$
\begin{tikzcd}
\coprod_{\alpha} A_\alpha \arrow{d}{\sqcup f_\alpha} \arrow{r}
& Y_{\bullet}^{(n-1)} \arrow{d}  \\
\coprod_{\alpha} B_\alpha \arrow{r}
& Y_{\bullet}^{(n)}
\end{tikzcd}
$$
such that all $f_\alpha\colon A_{\alpha} \to B_\alpha$ belong to the generating set $I_\Lie$. Define $L_\bullet^{(n)}$ as a pullback of the quotient map $q\colon L_\bullet \to L_\bullet/M_\bullet$ along $Y_\bullet^{(n)} \hookrightarrow Y_\bullet \to X_\bullet$. Then $L_\bullet$ is a retract of $L'_\bullet = \colim_n L_\bullet^{(n)}$. By applying inductively Lemma~\ref{lemma: principal fibration over a simplex} together with Proposition~\ref{proposition: pushouts and products in rlie}, we obtain that $L'_\bullet$ is an $I_{M_\bullet}$-cell complex, and so $L_\bullet$ is a cofibrant object in $\srLie_{M_\bullet}$.
\end{proof}

\begin{dfn}\label{definition: lie principal fibration}
Let $M_\bullet\in \sMod_{\kk\{\xi\}}$ be a simplicial left $\kk\{\xi\}$-module. A \emph{principal fibration} (or principal $M_\bullet$-fibration) is a fibration $\pi\colon E_\bullet \to B_\bullet$ in $\srLie_{M_\bullet}$ so that
\begin{enumerate}
\item the base $B_\bullet$ is cofibrant and has trivial $M_\bullet$ action;
\item $E_\bullet$ is a cofibrant object in $\srLie_{M_\bullet}$;
\item the induced map from the quotient $E_\bullet/M_\bullet \to B_\bullet$ is an isomorphism. 
\end{enumerate}
\end{dfn}

As usual, we say that two principal fibrations $\pi\colon E_\bullet \to B_\bullet$ and $\pi'\colon E'_\bullet \to B_\bullet$ are \emph{isomorphic} if there is an isomorphism $f\colon E_\bullet \to E'_\bullet$ in $\srLie_{M_\bullet}$ such that $\pi=\pi'\circ f$. Furthermore, we point out that any map $f\colon E_\bullet \to E'_\bullet$ of principal fibrations over $B_\bullet$ such that $f/M_\bullet =\id$ is an isomorphism of principal fibrations. We will write $PF_{M_\bullet}(B_\bullet)$ for the set of isomorphism classes of principal $M_\bullet$-fibrations over $B_\bullet \in \srLie$. 

Here we fix some notation. Let $\pi\colon E_\bullet \to B_\bullet$ be a principal $M_\bullet$-fibration and let $f\colon B'_\bullet \to B_\bullet$ be any map in $\srLie$ between cofibrant objects. Then, by Proposition~\ref{proposition: cofibrant in srlie with action}, the pullback $\pi'\colon E_\bullet \times_{B_\bullet} B'_\bullet \to B'_\bullet$ is again a principal fibration and we write $E_\bullet\vert_{f}$ for its total space $E_\bullet \times_{B_\bullet}B'_\bullet$. We denote by $\bar{f}\colon E_{\bullet}\vert_{f}\to E_\bullet$ the map between total spaces induced by $f$. 

Moreover, if 
$$f\colon (B'_\bullet,A'_\bullet) \to (B_\bullet,A_\bullet)$$ is a map of pairs and $E_\bullet$ is a principal fibration over $B_\bullet$, then we denote by $E_\bullet\vert_{\partial f}$ the restriction of $E_\bullet\vert_{f}$ to $A'_\bullet$. This notation is slightly awkward and non-standard, but nevertheless it will be convenient for us in the next section. 

Note that $PF_{M_\bullet}(-)$ is a contravariant functor on the category of cofibrant objects in $\srLie$ and we will show that this is a homotopy one.

\begin{lmm}\label{lemma: principal fibrations are homotopy invariant}
Suppose that $f_1,f_2\colon B'_\bullet \to B_\bullet$ are two homotopic maps in $\srLie$ between cofibrant objects. Then $PF_{M_\bullet}(f_1)=PF_{M_\bullet}(f_2)$.
\end{lmm}

\begin{proof}
Compare with~\cite[Lemma~V.3.4]{GoerssJardine}. It is enough to consider the universal example: given a principal fibration $\pi\colon E_\bullet \to B_\bullet \times \Delta^1$, the restrictions $E_\bullet\vert_0 \to B_\bullet$ and $E_\bullet\vert_1\to B_\bullet$ over the vertices of $\Delta^1$ are isomorphic. For this consider the lifting problem in $\srLie_{M_\bullet}$
$$
\begin{tikzcd}
E_\bullet\vert_0 \arrow[swap]{d}{d^0} \arrow{r}
& E_\bullet \arrow{d}{\pi}  \\
E_\bullet\vert_0 \times \Delta^1 \arrow{r}\arrow[dashed]{ur}
& B_\bullet\times \Delta^1.
\end{tikzcd}
$$
Since $E_\bullet \vert_0$ is cofibrant in $\srLie_{M_\bullet}$, the left vertical arrow $$E_\bullet\vert_0 = E_\bullet\vert_0 \times \Delta^0 \xrightarrow{\id\times d^0} E_\bullet\vert_0 \times \Delta^1$$ is a trivial cofibration. Therefore a lifting exists and defines an isomorphism of principal fibrations $E_\bullet \vert_0 \times \Delta^1 \cong E_\bullet$ over $B_\bullet$. The pullback of the last isomorphism along $d^1$ gives the desired isomorphism $E_\bullet \vert_0 \cong E_\bullet\vert_1$.
\end{proof}

As a corollary, we obtain the following statement, cf.~\cite[Lemma~V.3.5]{GoerssJardine}.

\begin{lmm}\label{lemma: principal fibrations over contractible}
Suppose that $B_\bullet$ is a contractible cofibrant object in $\srLie$. Then any principal fibration $E_\bullet$ over $B_\bullet$ is trivializable, i.e. $E_\bullet \cong \trivxi M_\bullet \times B_\bullet$. \qed
\end{lmm}

We now define the classifying object for principal fibrations. Recall that $M_\bullet\in \sMod_{\kk\{\xi\}}$ is a simplicial left $\kk\{\xi\}$-module.

\begin{dfn}\label{definition: classifying space for principal fibrations} Let $EM_\bullet$ be any cofibrant object in $\srLie_{M_\bullet}$ such that the unique map $EM_\bullet\to 0$ is a weak equivalence. Let $BM_\bullet =EM_\bullet/M_\bullet$ and $\pi_{M_\bullet}\colon EM_\bullet \to BM_\bullet$ be the resulting principal fibration.
\end{dfn}

\begin{exmp}\label{example: kan cone as classifying space}
Note that the Kan cone $\trivxi C_\bullet M_\bullet \in \srLie$ (\cite[Section~III.5]{GoerssJardine}) is contractible and it has a left $M_\bullet$-action such that $M_n$ acts freely on 
$$\trivxi C_n M_\bullet = \trivxi(M_n\times M_{n-1} \times M_{n-2} \times \ldots \times M_0).$$ 
However, $\trivxi C_\bullet M_\bullet$ is not $EM_\bullet$ because the Kan suspension $$\trivxi \Sigma_\bullet M_\bullet = \trivxi C_\bullet M_\bullet / M_\bullet$$ is not a cofibrant object in $\srLie$. This can be fixed as follows. Let $f\colon B_\bullet \to \trivxi\Sigma_\bullet M_\bullet$ be a cofibrant replacement. Then we observe that the pullback 
$$E_\bullet = B_\bullet \times_{\trivxi \Sigma_\bullet M_\bullet} \trivxi(C_\bullet M_\bullet) $$
in $\srLie_{M_\bullet}$ is contractible, $M_n$ acts freely on $E_n$ for each $n\geq 0$, and the quotient $E_\bullet/M_\bullet= B_\bullet$ is cofibrant. In this way, we constructed a principal fibration from Definition~\ref{definition: classifying space for principal fibrations}.
\end{exmp}

\begin{thm}\label{theorem: principal fibrations, classifying}
For all cofibrant objects $B_\bullet \in \srLie$, the map
$$\theta\colon [B_\bullet, BM_\bullet] \to PF_{M_\bullet}(B_\bullet) $$
sending a class $[f]\in [B_\bullet, BM_\bullet]$ to the pullback of $\pi_{M_\bullet}\colon EM_{\bullet} \to BM_\bullet$ along $f$ is a bijection.
\end{thm}

\begin{proof}
Note that $\theta$ is well-defined by Lemma~\ref{lemma: principal fibrations are homotopy invariant}. We will show that $\theta$ is a bijection by constructing an inverse. If $\pi\colon E_\bullet \to B_\bullet$ is a principal fibration, there is a lifting in the diagram in $\srLie_{M_\bullet}$
\begin{equation}\label{equation: classifying,eq1} 
\begin{tikzcd}
0 \arrow[swap]{d} \arrow{r}
& EM_\bullet \arrow{d}  \\
E_\bullet \arrow{r}\arrow[dashed]{ur}
& 0
\end{tikzcd}
\end{equation}
because $E_\bullet$ is cofibrant and $EM_\bullet$ is fibrant. Let $f\colon B_\bullet \to BM_\bullet$ be the quotient map. We define $$\Psi\colon PF_{M_\bullet}(B_\bullet) \to [B_\bullet, BM_\bullet] $$ by sending $\pi\colon E_\bullet \to B_\bullet$ to the class of $f$. The last map is well-defined because a lifting in the diagram~\eqref{equation: classifying,eq1} is unique up to an equivariant homotopy.

Note that if $EM_\bullet \vert_{f}$ is the pullback of $\pi_{M_\bullet}$ along $f=\Psi(\pi)$, then there is a commutative diagram
$$
\begin{tikzcd}
E_\bullet \arrow{rr} \arrow{dr}{\pi}
&& EM_\bullet\vert_{f} \arrow{dl}  \\
& B_\bullet.
&
\end{tikzcd} 
$$
Here the horizontal arrow is a map of principal fibrations over the same base, and so this is an isomorphism. Therefore, $\theta \circ \Psi =\id$. On the other hand, given a representative $g\colon B_\bullet \to BM_\bullet$ of a homotopy class in $[B_\bullet, BM_\bullet]$, the map $\bar{g}$ in the pullback diagram
$$
\begin{tikzcd}
\theta(g)=EM_\bullet\vert_{g}  \arrow{d} \arrow{r}{\bar{g}}
& EM_\bullet \arrow{d}{\pi_{M_\bullet}}  \\
B_\bullet \arrow{r}{g}
& BM_\bullet
\end{tikzcd}
$$
gives a lifting in the diagram~\eqref{equation: classifying,eq1}, so $\Psi\circ \theta=\id$.
\end{proof}

\begin{cor}\label{corollary: classifying space is unique}
Let $M_\bullet \in \sMod_{\kk\{\xi\}}$ be a simplicial left $\kk\{\xi\}$-module. Then the classifying Lie algebra $BM_\bullet\in \srLie$ is unique up to homotopy. Moreover, $\pi_n(BM_\bullet)=\pi_{n-1}(M_\bullet)$.
\end{cor}

\begin{proof}
By Example~\ref{example: kan cone as classifying space}, the classifying Lie algebra $BM_\bullet$ exists, and by Theorem~\ref{theorem: principal fibrations, classifying}, $BM_\bullet$ represents a contravariant functor $$PF_{M_\bullet}\colon \mathrm{Ho}(\srLie)^{op} \to \mathsf{Set},$$ where $\mathrm{Ho}(\srLie)$ is the homotopy category of $\srLie$. Therefore, by the Yoneda lemma, $BM_\bullet$ is unique up to a homotopy. The last statement follows from the long exact sequence of homotopy groups applying to the fibration $\pi_{M_\bullet}\colon EM_\bullet\to BM_{\bullet}$.
\end{proof}

Corollary~\ref{corollary: classifying space is unique} and Proposition~\ref{proposition: fibration with EM as fiber} imply together the following statement.

\begin{cor}\label{corollary: fibration with KM is principal}
Let $\pi\colon E_\bullet \to B_\bullet$ be a fibration in $\srLie$ such that the total space $E_\bullet$ and the base $B_\bullet$ are connected, and the fiber $\fib(\pi)$ is an Eilenberg-MacLane Lie algebra $K(M,n)$, $M\in \Mod_{\kk\{\xi\}}$, $n\geq 0$. Then there exist a principal fibration $\pi'\colon E'_\bullet \to B'_\bullet$ and a commutative diagram
$$
\begin{tikzcd}
E'_\bullet \arrow[swap]{d}{\pi'} \arrow{r}{\bar{g}} 
& E_\bullet \arrow{d}{\pi} \\
B'_\bullet \arrow{r}{g} 
& B_\bullet
\end{tikzcd}
$$
such that both maps $g$ and $\bar{g}$ are  weak equivalences. \qed
\end{cor}

\subsection{Serre spectral sequence}\label{section: serre spectral sequence} 
Let $V_*\in \Vect^{gr}_{\kk}$ be a graded vector space. Throughout this section we write $V_*[t]=\Sigma^tV_*$ for the shift of $V_*$ by $t\in\Z$. 

Let $\pi\colon E_\bullet \to B_\bullet$ be a fibration in the category $\srLie$ over a reduced base $B_\bullet$, $B_0=0$ with the fiber $F_\bullet=\pi^{-1}(0)$. Note that the base Lie algebra $B_\bullet$ is equipped with the increasing complete \emph{skeletal filtration}:
$$B^{(s)}_\bullet = \left\{
\begin{array}{ll}
0 & \mbox{if $ s=0$,}\\
\sk_{s-1}B_\bullet & \mbox{if $s>0$;}
\end{array}
\right. $$
and we define the increasing filtration $E^{(s)}_\bullet$ on the total space $E_\bullet$ by taking the preimages, i.e. $E^{(s)}_\bullet = \pi^{-1}(B^{(s)}_\bullet)$ if $s\geq 0$, and $E^{(-1)}_\bullet=0$. By applying (non-reduced) homology to the filtered simplicial restricted Lie algebra $E^{(s)}_\bullet$, we obtain the following result.

\begin{thm}\label{theorem: serre spectral sequence, filtered}
Let $\pi\colon E_\bullet \to B_\bullet$ be a fibration in $\srLie$. For homology with any coefficient module $M\in \Mod^{\kk\{\xi\}}$ there is a convergent $E^1$-spectral sequence with 
$$E^1_{s,t} \cong H_{s+t}(E^{(s)}_\bullet, E^{(s-1)}_\bullet;M), $$
where the differential $d_1$ is the boundary operator of the triple $(E^{(s)}_\bullet, E^{(s-1)}_\bullet, E^{(s-2)}_\bullet)$, and $E^{\infty}$ is the bigraded module associated to the filtration of $H_*(E_\bullet;M)$ defined by
$$F_sH_*(E_\bullet;M) = \im\left(H_*(E^{(s)}_\bullet;M) \to H_*(E_\bullet; M)\right).\eqno\qed$$
\end{thm}

\begin{exmp}\label{example: serre, idenity morphism}
Let $\pi =\id \colon B_\bullet \to B_\bullet$ be the identity map, and suppose that the base Lie algebra $B_\bullet=\free(V_\bullet)$ is almost-free, $V_\bullet\in \tilde{\mathsf{s}}\Vect_{\kk}$. Then the quotient Lie algebra $B_\bullet^{(s)}/B^{(s-1)}_{\bullet}, s>0$ is a free simplicial restricted Lie algebra $\free(W_\bullet)$ generated by 
$$W_\bullet = \Gamma(NV_{s-1}[s-1]),$$ where $NV_{s-1}\subseteq V_{s-1}$ is the vector subspace of normalized chains in $V_\bullet$, $NV_{s-1}[s-1]$ is the chain complex having $NV_{s-1}$ concentrated in degree $s-1$, and $$\Gamma\colon \mathsf{Ch}_{\geq 0} \to \sVect_{\kk}$$ is the inverse to the normalized chain complex functor $N$, see Section~\ref{section: notation}.
Therefore, $E^1_{s,t}=0$ if $t\neq 0$ in the spectral sequence of Theorem~\ref{theorem: serre spectral sequence, filtered}, and so it degenerates at the second page. Moreover, the complex 
$$C_s(B_\bullet;M) = \left(H_s(B^{(s)},B^{(s-1)};M), d_1\right), $$
where $d_1$ is the boundary operator $\delta$ of the triple $(B^{(s)}_\bullet, B^{(s-1)}_\bullet, B^{(s-2)}_\bullet)$ (see the sequence~\eqref{equation: les, homology triple}) computes the homology groups $H_s(B_\bullet;M)$.
\end{exmp}

Next, we will compute the second page $E^2_{s,t}$ of the spectral sequence in Theorem~\ref{theorem: serre spectral sequence, filtered} provided the fibration $\pi\colon E_\bullet \to B_\bullet$ is a principal fibration (Section~\ref{section: principal fibration}) and the base $B_\bullet$ is reduced. We will follow the classical approach of~\cite[Chapter~9.2]{Spanier94} and~\cite[Chapter~15]{Switzer}.

We begin with defining analogs of a disk and a sphere in the category $\srLie$. Let $W_*\in \Vect^{gr}_{\kk}$ be a graded vector space concentrated in only one degree $s\geq 1$; that is $W_s=W$ and $W_*=0$ if $*\neq s$. Consider a chain complex $C(W_*)_\bullet$ given as follows
$$C(W_*)_k= \left\{
\begin{array}{ll}
W & \mbox{if $k=s,s-1$,}\\
0, & \mbox{otherwise,}
\end{array}
\right.$$
endowed with the only one non-trivial differential $\partial_s=\id\colon C(W_*)_{s}\to C(W_*)_{s-1}$. Note that $W_*$ can be considered itself as a chain complex; then $C(W_*)$ is the cone of the identity morphism $\id\colon W_*[-1]\to W_*[-1]$. We write $D(W_*)\in \srLie$ (resp. $\partial D(W_*) \in \srLie$) for the free simplicial restricted Lie algebra $\free(\Gamma C(W_*))$ (resp. $\free(\Gamma W_*[-1])$).

\begin{rmk}\label{remark: relative homotopy groups as homotopy classes of maps} We list several properties of $D(W_*)$ and $\partial D(W_*)$, $W_*\in \Vect_{\kk}^{gr}$. Both $D(W_*)$ and $\partial D(W_*)$ are cofibrant objects in $\srLie$. The unique map $0\to D(W_*)$ is a weak equivalence, and the simplicial restricted Lie algebra $D(W_*)$ is contractible. The canonical inclusion $\partial D(W_*)\subset D(W_*)$ is a cofibration and the quotient Lie algebra $D(W_*)/\partial D(W_*)$ is isomorphic to $\partial D(W_*[1])$. Since any $L_\bullet \in \srLie$ is a fibrant object, one has
$$\pi_s(L_\bullet) \cong [\partial D(\kk[s+1]), L_\bullet], \; s\geq 0, $$
where $[\partial D(\kk[s+1]), L_\bullet]$ is the set of homotopy classes of maps in $\srLie$. Finally, for a pair $(L_\bullet, A_\bullet) \in \srLie$, one has
$$\pi_s(L_\bullet,A_\bullet)\cong [(D(\kk[s]),\partial D(\kk[s])),(L_\bullet,A_\bullet)], \; s\geq 1,$$
where the left hand side is the relative homotopy group (Definition~\ref{definition: relative homotopy and homology}) and the right hand side is the set of homotopy classes of maps of pairs in $\srLie$.
\end{rmk}

Let $\pi\colon E_\bullet \to B_\bullet$ be a principal fibration with the fiber $M_\bullet\in \sMod_{\kk\{\xi\}}$ and with the almost-free base $B_\bullet = \free(V_\bullet)$, $V_\bullet \in \tilde{s}\Vect_{\kk}$. Suppose $$\alpha \colon (D(\kk[s-1]),\partial D(\kk[s-1]))\to (B_\bullet^{(s)},B_\bullet^{(s-1)}), \; s\geq 2 $$ is a map of pairs. The simplicial restricted Lie algebra $D(\kk[s-1])$ is contractible, so by Lemma~\ref{lemma: principal fibrations over contractible},
there is an isomorphism of principal fibrations
$$\theta_\alpha\colon \trivxi M_\bullet \times D(\kk[s-1]) \xrightarrow{\cong} E_\bullet \vert_{\alpha}.$$ 
Then we have the composite $\kappa_{\alpha}$  given by 
\begin{align*}
\kappa_\alpha\colon H_t(\trivxi M_\bullet;\kk) &\xrightarrow{\sigma} H_{t+s}(\trivxi M_\bullet \times D(\kk[s-1]),\trivxi M_\bullet\times \partial D(\kk[s-1]);\kk) \\
&\xrightarrow{\theta_{\alpha *}} H_{t+s}(E_\bullet\vert_{\alpha}, E_\bullet\vert_{\partial \alpha};\kk) \\
&\xrightarrow{\bar{\alpha}_*} H_{t+s}(E_\bullet^{(s)},E_\bullet^{(s-1)};\kk),
\end{align*}
where $\sigma$ is given by $\sigma(x) = x\times \kappa_s$, $$\kappa_s \in H_s(D(\kk[s-1]),\partial D(\kk[s-1]);\kk) = \widetilde{H}_s(\partial D(\kk[s]);\kk)$$ is the canonical generator; here, we use the K\"{u}nneth isomorphism (Corollary~\ref{corollary: kunneth formula}). Since any two isomorphisms of a principal fibration over a contractible object in $\srLie$ are homotopic to each other, the map $\kappa_\alpha$ does not depend on the choice of the trivialization $\theta_\alpha$. Similarly, suppose 
$$\alpha'\colon (D(\kk[s-1]),\partial D(\kk[s-1]))\to (B_\bullet^{(s)},B_\bullet^{(s-1)}) $$
is another map of pairs homotopic to $\alpha$. By Lemma~\ref{lemma: principal fibrations are homotopy invariant}, there is an isomorphism $\theta\colon E_\bullet\vert_{\alpha} \cong E_\bullet\vert_{\alpha'}$ such that $\bar \alpha'\circ \theta =\bar\alpha$. Therefore, if 
$$\theta_{\alpha'}\colon \trivxi M_\bullet \times D(\kk[s-1]) \xrightarrow{\cong} E_\bullet \vert_{\alpha'}$$
is a trivialization, then there exists an isomorphism:
$$\phi\colon \trivxi M_\bullet \times D(\kk[s-1]) \xrightarrow{\cong} \trivxi M_{\bullet} \times D(\kk[s-1])$$
such that $\theta \circ \theta_\alpha = \theta_{\alpha'}\circ \phi$. Then we have 
$$\bar \alpha_* \circ \theta_{\alpha *} \circ \sigma = \bar \alpha'_* \circ \theta_* \circ \theta_{\alpha *} \circ \sigma = \bar \alpha'_* \circ \theta_{\alpha' *} \circ \phi_* \circ \sigma. $$
Since $\phi$ is an isomorphism of a principal fibration over a contractible object, $\phi_* = \id$, and so $\kappa_{\alpha}=\kappa_{\alpha'}$. In this way, the following map is well-defined
\begin{align}\label{equation: SSS comparasing}
\kappa\colon H_t(\trivxi M_\bullet;\kk) \times \pi_{s-1}(B_\bullet^{(s)},B_\bullet^{(s-1)}) &\to H_{t+s}(E_\bullet^{(s)},E_\bullet^{(s-1)};\kk), \\
(x,[\alpha])&\mapsto \kappa_{\alpha}(x). \nonumber
\end{align}

\begin{lmm}\label{lemma:kappa is bilinear}
$\kappa$ is bilinear.
\end{lmm}

\begin{proof}
Compare the next proof with~\cite[Lemma~15.23]{Switzer}. Clearly, the mapping $\kappa$ is linear in the first variable. The sum in $$\pi_{s-1}(B^{(s)}_{\bullet}, B^{(s-1)_\bullet})=[(D(\kk[s-1]),\partial D(\kk[s-1])), (B^{(s)}_{\bullet}, B^{(s-1)_\bullet})]$$
is defined by using the coproduct $$\mu\colon D(\kk[s-1]) \to D((\kk\oplus \kk)[s-1])\cong D(\kk[s-1])\sqcup D(\kk[s-1]),$$
induced by the diagonal $\kk \to \kk\oplus \kk$. Namely, if $[\alpha],[\beta] \in \pi_{s-1}(B^{(s)}_{\bullet}, B^{(s-1)_\bullet})$, then $[\alpha]+[\beta]$ is the class $[\nabla \circ (\alpha\sqcup \beta) \circ \mu]$, where $\nabla\colon B^{(s)}_\bullet \sqcup B^{(s)}_\bullet \to B^{(s)}_\bullet$ is the codiagonal. Note that, the pullback $E_\bullet \vert_{\nabla \circ (\alpha \sqcup \beta)}$ can be identified with the pushout (Proposition~\ref{proposition: pushouts and products in rlie}):
$$E_\bullet \vert_{\alpha} \sqcup_{\trivxi M_\bullet} E_\bullet \vert_{\beta},$$
where $\trivxi M_\bullet \hookrightarrow E_\bullet \vert_{\alpha}$ and $\trivxi M_\bullet \hookrightarrow E_\bullet \vert_{\beta}$ are embeddings of the fibers. So, if $$\theta_\alpha \colon \trivxi M_\bullet \times D(\kk[s-1]) \xrightarrow{\cong} E_\bullet\vert_{\alpha}$$ and
$$\theta_\beta \colon \trivxi M_\bullet \times D(\kk[s-1]) \xrightarrow{\cong} E_\bullet\vert_{\beta} $$
are trivializations, we obtain a trivialization
$$\theta_{\alpha}\sqcup \theta_{\beta} \colon \trivxi M_\bullet \times D((\kk\oplus\kk)[s-1]) \to E_\bullet \vert_{\nabla\circ (\alpha \sqcup \beta)} $$
by ``gluing'' together $\theta_{\alpha}$ and $\theta_{\beta}$. The pullback of $\theta_{\alpha}\sqcup \theta_{\beta}$ along $\mu$ is an isomorphism $$\theta_{\alpha+\beta}\colon \trivxi M_{\bullet} \times D(\kk[s-1]) \xrightarrow{\cong} E_{\bullet}\vert_{\alpha+\beta}$$
such that $\bar \mu\circ \theta_{\alpha+\beta} = (\theta_{\alpha}\sqcup \theta_\beta)\circ (\id\times \mu)$.

Since
$$H_*(E_\bullet \vert_{\nabla\circ (\alpha \sqcup \beta)},E_\bullet \vert_{\partial\nabla\circ (\alpha \sqcup \beta)};\kk) \cong H_*(E_\bullet\vert_{\alpha},E_\bullet\vert_{\partial\alpha};\kk)\oplus H_*(E_\bullet\vert_{\beta},E_\bullet\vert_{\partial\beta};\kk), $$
we calculate
\begin{align*}
\kappa_{\alpha+\beta}(x) &= \overline{\nabla}_*\circ (\overline{\alpha \sqcup \beta})_*\circ \bar{\mu}_* \circ \theta_{(\alpha+\beta)*} \circ \sigma(x) \\
&= \overline{\nabla}_* \circ (\bar\alpha \sqcup \bar\beta)_* \circ (\theta_{\alpha} \sqcup \theta_{\beta})_* \circ (\id \times \mu)_* \circ \sigma(x) \\
&= \bar\alpha_*\circ \theta_{\alpha *}\circ \sigma(x) + \bar\beta_*\circ \theta_{\beta *}\circ \sigma(x) = \kappa_{\alpha}(x)+\kappa_{\beta}(x)
\end{align*}
for all $x\in H_*(\trivxi M_{\bullet};\kk)$.
\end{proof}

Thus, $\kappa$ induces a natural homomorphism
$$\kappa\colon H_t(\trivxi M_\bullet;\kk)\otimes \pi_{s-1}(B_\bullet^{(s)},B^{(s-1)}_\bullet) \to H_{t+s}(E^{(s)}_{\bullet},E^{(s-1)}_{\bullet};\kk).$$

\begin{lmm}\label{lemma:kappa commutes with boundary}
The mapping $\kappa$ commutes with the boundary operators of triples $(B^{(s)}_\bullet,B^{(s-1)}_\bullet,B^{(s-2)}_\bullet)$ and $(E^{(s)}_\bullet,E^{(s-1)}_\bullet,E^{(s-2)}_\bullet)$ in the sense that the following diagram commutes 
$$
\begin{tikzcd}
H_t(\trivxi M_\bullet;\kk)\otimes  \pi_{s-1}(B_\bullet^{(s)},B^{(s-1)}_\bullet) \arrow{r}{\kappa} \arrow{d}{\id \otimes \delta}
& H_{s+t}(E_\bullet^{(s)},E^{(s-1)}_\bullet;\kk) \arrow{d}{\delta} \\
H_t(\trivxi M_\bullet;\kk)\otimes  \pi_{s-2}(B_\bullet^{(s-1)},B^{(s-2)}_\bullet) \arrow{r}{\kappa} 
& H_{s+t-1}(E_\bullet^{(s-1)},E^{(s-2)}_\bullet;\kk).
\end{tikzcd}
$$
\end{lmm}

\begin{proof}
Compare with~\cite[Lemma~15.24]{Switzer}. Let $x\in H_t(\trivxi M_\bullet;\kk)$ and $[\alpha] \in \pi_{s-1}(B_\bullet^{(s)},B_\bullet^{(s-1)})$. Then
\begin{align*}\delta\circ \kappa(x\otimes [\alpha]) &= j_*\circ \partial \circ \bar \alpha_*\circ \theta_{\alpha *} (x\times \kappa_s) \\
&=j_*\circ (\bar{\alpha}|E_\bullet\vert_{\partial \alpha})_* \circ \theta_{\partial \alpha *}(x\times \partial \kappa_{s}),
\end{align*}
where $\partial \colon H_{s+t}(E_\bullet^{(s)},E_{\bullet}^{(s-1)};\kk) \to \widetilde{H}_{s+t-1}(E^{(s-1)}_\bullet;\kk)$ is the boundary homomorphism, $j\colon (E^{(s-1)}_\bullet,0) \to (E^{(s-1)}_\bullet, E^{(s-2)}_\bullet)$ is the evident map of pairs, and $\partial \kappa_s \in H_{s-1}(\partial D(\kk[s-1]);\kk)$ is a generator.

Let $\rho\colon (D(\kk[s-2]),\partial D(\kk[s-2])) \to (\partial D(\kk[s-1]),0)$ be the standard projection; then $\delta[\alpha]$ is the class of the composite
\begin{align*}
(D(\kk[s-2]),\partial D(\kk[s-2])) \xrightarrow{\rho} (\partial D(\kk[s-1]),0) &\xrightarrow{\alpha\vert_{\partial D(\kk[s-1])}}
(B^{(s-1)}_\bullet,0) \\  &\xrightarrow{j} (B^{(s-1)}_\bullet, B^{(s-2)}_\bullet).
\end{align*}
For convenience, we abbreviate this composite by $\gamma$. Then there is a trivialization
$$\theta_\gamma\colon \trivxi M_\bullet \times D(\kk[s-2])\xrightarrow{\cong} E_\bullet \vert_{\gamma}$$ such that $\bar\rho \circ \theta_{\gamma} = \theta_{\partial f}\circ (\id \times \rho)$. Finally, we have
\begin{align*}
\kappa\circ (\id\otimes \delta)(x\otimes [\alpha]) &= \bar{\gamma}_* \circ \theta_{\gamma *}(x\times \kappa_{s-1}) \\
&= j_*\circ (\bar{\alpha}|E_\bullet\vert_{\partial \alpha})_* \circ \bar{\rho}_* \circ \theta_{\gamma *}(x\times \kappa_{s-1}) \\
&=j_*\circ (\bar{\alpha}|E_\bullet\vert_{\partial \alpha})_* \circ  \theta_{\partial f *}\circ (\id \times \rho)_*(x\times \kappa_{s-1}) \\
&=j_*\circ (\bar{\alpha}|E_\bullet\vert_{\partial \alpha})_* \circ  \theta_{\partial f *}(x\times \partial\kappa_{s})\\
&=\delta\circ \kappa(x\otimes [\alpha]). \qedhere
\end{align*}
\end{proof}

\begin{lmm}\label{lemma: kappa over contractible}
Let $\pi\colon E_\bullet \to B_\bullet$ be a principal fibration with the fiber $M_\bullet \in \sMod_{\kk\{\xi\}}$. Suppose that $B_\bullet = D(W_*)$, $W_*\in \Vect^{gr}_{\kk}$, $W_{k} = 0 $ if $k\neq s-1$, and $s\geq 2$. Then
\begin{enumerate}
\item $\kappa(x,\xi_*[\alpha])=0$ for all $x\in H_t(\trivxi M_\bullet;\kk)$ and $[\alpha]\in \pi_{s-1}(B^{(s)}_\bullet, B^{(s-1)}_\bullet)$;
\item The induced map
$$\kappa\colon H_t(\trivxi M_\bullet;\kk)\otimes \pi_{s-1}(B_\bullet^{(s)},B^{(s-1)}_\bullet)/\xi \to H_{t+s}(E^{(s)}_{\bullet},E^{(s-1)}_{\bullet};\kk)$$
is an isomorphism.
\end{enumerate}
\end{lmm}

\begin{proof}
Since the base $B_\bullet$ is contractible, we can assume that $\pi$ is a trivial principal fibration (Lemma~\ref{lemma: principal fibrations over contractible}). In this case, the diagram
$$
\begin{tikzcd}[column sep=small]
H_t(\trivxi M_\bullet;\kk)\otimes  \pi_{s-1}(B_\bullet^{(s)},B^{(s-1)}_\bullet) \arrow{r}{\kappa} \arrow{d}{\cong}
& H_{s+t}(E_\bullet^{(s)},E^{(s-1)}_\bullet;\kk) \\
H_t(\trivxi M_\bullet;\kk)\otimes  H_{s}(B_\bullet^{(s)},B^{(s-1)}_\bullet;\kk\{\xi\}) \arrow{r} 
& H_t(\trivxi M_\bullet;\kk)\otimes  H_{s}(B_\bullet^{(s)},B^{(s-1)}_\bullet;\kk) \arrow{u}{\cong}.
\end{tikzcd}
$$
commutes, where the left vertical arrow is the Hurewicz isomorphism and the right one is the K\"{u}nneth isomorphism. Since 
$$H_s(B_\bullet^{(s)},B_\bullet^{(s-1)};\kk\{\xi\})/\xi \cong H_s(B_\bullet^{(s)},B_\bullet^{(s-1)};\kk),$$
the lemma follows.
\end{proof}

Let $B_\bullet = \free(V_\bullet)$ be an almost-free simplicial restricted Lie algebra, $V_\bullet \in \tilde{s}\Vect_{\kk}$. Then the natural inclusion of normalized chains
$$NV_{s-1} \hookrightarrow V_{s-1}\hookrightarrow \free(V_{s-1})$$ induce the following map of pairs
$$c_s\colon (D(NV_{s-1}), \partial D(NV_{s-1})) \to (B^{(s)}_{\bullet}, B^{(s-1)}_\bullet), $$
where we consider $NV_{s-1}$ as a graded vector space concentrated in degree $s-1$. The next lemma follows immediately from the relative Hurewicz Theorem (Corollary~\ref{corollary: relative hurewicz theorem}).

\begin{lmm}\label{lemma: relative homotopy groups of skeleta}
For $B_\bullet = \free(V_\bullet)$ as above, suppose that $B_\bullet$ is reduced. Then the induced map 
$$c_{s*}\colon \pi_{s-1}(D(NV_{s-1}), \partial D(NV_{s-1})) \to \pi_{s-1}(B^{(s)}_{\bullet}, B^{(s-1)}_\bullet)$$
is an isomorphism. Moreover, 
\begin{equation*}
\pi_{s-1}(D(NV_{s-1}), \partial D(NV_{s-1}))\cong \kk\{\xi\}\otimes_{\kk}NV_{s-1}. \eqno\qed
\end{equation*}
\end{lmm}

\begin{prop}\label{proposition: second page of SSS is local}
Let $\pi\colon E_\bullet \to B_\bullet$ be a principal fibration with the fiber $M_\bullet\in \sMod_{\kk\{\xi\}}$ and the base $B_\bullet = \free(V_\bullet)$, $V_\bullet \in \tilde{s}\Vect_{\kk}$. Then the induced map
$$\bar{c}_{s*}\colon H_*(E_\bullet\vert_{c_s},  E_\bullet\vert_{\partial c_s};\kk) \to H_*(E^{(s)}_\bullet, E^{(s-1)}_\bullet;\kk) $$
is an isomorphism for any $s\geq 0$.
\end{prop}

\begin{proof}
Let $(X_\bullet, A_\bullet)$ be a pair in $\srLie$. %Since the functor $\widetilde{C}(-;\kk)\colon \srLie \to D(\Vect_{\kk})$ commutes with sifted colimits, 
There is a strongly convergent functorial spectral sequence
$$E^1_{n,m}=H_n(L_m,A_m;\kk) \Rightarrow H_{n+m}(L_\bullet,A_\bullet;\kk), $$
where $H_*(L_m,A_m;\kk)$ are the homology groups of the \emph{constant} simplicial restricted Lie algebra pair $(L_m,A_m)$. 

\iffalse
If $(L_\bullet, A_\bullet) = (E^{(s)}_\bullet, E^{(s-1)}_\bullet)$, %(resp. $E_\bullet\vert_{D(NV_{s-1})},  E_\bullet\vert_{\partial D(NV_{s-1})}$)
then by the K\"{u}nneth isomorphism and the definition of a principal fibration, one has 
\begin{align*}
H_n(E^{(s)}_m,E^{(s-1)}_m;\kk)&\cong H_n(\trivxi M_m \times B^{(s)}_m, \trivxi M_m \times B^{(s-1)}_m ;\kk)\\ &\cong H_{n-1}(\trivxi M_m;\kk) \otimes H_1(B_m^{(s)},B_m^{(s-1)};\kk).
\end{align*}
\fi
If $(L_\bullet, A_\bullet) = (E^{(s)}_\bullet, E^{(s-1)}_\bullet)$, %(resp. $E_\bullet\vert_{D(NV_{s-1})},  E_\bullet\vert_{\partial D(NV_{s-1})}$)
then by the definition of a principal fibration, we have
\begin{equation*}
H_n(E^{(s)}_m,E^{(s-1)}_m;\kk)\cong H_n(\trivxi M_m \times B^{(s)}_m, \trivxi M_m \times B^{(s-1)}_m ;\kk).
\end{equation*}
Since $B_\bullet$ is almost-free, the relative homology $H_i(B^{(s)}_m, B^{(s-1)}_m ;\kk)$ are concentrated only in degree $i=1$. Hence, by the K\"{u}nneth isomorphism, we have
\begin{equation*}
H_n(E^{(s)}_m,E^{(s-1)}_m;\kk)\cong H_{n-1}(\trivxi M_m;\kk) \otimes H_1(B_m^{(s)},B_m^{(s-1)};\kk).
\end{equation*}

In a similar way, we obtain that $H_n(E_m\vert_{c_s},  E_m\vert_{\partial c_s};\kk)$ is isomorphic to $$H_{n-1}(\trivxi M_m;\kk) \otimes H_1(D_m(NV_{s-1}),\partial D_m(NV_{s-1});\kk).$$
Since the map $c_s$ induces the isomorphism
$$c_{s*}\colon H_1(D_m(NV_{s-1}),\partial D_m(NV_{s-1});\kk) \xrightarrow{\cong} H_1(B_m^{(s)},B_m^{(s-1)};\kk), \; m,s\geq 0, $$
the proposition follows.
%and the differential $d_1=\partial_1\otimes \partial_2$ is the tensor product of $$\partial_1\colon H_{n-1}(\trivxi M_m;\kk) \to H_{n-1}(\trivxi M_{m-1};\kk),$$ which induced by the face operators $d_i\colon \trivxi M_m \to \trivxi M_{m-1}$, and $$ \partial_2 \colon \Gamma_m(NV_{s-1}) \to \Gamma_{m-1}(NV_{s-1}),$$
%which again is the alternating sum of face operators. 
\end{proof}

\begin{prop}\label{proposition: kappa is isomorphism}
Let $\pi\colon E_\bullet \to B_\bullet$ be a principal fibration with the fiber $M_\bullet \in \sMod_{\kk\{\xi\}}$. Suppose that the base $B_\bullet = \free(V_\bullet), V_\bullet\in \tilde{s}\Vect_{\kk}$ is almost-free, and $V_0=0$. Then
\begin{enumerate}
\item $\kappa(x,\xi_*[\alpha])=0$ for all $x\in H_t(\trivxi M_\bullet;\kk)$ and $[\alpha]\in \pi_{s-1}(B^{(s)}_\bullet, B^{(s-1)}_\bullet)$;
\item The induced map
$$\kappa\colon H_t(\trivxi M_\bullet;\kk)\otimes \pi_{s-1}(B_\bullet^{(s)},B^{(s-1)}_\bullet)/\xi \to H_{t+s}(E^{(s)}_{\bullet},E^{(s-1)}_{\bullet};\kk)$$
is an isomorphism for all $t,s\geq 0$.
\end{enumerate}
\end{prop}

\begin{proof}
In the commutative diagram
$$
\begin{tikzcd}[column sep=small]
H_t(\trivxi M_\bullet;\kk)\otimes  \pi_{s-1}(D(NV_{s-1}),\partial D(NV_{s-1})) \arrow{r}{\kappa} \arrow{d}{c_{s *}}
& H_{s+t}(E_\bullet^{(s)}\vert_{c_s},E^{(s-1)}_\bullet\vert_{\partial c_s};\kk) \arrow{d}{\bar{c}_{s *}}\\
H_t(\trivxi M_\bullet;\kk)\otimes  \pi_{s-1}(B_\bullet^{(s)},B^{(s-1)}_\bullet) \arrow{r}{\kappa}
& H_{s+t}(E_\bullet^{(s)},E^{(s-1)}_\bullet;\kk)
\end{tikzcd}
$$
vertical arrows are isomorphisms by Lemma~\ref{lemma: relative homotopy groups of skeleta} and Proposition~\ref{proposition: second page of SSS is local}. Thus Lemma~\ref{lemma: kappa over contractible} implies the proposition.
\end{proof}

\begin{thm}[Serre spectral sequence]\label{theorem: SSS}
Let $\pi\colon E_\bullet \to B_\bullet$ be a principal fibration with the fiber $M_\bullet \in \sMod_{\kk\{\xi\}}$. Suppose that the base $B_\bullet=\free(V_\bullet)$, $V_\bullet\in \tilde{s}\Vect_{\kk}$ is almost-free, and $V_0=0$. Then there is a strongly convergent spectral sequence with 
$$E^2_{s,t}\cong H_t(\trivxi M_\bullet;\kk)\otimes H_s(B_\bullet;\kk)$$
and $E^{\infty}$ is the bigraded module associated to the filtration of $H_*(E_\bullet;\kk)$ defined by
$$F_sH_*(E_\bullet;\kk) = \im\left(H_*(E^{(s)}_\bullet;\kk) \to H_*(E_\bullet; \kk)\right).$$
\end{thm}

\begin{proof}
In the spectral sequence of Theorem~\ref{theorem: serre spectral sequence, filtered}, we have $$E^1_{s,t}\cong H_{s+t}(E^{(s)}_\bullet, E^{(s-1)}_\bullet;\kk);$$ the differential $d_1$ is the boundary operator $\delta$ of the triple $(E^{(s)}_\bullet, E^{(s-1)}_\bullet, E^{(s-2)}_\bullet)$, see the sequence~\eqref{equation: les, homology triple}. By Proposition~\ref{proposition: kappa is isomorphism} and Lemma~\ref{lemma:kappa commutes with boundary}, the map $\kappa$
$$\kappa\colon (H_{t}(\trivxi M_\bullet;\kk)\otimes H_s(B_\bullet^{(s)},B_\bullet^{(s-1)};\kk),\id\otimes \delta) \xrightarrow{\cong} (H_{s+t}(E^{(s)}_\bullet, E^{(s-1)}_\bullet;\kk),\delta). $$
is an isomorphism of chain complexes.
Therefore, by Example~\ref{example: serre, idenity morphism}, we obtain that
\begin{equation*}
E^2_{s,t}\cong H_t(\trivxi M_\bullet;\kk)\otimes H_s(B_\bullet;\kk). \qedhere
\end{equation*}
\end{proof}

By combining Theorem~\ref{theorem: SSS} with Corollary~\ref{corollary: fibration with KM is principal}, we obtain the following statement.

\begin{cor}\label{corollary: SSS fiber is EM}
Let $\pi\colon E_\bullet \to B_\bullet$ be a fibration in $\srLie$ such that the total space $E_\bullet$ and the base $B_\bullet$ are connected, and the fiber $\fib(\pi)$ is an Eilenberg-MacLane Lie algebra $K(M,n)$, $M\in \Mod_{\kk\{\xi\}}$, $n\geq 0$. Then there is a functorial convergent spectral sequence
$$E^2_{s,t}=H_t(K(M,n);\kk)\otimes H_s(B_\bullet;\kk) \Rightarrow H_{s+t}(E_\bullet;\kk). \eqno\qed$$
\end{cor}

\begin{rmk}\label{remark: any fibration}
It seems likely that the spectral sequence of Corollary~\ref{corollary: SSS fiber is EM} exists for \emph{any} fibration over a connected base. All steps in the proof can be generalized that much, except Proposition~\ref{proposition: second page of SSS is local}. At the time of writing, we do not know how to extend this proposition on any larger class of fibrations. 

We also point out that Proposition~\ref{proposition: second page of SSS is local} can be viewed as a consequence of Mather's second cube theorem in the category $\srLie$, see~\cite[Definition~1.4]{Doeraene93} and~\cite{DH21}. At the time of writing, we are not aware if the cube theorem holds for $\srLie$.
\end{rmk}

\section{\texorpdfstring{$\kk$}{F}-complete simplicial restricted Lie algebras} \label{section: F-complete}

In this section we prove Theorem~\ref{theorem: intro, D} from the introduction. In Section~\ref{section: xi-complete modules} we discuss basic properties of the ring $\kk\{\xi\}$, define \emph{$\xi$-adic completion} (Definition~\ref{definition: xi-completion}), and prove the Artin-Rees property for the ideal $(\xi)$ in $\kk\{\xi\}$ (Proposition~\ref{proposition: ar property for xi}). The latter implies the exactness property of the $\xi$-adic completion for finitely-generated modules, see Proposition~\ref{proposition: xi-completion is exact}. For non-finitely generated $\kk\{\xi\}$-modules, the $\xi$-completion is not exact, and we define its \emph{left derived functors} $L_0$ and $L_1$ in Section~\ref{section: derived xi-complete modules}. We define \emph{derived $\xi$-complete modules} in Definition~\ref{definition: derived xi-completion} and show in Proposition~\ref{proposition: derived modules, summary} that derived $\xi$-complete modules form a weak Serre subcategory in $\Mod_{\kk\{\xi\}}$. After that, we completely set up to prove Theorem~\ref{theorem: intro, D}.

In Section~\ref{section: F-completion} we define both \emph{$\kk$-complete} objects in $\srLie$ and the \emph{$\kk$-completion} functor $L_\xi$ (Definition~\ref{definition: F-completion}). We prove Theorem~\ref{theorem: intro, D} by induction along the Postnikov tower. In Corollary~\ref{corollary: EM is F-complete iff derived completed}, we show that Theorem~\ref{theorem: intro, D} holds for Eilenberg-MacLane Lie algebras $K(M,n)$. After that, we heavily use the Serre spectral sequence (Corollary~\ref{corollary: SSS fiber is EM}) to prove Theorem~\ref{theorem: intro, D} as Corollary~\ref{corollary: F-complete are derived Xi-complete}. Furthermore, in Corollary~\ref{corollary: F-completion, homotopy groups}, we describe the homotopy groups of the $\kk$-completion $L_\xi L_\bullet$ in terms of $\pi_*(L_\bullet)$ and derived functors $L_0$ and $L_1$.

\subsection{\texorpdfstring{$\xi$}{Xi}-complete modules}\label{section: xi-complete modules} 
We begin with a few algebraic preliminaries. Recall from Definition~\ref{definition: twisted polynomial ring} that $\kk\{\xi\}$ is the ring of twisted polynomials. If $\kk\neq \F_p$, then $\kk\{\xi\}$ is a non-commutative ring, however it still shares a lot of common properties with the usual polynomial ring $\kk[t]$. First of all, we note that the ring $\kk\{\xi\}$ has no zero divisors.

We say that a twisted polynomial 
$$f(\xi) = a_0 +a_1 \xi +a_2 \xi^2 +\ldots + a_n \xi^n \in \kk\{\xi\}$$
has \emph{degree $n$} if the leading coefficient $a_n\neq 0 $. We will denote by $\deg(f)$ the degree of $f$. The next lemma shows that the function $\deg\colon \kk\{\xi\}\setminus 0 \to \N$ can be used for left and right divisions with a remainder. The proof is straightforward and we leave it to the reader.

\begin{lmm}\label{lemma: twisted polynomials are eucledian domain}
If $f$ and $g$ are in $\kk\{\xi\}$ and $g$ is nonzero, then there are $q,r \in \kk\{\xi\}$ such that $f = qg + r$ and either $r = 0$ or $\deg(r) < \deg(g)$. Similarly, there are $q'$ and $r'$ such that $f=gq'+r'$ and either $r'=0$ or $\deg(r')<\deg(g)$. \qed
\end{lmm}

\begin{cor}\label{corollary: twisted are PID}
Any left (resp. right) ideal in $\kk\{\xi\}$ is principal. In particular, the ring $\kk\{\xi\}$ is left (resp. right) Noetherian.\qed
\end{cor}

\begin{cor}\label{corollary: submodule of free over twisted}
Let $M \in \Mod_{\kk\{\xi\}}$ (resp. $M\in \Mod^{\kk\{\xi\}}$) be a free module. Then any submodule $N$ of $M$ is also free. \qed
\end{cor}

\begin{cor}\label{corollary: module over twisted}
Let $M \in \Mod_{\kk\{\xi\}}$ (resp. $M\in \Mod^{\kk\{\xi\}}$) be any left (resp. right) module over $\kk\{\xi\}$. Then there is a short exact sequence
$$0 \to F_1 \to F_0 \to M \to 0, $$
where $F_0,F_1$ are free modules. \qed
\end{cor}

\begin{cor}\label{corollary: torsion-free is a colimit of free}
Any torsion-free left (resp. right) $\kk\{\xi\}$-module is a filtered colimit of finitely generated free submodules. \qed
\end{cor}

Consider now the \emph{left} ideal $(\xi) \subset \kk\{\xi\}$ generated by the element $\xi$. Since the field $\kk$ is perfect, the ideal $(\xi)$ is actually two-sided and coincides with the \emph{right} ideal of $\kk\{\xi\}$ generated by the same element $\xi$. Recall from~\cite[Section~13]{GW04} that an ideal $I$ in a ring $R$ has the \emph{left Artin-Rees property} if, for every left ideal $K\subset R$, there is a positive integer $n$ such that $K\cap I^n = IK$.

\begin{prop}\label{proposition: ar property for xi}
The ideal $I=(\xi)\subset \kk\{\xi\}$ has the left Artin-Rees property. 
\end{prop}

\begin{proof}
Recall that the \emph{Rees ring} $\mathcal{R}(\xi)$ is the subring of the polynomial ring $\kk\{\xi\}[x]$ generated by $\kk\{\xi\}+ Ix$, that is, 
$$\mathcal{R}(\xi) = \kk\{\xi\} + Ix+I^2x^2+\ldots I^jx^j+\ldots $$
Note that the ring $\mathcal{R}(\xi)$ is generated by $\kk\{\xi\}$ together with the element $y=\xi x$. Therefore $\mathcal{R}(\xi)$ is a quotient of the twisted polynomial ring $\kk\{\xi,y\}$, and so it is left Noetherian by~\cite[Theorem~1.14]{GW04}. Finally, this implies that the ideal $(\xi)$ has the left Artin-Rees property by~\cite[Lemma~13.2]{GW04}.
\end{proof}

Let $M\in \Mod_{\kk\{\xi\}}$ be a left module over the ring $\kk\{\xi\}$ of twisted polynomials. Since the field $\kk$ is perfect, the subset 
$$\xi^r(M)=\{\xi^r\cdot m \; |\; m\in M\} \subset M, \;\; r\geq 0 $$
is a \emph{submodule} of $M$. Similar to Definition~\ref{definition: frobenius twist}, we define the \emph{$r$-th Frobenius twist} $M^{(r)}\in \Mod_{\kk\{\xi\}}$ of $M$ as follows. As an abelian group, $M^{(r)}=M$ and we endow it with a new $\kk\{\xi\}$-action 
$$-\cdot-\colon \kk\{\xi\}\times M^{(r)} \to M^{(r)} $$
given by formulas: $\xi\cdot m = \xi m$, $a\cdot m=\varphi^{-r}(a)m$, where $a\in \kk$, $m\in M^{(r)}=M$, and $\varphi\colon \kk \to \kk$ is the Frobenius automorphism, $\varphi(a)=a^p$. Then a map 
\begin{equation}\label{equation: multiplication by xir}
\xi^r\colon M^{(r)}\to M, \;\; m\mapsto \xi^rm
\end{equation}
is a homomorphism of $\kk\{\xi\}$-modules and $\xi^r(M)=\im(\xi^r)$. We set $M_r=\ker(\xi^r)$; as an abelian group, $M_r$ consists of all $m\in M$ such that $\xi^rm=0$.  

%There are natural inclusions 
%$$\xi^{r+1}(M) \hookrightarrow \xi^r(M), \;\; r\geq 0 $$
%and projections
%$$M/\xi^{r+1}(M) \to M/\xi^r(M), \;\; r\geq 0. $$

\begin{dfn}\label{definition: xi-completion}
Define the \emph{$\xi$-adic completion} (or \emph{$\xi$-completion}) of a left $\kk\{\xi\}$-module $M$ to be
$$\widehat{M}=\lim_r M/\xi^r(M) \in \Mod_{\kk\{\xi\}}. $$
A left $\kk\{\xi\}$-module $M$ is \emph{$\xi$-adic complete} (or \emph{$\xi$-complete}) if the natural map $M \to \widehat{M}$ is an isomorphism.
\end{dfn}

We write $\kk\{\{\xi\}\}$ for the $\xi$-completion $\widehat{\kk\{\xi\}}$. Since ideals $(\xi)^r=(\xi^r)$ are two-sided, $\kk\{\{\xi\}\}$ is endowed with a (non-commutative) ring structure such that $\kk\{\xi\} \to \kk\{\{\xi\}\}$ is a ring homomorphism. We observe that the $\xi$-completion takes value in the category $\Mod_{\kk\{\{\xi\}\}}$ of left $\kk\{\{\xi\}\}$-modules. %It means that for any $M\in \Mod_{\kk\{\xi\}}$, its $\xi$-completion $\widehat{M}$ is a left $\kk\{\{\xi\}\}$-module, and for any homomorphism $f\colon N\to M$ in $\Mod_{\kk\{\xi\}}$, its $\xi$-completion $\hat{f}$ is a homomorphism of left $\kk\{\{\xi\}\}$-modules. 
We note that $\kk\{\{\xi\}\}$ is a \emph{torsion-free} left $\kk\{\xi\}$-module.

The Artin-Rees property of the ideal $(\xi)$ implies the following exactness statement.

\begin{prop}\label{proposition: xi-completion is exact}
Suppose that $$0\to M' \to M \to M'' \to 0$$ is an exact sequence of finitely generated left $\kk\{\xi\}$-modules. Then the completed sequence $$0\to \widehat{M}' \to \widehat{M} \to \widehat{M}'' \to 0$$ is again exact.
\end{prop}

\begin{proof} By~\cite[Lemma~13.1(a)]{GW04}, the $\xi$-adic topology on $M'$ is induced from the $\xi$-adic topology on $M$. This implies the proposition, cf. e.g.~\cite[Proposition~10.12]{AM69}.
\end{proof}

\subsection{Derived \texorpdfstring{$\xi$}{Xi}-complete modules}\label{section: derived xi-complete modules} 
Here we define left derived functors of the $\xi$-adic completion functor. By Corollary~\ref{corollary: module over twisted}, for any left $\kk\{\xi\}$-module $M$ there is a free resolution
$$0 \to F_1 \to F_0 \to M \to 0. $$
We define the \emph{left derived functors} of the $\xi$-completion by formulas
$$L_0(M) = \coker(\widehat{F}_1 \to \widehat{F}_0), \;\; L_1(M) = \ker(\widehat{F}_1 \to \widehat{F}_0). $$
These groups are independent on the choice of the resolution and they are functorial in $M$, as one checks by comparing resolutions. The commutative square 
$$
\begin{tikzcd}
F_1 \arrow[swap]{d} \arrow{r}
& F_0 \arrow{d} \\
\widehat{F}_1 \arrow{r}
& \widehat{F}_0
\end{tikzcd}
$$
induces a natural map 
\begin{equation}\label{equation: derived completion map}
\phi_M\colon M \to L_0(M).
\end{equation}

The next lemma follows immediately from the results of the previous section.
\begin{lmm}\label{lemma: derived completion coincides with usual}
The functors $L_0$ and $L_1$ take values in left $\kk\{\{\xi\}\}$-modules. If $M$ is either a finitely generated or a free left $\kk\{\xi\}$-module, then $L_0M = \widehat{M}$, $L_1M=0$, and $\phi_M\colon M \to L_0M$ coincides with the $\xi$-adic completion. \qed
\end{lmm}

\begin{dfn}\label{definition: derived xi-completion}
%We say that the derived $\xi$-adic completion of $M\in \Mod_{\kk\{\xi\}}$ is defined if $L_1M=0$, and 
Let $M\in \Mod_{\kk\{\xi\}}$ be a left $\kk\{\xi\}$-module. We define the \emph{derived $\xi$-adic completion} (or derived $\xi$-completion) of $M$ to be the homomorphism $\phi_M\colon M \to L_0M$. A left $\kk\{\xi\}$-module $M$ is \emph{derived $\xi$-adic complete} (or derived $\xi$-complete) if $\phi_M\colon M \to L_0M$ is an isomorphism and $L_1M=0$.
\end{dfn}

\begin{rmk}\label{remark: M is xi-complete implies completion id defined}
As we will see in Proposition~\ref{proposition: everything is xi-complete}, if $M$ is $\xi$-complete, then $L_1M=0$.
\end{rmk}

\begin{lmm}\label{lemma: six-term exact sequence, completion}
For a short exact sequence of left $\kk\{\xi\}$-modules
$$0 \to M' \to M \to M'' \to 0, $$
there is a natural six term exact sequence of $\kk\{\{\xi\}\}$-modules
\begin{equation*}
0 \to L_1M' \to L_1M \to L_1M'' \to L_0 M' \to L_0 M \to L_0 M'' \to 0. \eqno\qed
\end{equation*}
\end{lmm}

Next, we will give an interpretation of the derived $\xi$-adic completion in terms of $\Hom$ and $\Ext$-functors. We begin with the following observation. Let $M\in \Mod_{\kk\{\xi\}}$ be a left $\kk\{\xi\}$-module and let $N$ be an $\kk\{\xi\}$-bimodule. Then the set of maps $\Hom_{\kk\{\xi\}}(N,M)$ is naturally a left $\kk\{\xi\}$-module. Indeed, if $f\in \Hom_{\kk\{\xi\}}(N,M)$ and $a\in \kk\{\xi\}$, then we set $af\colon N\to M$, $n\mapsto f(na)$, $n\in N$. By the same argument, the $\Ext$-group $\Ext_{\kk\{\xi\}}(N,M)$ is a left $\kk\{\xi\}$-module as well.

We denote $\kk\{\xi^{\pm}\}$ the ring of \emph{twisted Laurent polynomials}, i.e. $\kk\{\xi^{\pm}\}$ is defined as the set of Laurent polynomials in the variable $\xi$ and coefficients in $\kk$. It is endowed with a ring structure with the usual addition and with a non-commutative multiplication that can be summarized with relations: $$\xi a = \varphi(a)\xi = a^{p} \xi, \;\; \xi^{-1}a=\varphi^{-1}(a)\xi^{-1}, \; \; a\in \kk.$$
Here $\varphi\colon \kk\to \kk$ is the Frobenius automorphism. 

We denote by $\kk\{\xi\}/\xi^{\infty}$ the quotient $\kk\{\xi\}$-bimodule $\kk\{\xi^{\pm}\}/\kk\{\xi\}$. The bimodule $\kk\{\xi\}/\xi^{\infty}$ is the union of its subbimodules $$C_r\subset \kk\{\xi\}/\xi^{\infty}, \;\; r\geq 1,$$
where $C_r$ is generated by $\xi^{-r}\in \kk\{\xi\}/\xi^{\infty}$. Notice that there is a bimodule isomorphism $C_r\cong \kk\{\xi\}/\xi^{r}$. 

\begin{lmm}\label{lemma: Hom, Ext, 1}
Let $M\in \Mod_{\kk\{\xi\}}$ be a left $\kk\{\xi\}$-module. Then
\begin{enumerate}
\item $\Hom_{\kk\{\xi\}}(C_r,M)\cong M_r$, where $M_r$ is the kernel~\eqref{equation: multiplication by xir}.
\item $\Ext_{\kk\{\xi\}}(C_r,M)\cong M/\xi^r(M)=\coker(\xi^r)$.
\item $\Hom_{\kk\{\xi\}}(\kk\{\xi\}/\xi^{\infty}, M) \cong \lim_r \Hom_{\kk\{\xi\}}(C_r,M)$. 
\end{enumerate}
Moreover, there is a short exact sequence:
\begin{equation*}
0 \to {\lim}^1 \Hom_{\kk\{\xi\}}(C_r,M) \to \Ext_{\kk\{\xi\}}(\kk\{\xi\}/\xi^{\infty}, M) \to \widehat{M} \to 0. \eqno\qed
\end{equation*}
\end{lmm}

For convenience, we set $$\mathbb{E}_{\xi}M=\Ext_{\kk\{\xi\}}(\kk\{\xi\}/\xi^{\infty},M) \;\; \text{and} \;\; \mathbb{H}_\xi M = \Hom_{\kk\{\xi\}}(\kk\{\xi\}/\xi^{\infty},M),$$ 
where $M\in \Mod_\kk\{\xi\}$. Then Lemma~\ref{lemma: Hom, Ext, 1} implies that there is an exact sequences in $\Mod_{\kk\{\xi\}}$:
\begin{align}\label{equation: six term exact sequence H,E}
0 &\to \mathbb{H}_\xi M \to \prod_{r\geq 1} M_r \to \prod_{r\geq 1} M_r \\
&\to \mathbb{E}_\xi M \to \prod_{r\geq 1} M/\xi^rM \to \prod_{r\geq 1} M/\xi^rM \to 0. \nonumber
\end{align}

\begin{cor}\label{corollary: derived completion are hom and ext}
Let $M\in \Mod_{\kk\{\xi\}}$ be a left $\kk\{\xi\}$-module. Then there are natural isomorphisms
$$L_0M\cong \mathbb{E}_{\xi} M, \;\;\; L_1M\cong \mathbb{H}_\xi M. $$
\end{cor}

\begin{proof}
By Lemma~\ref{lemma: Hom, Ext, 1}, we have $\mathbb{H}_\xi F=0$ and $\mathbb{E}_\xi F \cong \widehat{F}$ for a free $\kk\{\xi\}$-module $F\in \Mod_{\kk\{\xi\}}$. Let $0 \to F_1 \to F_0 \to M\to 0$ be a free resolution of $M$. Then, in the commutative diagram
$$
\begin{tikzcd}
0 \arrow{r} 
& \mathbb{H}_\xi M \arrow{r} \arrow[dashed]{d} 
& \mathbb{E}_\xi F_1 \arrow{r} \arrow{d}{\cong}
& \mathbb{E}_\xi F_0 \arrow{r} \arrow{d}{\cong}
& \mathbb{E}_\xi M \arrow{r} \arrow[dashed]{d} 
& 0 \\
0 \arrow{r} 
& L_1M \arrow{r}
& \widehat{F}_1 \arrow{r}
& \widehat{F}_0 \arrow{r}
& L_0 M \arrow{r}
& 0,\end{tikzcd}
$$
both rows are exact sequence and two middle vertical arrows are isomorphisms. This implies the statement.
\end{proof}

\begin{rmk}\label{remark: completion map}
Under the isomorphism of Corollary~\ref{corollary: derived completion are hom and ext}, the derived completion map~\eqref{equation: derived completion map}
$$\phi_M\colon M \to L_0M$$ coincides with the boundary homomorphism
$$\delta_M\colon M=\Hom_{\kk\{\xi\}}(\kk\{\xi\},M) \to \Ext_{\kk\{\xi\}}(\kk\{\xi\}/\xi^\infty,M), $$
which is associated with the short exact sequence
\begin{equation}\label{equation: Cinfty}
0 \to \kk\{\xi\} \to \kk\{\xi^{\pm}\} \to \kk\{\xi\}/\xi^{\infty} \to 0. 
\end{equation}
\end{rmk}

\begin{prop}\label{proposition: everything is xi-complete}
Let $M$ be a left $\kk\{\xi\}$-module and let $N$ be any of $\widehat{M}$, $\mathbb{H}_\xi M$, and $\mathbb{E}_\xi M$. Then $L_1N=0$ and $\phi_N\colon N\to L_0N$ is an isomorphism. In particular, if $\phi_M\colon M \to L_0M$ is an isomorphism, then $L_1M\cong L_1L_0M=0$.
\end{prop}

\begin{proof}
Using the exact sequence~\eqref{equation: Cinfty} and Corollary~\ref{corollary: derived completion are hom and ext} we obtain that $L_1N=0$ and $\phi_N$ is an isomorphism if and only if 
\begin{equation}\label{equation: derived complete = right orthogonal}
\Hom_{\kk\{\xi\}}(\kk\{\xi^{\pm}\},N)=0 \;\; \text{and}\;\; \Ext_{\kk\{\xi\}}(\kk\{\xi^{\pm}\},N)=0. 
\end{equation}
This condition certainly holds if $p^rN$ for some $r\geq 1$, so it holds for all $M_r$ and $M/\xi^rM$. If~\eqref{equation: derived complete = right orthogonal} holds for modules $N_i\in \Mod_{\kk\{\xi\}}$, then it holds for their product $\prod N_i$. It is also clear that if~\eqref{equation: derived complete = right orthogonal} holds for any two modules in a short exact sequence
$$0 \to M' \to M \to M'' \to 0,$$
then it holds for the third one as well. Now the proposition follows for $\widehat{M}$ by applying these observations to a short exact sequence
$$0 \to \widehat{M} \to \prod_{r\geq 1}M/\xi^r M \to \prod_{r\geq 1}M/\xi^r M \to 0. $$
In a similar way, the proposition follows for $L_0M=\mathbb{E}_{\xi}M$ and $L_1M=\mathbb{H}_{\xi}M$ as well by considering the exact sequence~\eqref{equation: six term exact sequence H,E}.
\end{proof}

We write $\Mod_{\kk\{\xi\}}^{\mathrm{comp}}$ for the full subcategory of $\Mod_{\kk\{\xi\}}$ spanned by derived $\xi$-complete left $\kk\{\xi\}$-modules. We summarize the results of this section in the next proposition.

\begin{prop}\label{proposition: derived modules, summary}
$\quad$
\begin{enumerate}
\item Let $M\in \Mod_{\kk\{\xi\}}$. Then $M\in \Mod_{\kk\{\xi\}}^{\mathrm{comp}}$ if and only if $\Hom(\kk\{\xi^{\pm}\},N)=0$ and $\Ext(\kk\{\xi^{\pm}\},N)=0$.
\item The subcategory $\Mod_{\kk\{\xi\}}^{\mathrm{comp}} \subset \Mod_{\kk\{\xi\}}$ is closed under taking any limits, cokernels, and extensions. In particular, $\Mod_{\kk\{\xi\}}^{\mathrm{comp}}$ is a weak Serre subcategory in $\Mod_{\kk\{\xi\}}$, see e.g.~\cite[{Tag 02MO}]{stacks-project} \qed
\end{enumerate}
\end{prop}

\subsection{The \texorpdfstring{$\kk$}{F}-completion of simplicial restricted Lie algebras}\label{section: F-completion}

Recall that a map $f\colon L'_\bullet \to L_\bullet$ in $\srLie$ is an $\kk$-equivalence if the induced map $$f_*\colon \widetilde{H}_*(L'_\bullet;\kk) \to \widetilde{H}_*(L_\bullet;\kk)$$ is an isomorphism. Note that this definition coincides with Definition~\ref{definition: barW-equivalence} because of Proposition~\ref{proposition: chain coalgebra and chain complex}. We say that $Z_\bullet \in \srLie$ is \emph{$\kk$-complete} if 
$$f^*\colon [L_\bullet,Z_\bullet] \to [L'_\bullet, Z_\bullet] $$
is a bijection for any $\kk$-equivalence $f\colon L'_\bullet \to L_\bullet$  between cofibrant objects in $\srLie$.

\begin{dfn}\label{definition: F-completion}
A map $\phi\colon L_\bullet \to L_\xi L_\bullet$ is called the \emph{$\kk$-completion} of $L_\bullet\in \srLie$ if $\phi$ is an $\kk$-equivalence and $L_\xi L_\bullet$ is $\kk$-complete. Note that a simplicial restricted Lie algebra $Z_\bullet\in \srLie$ is $\kk$-complete if and  only if $Z_\bullet$ is a fibrant object in the model structure of Theorem~\ref{theorem:model structure srlie, barW-equivalence}.
\end{dfn}

\begin{rmk}\label{remark: existence of F-completion} By Theorem~\ref{theorem:model structure srlie, barW-equivalence}, any $\kk$-equivalence between $\kk$-complete objects is a weak equivalence, the $\kk$-completion exists, and it is unique up to (a chain of) weak equivalences.
\end{rmk}

In this section we will give an explicit construction of the $\kk$-completion 
$$\phi\colon L_\bullet \to L_\xi L_\bullet$$
provided $\pi_0(L_\bullet)=0$. We begin with Eilenberg-MacLane Lie algebras.

\begin{lmm}\label{lemma: examples of F-complete EM}
The following Eilenberg-MacLane Lie algebras are $\kk$-complete:
\begin{enumerate}
\item $K(V,n)\in \srLie$, $V\in\Vect_\kk$ is a vector space over $\kk$, $n\geq 0$;
\item $K(F/\xi^r,n)\in \srLie$, $F\in \Mod_{\kk\{\xi\}}$ is a free left $\kk\{\xi\}$-module, $r\geq 1$, $n\geq 0$;
\item $K(\widehat{F},n) \in \srLie$, $\widehat{F}$ is the $\xi$-adic completion of a free $\kk\{\xi\}$-module $F$, $n\geq 0$.
\end{enumerate}
\end{lmm}

\begin{proof}
By Proposition~\ref{remark: cohomology are representable} and the universal coefficient theorem (Remark~\ref{remark: universal coefficient formula}),  we have $$[-, K(V,n)]\cong \widetilde{H}^{n+1}(-;V)\cong \Hom(\widetilde{H}_{n+1}(-;\kk),V),$$
which proves the first part. 

We will prove the second part by induction on $r$. The base case, $r=1$ is already done because $F/\xi$ is a vector space over $\kk$. Assume that $K(F/\xi^r,n)$ is $\kk$-complete; we will show that $K(F/\xi^{r+1},n)$ is $\kk$-complete as well. There is a fiber sequence
$$K(F/\xi^{r+1},n) \to K(F/\xi^r,n) \to K(\xi^r F/ \xi^{r+1} F,n+1)\simeq K(F/\xi,n+1). $$
By the inductive assumption, $K(F/\xi^r,n)$ and $K(\xi^r F/ \xi^{r+1} F,n+1)$ are $\kk$-complete, and so is the fiber $K(F/\xi^{r+1},n)$.

Finally, $K(\widehat{F},n)$ is $\kk$-complete since there is a fiber sequence
\begin{equation*}
K(\widehat{F},n) \to \prod_{r\geq 1} K(F/\xi^r,n) \to \prod_{r\geq 1} K(F/\xi^r,n). \qedhere
\end{equation*}
\end{proof}

\begin{lmm}\label{lemma: EM of free module}
Let $F$ be a free left $\kk\{\xi\}$-module. Then the $\xi$-adic completion map $\phi_F\colon F \to \widehat{F}$ induces an $\kk$-equivalence $$\phi^n_F\colon K(F,n) \to K(\widehat{F},n)$$
for any $n\geq 0$.
\end{lmm}

\begin{proof}
Note that $\widehat{F}\in \Mod_{\kk\{\xi\}}$ is a torsion-free module. Therefore, by Proposition~\ref{proposition: homology of abelian Lie algebra}, there is a commutative diagram
$$
\begin{tikzcd}
H_*(K(F,0);\kk)  \arrow{r}{\phi^0_{F*}}
& H_*(K(\widehat{F},0);\kk) \\
\Lambda^*(F/\xi)\arrow{r}{\phi_F} \arrow{u}{\cong}
& \Lambda^*(\widehat{F}/\xi), \arrow{u}{\cong}
\end{tikzcd}
$$
where $\Lambda^*(-)$ is the exterior algebra and both vertical arrows are isomorphisms. Since $F/\xi \cong \widehat{F}/\xi$, the lemma follows for $n=0$. For $n>0$, consider a map of fiber sequences
$$
\begin{tikzcd}
K(F,n) \arrow{r} \arrow{d}{\phi^n_F}
& 0 \arrow{d} \arrow{r}
& K(F,n+1)\arrow{d}{\phi^{n+1}_F} \\
K(\widehat{F},n) \arrow{r} 
& 0 \arrow{r}
& K(\widehat{F},n+1).
\end{tikzcd}
$$
By Corollary~\ref{corollary: SSS fiber is EM}, it induces the map $f^r_{s,t}\colon E_{s,t}^r \to \widehat{E}_{s,t}^r$ of strongly convergent spectral sequences:
$$E_{s,t}^2 = H_t(K(F,n);\kk)\otimes H_s(K(F,n+1);\kk) \Rightarrow H_{s,t}(0;\kk), $$
$$\widehat{E}_{s,t}^2 = H_t(K(\widehat{F},n);\kk)\otimes H_s(K(\widehat{F},n+1);\kk) \Rightarrow H_{s,t}(0;\kk), $$
such that $f^2_{0,*} = \phi^n_{F*}$, $f^2_{*,0}=\phi^{n+1}_{F*}$, and $f^{\infty}_{*,*}$ is an isomorphism. Then Zeeman’s comparison theorem (see e.g.~\cite[Theorem~3.26]{McCleary01}) implies that $\phi^n_{F*}$ is an isomorphism if and only if $\phi^{n+1}_{F*}$ is an isomorphism. Finally, the induction by $n$ implies the lemma.
\end{proof}

Let $M\in \Mod_{\kk\{\xi\}}$ be any left $\kk\{\xi\}$-module and let
$$0 \to F_1 \xrightarrow{f} F_0 \to M \to 0 $$
be a free resolution of $M$. The map $f$ induces the map $\hat{f}\colon \widehat{F}_1 \to \widehat{F}_0$ of $\xi$-adic completions, and so it induces the map 
$$\hat{f}^{n+1}\colon K(\widehat{F}_1,n+1) \to K(\widehat{F}_0,n+1) $$
of Eilenberg-MacLane Lie algebras. We denote by $L_\xi K(M,n)$ the \emph{homotopy fiber} of~$\hat{f}^n$. There is a map of fiber sequences in $\srLie$:
\begin{equation}\label{equation: F-completion of EM}
\begin{tikzcd}
K(F_0,n) \arrow{r} \arrow{d}{\phi^n_{F_0}}
& K(M,n) \arrow{r} \arrow[dashed]{d}{\phi^n_M}
& K(F_1,n+1) \arrow{r}{f^{n+1}} \arrow{d}{\phi^{n+1}_{F_1}}
& K(F_0,n+1) \arrow{d}{\phi^{n+1}_{F_0}} \\
K(\widehat{F}_0,n) \arrow{r}
& L_{\xi}K(M,n) \arrow{r}
& K(\widehat{F}_1,n+1) \arrow{r}{\hat{f}^{n+1}}
& K(\widehat{F}_0,n+1),
\end{tikzcd}
\end{equation}
where $\phi^n_M$ is the induced map between homotopy fibers $K(M,n)\simeq\fib(f^{n+1})$ and $L_{\xi}K(M,n)=\fib(\hat{f}^{n+1})$.

\begin{prop}\label{proposition: F-completion of any EM}
$\phi^n_M$ is the $\kk$-completion of $K(M,n)\in \srLie$.
\end{prop}

\begin{proof}
By Lemma~\ref{lemma: examples of F-complete EM}, Eilenberg-MacLane Lie algebras $K(\widehat{F}_1,n+1)$ and $K(\widehat{F}_0,n+1)$ are $\kk$-complete, and so is $L_{\xi}K(M, n)$ as a homotopy fiber between $\kk$-complete objects in $\srLie$. 

By Lemma~\ref{lemma: EM of free module}, the maps $\phi^n_{F_0}$ and $\phi^{n+1}_{F_1}$ are $\kk$-equivalences. By applying the Serre spectral sequence of Corollary~\ref{corollary: SSS fiber is EM} to the leftmost map of fiber sequences in~\eqref{equation: F-completion of EM}, we obtain that $\phi^n_{M}$ is an $\kk$-equivalence as well.
\end{proof}

\begin{cor}\label{corollary: EM is F-complete iff derived completed}
$\quad$ 
\begin{enumerate}
\item There are natural isomorphisms $$\pi_n(L_\xi K(M,n))\cong L_0M,\;\;\pi_{n+1}(L_\xi K(M,n))\cong L_1 M,$$
and the induced map $\pi_n(\phi^n_M)$ is the derived $\xi$-completion map~\eqref{equation: derived completion map}.
\item An Eilenberg-MacLane Lie algebra $K(M,n), n\geq 0$ is $\kk$-complete if and only if $M\in \Mod_{\kk\{\xi\}}$ is a left derived $\xi$-adic complete $\kk\{\xi\}$-module.
\end{enumerate}
\end{cor}

\begin{proof}
The first part follows immediately from definitions. For the second part, suppose first that $K(M,n)$ is $\kk$-complete. Then, by Proposition~\ref{proposition: F-completion of any EM}, the map~$\phi^n_M$ is a weak equivalence. Therefore $M\cong L_0M$ and $L_1M=0$. 

In the opposite direction, suppose that $M$ is derived $\xi$-complete. Then by the first part, $\phi^n_M$ is a weak equivalence, and so the Eilenberg-MacLane Lie algebra $K(M,n)$ is $\kk$-complete by Proposition~\ref{proposition: F-completion of any EM}.
\end{proof}

Let $L_\bullet \in \srLie$ be a \emph{connected} simplicial restricted Lie algebra. We will define inductively the $\kk$-completions $$\phi_n\colon L^{\leq n}_\bullet \to L_\xi(L^{\leq n}_\bullet)$$ of the Postnikov truncations $L_\bullet^{\leq n}$ of $L_\bullet$, see Section~\ref{section: postnikov system}.  

If $n=0$, then $L_\bullet^{\leq 0}$ is contractible, and so we set $L_{\xi}L_\bullet^{\leq 0}=0$. Clearly, the unique map $$\phi_0\colon L_\bullet^{\leq 0} \to L_{\xi}L_\bullet^{\leq 0}=0$$ is a weak equivalence.

Next, assume that the $\kk$-completion $\phi_n\colon L_\bullet^{\leq n} \to L_\xi L_\bullet^{\leq n}$ is defined, $n\geq 0$; we will construct an $\kk$-completion $\phi_{n+1}$. Consider the fiber sequence
$$K(M,n+1) \to L_\bullet^{\leq{(n+1)}} \xrightarrow{\beta^{n}} L_\bullet^{\leq n} \xrightarrow{k^n} K(M, n+2)  $$
from Corollary~\ref{corollary: postnikov tower}, where $M=\pi_{n+1}(L_\bullet)$. By the universal property of $\kk$-completions, there is a map
$$\hat{k}^n\colon L_\xi L_\bullet^{\leq n} \to L_\xi K(M, n+2) $$
such that the diagram
$$
\begin{tikzcd}
L_\bullet^{\leq n} \arrow[swap]{d}{\phi_n} \arrow{r}{k^n}
& K(M, n+2) \arrow{d}{\phi^{n+2}_{M}} \\
L_\xi L_\bullet^{\leq n}  \arrow{r}{\hat{k}^n}
& L_\xi K(M, n+2) 
\end{tikzcd}
$$
is homotopy commutative in $\srLie$. We define $L_\xi L_\bullet^{\leq (n+1)}$ as the homotopy fiber $\fib(\hat{k}^n)$ of the map $\hat{k}^n$, $L_\xi L_\bullet^{\leq(n+1)}=\fib(\hat{k}^n)$; and then we define $$\phi_{n+1}\colon L_\bullet^{\leq (n+1)} \to L_\xi  L_\bullet^{\leq (n+1)} $$ as the induced map between homotopy fibers of $k^n$ and $\hat{k}^n$.

\begin{lmm}\label{lemma: phin+1 is F completion}
$\phi_{n+1}$ is the $\kk$-completion of $L_\bullet^{\leq (n+1)}$.
\end{lmm}

\begin{proof}
Clearly, $L_\xi L_\bullet^{\leq (n+1)}$ is $\kk$-complete as a homotopy fiber of a map between $\kk$-complete objects in $\srLie$. Moreover, there is a map of fiber sequences
\begin{equation*}
\begin{tikzcd}
K(M,n+1) \arrow{r} \arrow{d}{\phi^{n+1}_{M}}
& L_\bullet^{\leq (n+1)} \arrow{r}{\beta^n} \arrow{d}{\phi_{n+1}}
& L_\bullet^{\leq n} \arrow{d}{\phi_{n}} \\
L_\xi K(M,n+1) \arrow{r}
& L_\xi L_\bullet^{\leq (n+1)} \arrow{r}{\hat{\beta}^n}
& L_\xi L_\bullet^{\leq n}
\end{tikzcd}
\end{equation*}
such that the maps $\phi^{n+1}_M$ and $\phi_n$ are $\kk$-equivalences. By applying the Serre spectral sequence (Corollary~\ref{corollary: SSS fiber is EM}), we obtain that $\phi_{n+1}$ is an $\kk$-equivalence as well.
\end{proof}

Note that one can compute homotopy groups $\pi_*(L_\xi L^{\leq n}_\bullet)$ of the $\kk$-completion $L_\xi L^{\leq n}_\bullet$ by using this inductive construction. Namely, we have the following statement.
\begin{cor}\label{corollary: F-completion, homotopy groups, postnikov}
Let $L_\bullet \in \srLie$ be a simplicial restricted Lie algebra, $\pi_0(L_\bullet)=0$. Then we have
$$
\pi_i(L_\xi L_\bullet^{\leq n})=\left\{
\begin{array}{ll}
L_1\pi_{n}(L_\bullet^{\leq n}) = L_1\pi_n(L_\bullet) & \mbox{if $i=n+1$,}\\
0 & \mbox{if $i>n+1$.}
\end{array}
\right.
$$
Moreover, there is a short exact sequence
$$0 \to L_1\pi_{i-1}(L_\bullet) \to \pi_i(L_\xi L_\bullet^{\leq n}) \to L_0\pi_i(L_\bullet) \to 0$$
for each $i\leq n$. \qed
\end{cor}

Finally, we define a map $\phi\colon L_\bullet \to L_\xi L_\bullet$ %for any connected simplicial restricted Lie algebra $L_\bullet \in \srLie$ 
as follows
$$\phi = \holim_{n} \phi_n\colon L_\bullet \to L_\xi L_\bullet =\holim_{n} L_\xi(L_\bullet^{\leq n}), \; L_\bullet\in \srLie, \; \pi_0(L_\bullet)=0.$$

\begin{prop}\label{proposition: F-completion for any connected}
$\phi$ is the $\kk$-completion of $L_\bullet\in \srLie$, $\pi_0(L_\bullet)=0$.
\end{prop}

\begin{proof}
The simplicial restricted Lie algebra $L_\xi L_\bullet\in \srLie$ is $\kk$-complete as a homotopy limit over a diagram of $\kk$-complete objects in $\srLie$. We will show that $\phi$ is an $\kk$-equivalence. For each $i\geq 0$, there is an integer $n_i$ such that maps 
$$\alpha^n_*\colon \widetilde{H}_i(L_\bullet;\kk) \to \widetilde{H}_i(L_\bullet^{\leq n};\kk), $$
$$\hat{\alpha}^n_* \colon \widetilde{H}_i(L_\xi L_\bullet;\kk) \to \widetilde{H}_i(L_\xi L_\bullet^{\leq n};\kk) $$
induced by the projections $\alpha^n\colon L_\bullet\to L_\bullet^{\leq n}$, $\hat{\alpha}^n\colon L_\xi L_\bullet \to L_\xi L_\bullet^{\leq n}$ are isomorphisms for all $n\geq n_i$. Indeed, by Corollary~\ref{corollary: F-completion, homotopy groups, postnikov}, the connectivity of $\alpha^n$ and $\hat{\alpha}^{n}$ increases with $n$; this fact and Lemma~\ref{lemma: reduced replacement} imply the claim. Finally, the proposition follows by a commutative diagram
$$
\begin{tikzcd}
\widetilde{H}_i(L_\bullet;\kk) \arrow[swap]{d}{\alpha^n_*} \arrow{r}{\phi_*}
& \widetilde{H}_i(\holim L_\xi L_\bullet^{\leq n};\kk) \arrow{d}{\hat{\alpha}^n_*} \\
\widetilde{H}_i(L^{\leq n}_\bullet;\kk) \arrow{r}{\phi_{n*}}
& \widetilde{H}_i(L_\xi L_\bullet^{\leq n};\kk)
\end{tikzcd}
$$
and Lemma~\ref{lemma: phin+1 is F completion}.
\end{proof}

\begin{cor}\label{corollary: F-completion, homotopy groups}
Let $L_\bullet \in \srLie$ be a simplicial restricted Lie algebra, $\pi_0(L_\bullet)=0$. Then we have a short exact sequence
$$0 \to L_1\pi_{i-1}(L_\bullet) \to \pi_i(L_\xi L_\bullet) \to L_0\pi_i(L_\bullet) \to 0$$
for each $i\geq 0$. \qed
\end{cor}

\begin{cor}\label{corollary: F-complete are derived Xi-complete}
A connected simplicial restricted Lie algebra $L_\bullet \in \srLie$ is $\kk$-complete if and only if all homotopy groups $\pi_i(L_\bullet),i\geq 0$ are derived $\xi$-adic complete left modules over the ring $\kk\{\xi\}$.
\end{cor}

\begin{proof}
First, suppose that $\pi_i(L_\bullet),i\geq 0$ are derived $\xi$-adic complete. Then $$L_1\pi_i(L_\bullet)=0 \;\; \text{and} \;\; L_0\pi_i(L_\bullet)\cong \pi_i(L_\bullet)$$ for each $i\geq 0$. By Corollary~\ref{corollary: F-completion, homotopy groups}, this implies that $\pi_i(L_\xi L_\bullet)\cong \pi_i{L_\bullet}$, $i\geq 0$. Therefore the $\kk$-completion map $\phi\colon L_\bullet \to L_\xi L_\bullet$ is a weak equivalence, and so $L_\bullet$ is $\kk$-complete.

Next, suppose that $L_\bullet \in \srLie$ is $\kk$-complete. Then $\phi\colon L_\bullet \to L_\xi L_\bullet$ is a weak equivalence, and so, by Corollary~\ref{corollary: F-completion, homotopy groups}, there is a short exact sequence
$$0 \to L_1\pi_{i-1}(L_\bullet) \to \pi_i(L_\bullet) \to L_0\pi_i(L_\bullet) \to 0, \; i\geq 0. $$
By Proposition~\ref{proposition: everything is xi-complete}, the outer terms in this exact sequence are derived $\xi$-complete $\kk\{\xi\}$-modules. Therefore, by Proposition~\ref{proposition: derived modules, summary}, the homotopy groups $\pi_{i}(L_\bullet)$, $i\geq 0$ are derived $\xi$-adic complete as well.
\end{proof}

\section{Adams-type spectral sequence}\label{section:adams}
In this section we illustrate our previous results by proving Theorems~\ref{theorem: intro, E} and~\ref{theorem: intro, F} from the introduction. 

In Section~\ref{section: steenrod} we recall basic properties of the Steenrod operations. In Definition~\ref{definition: homogenized steenrod algebra}, we introduce the \emph{homogenized mod-$p$ Steenrod algebra} $\calA^h_p$, which is the associated graded algebra of the classical Steenrod algebra $\calA_p$ with respect to a certain filtration, see Remark~\ref{remark: filtration on steenrod}. In Definition~\ref{definition: unstable algebra}, we define the category $\calU^h$ of \emph{unstable $\calA^h_p$-algebras} and we show that the reduced cohomotopy groups $\widetilde{\pi}^*(C_\bullet), C_\bullet\in \sotrcoalg$ and the cohomology groups $\widetilde{H}^*(L_\bullet;\kk), L_\bullet \in \srLie$ are unstable $\calA^h_p$-algebras in Examples~\ref{example: steenrod algebra, truncated} and~\ref{example: steenrod algebra, restricted Lie} respectively. We also define abelian categories $\calM^h$ and $\calM^h_0$ of \emph{unstable $\calA^h_p$-modules} and \emph{strongly unstable $\calA^h_p$-algebras} in Definitions~\ref{definition: unstable module} and~\ref{definition: strongly unstable module} respectively. We describe free objects in each of these categories $\calU^h_p$, $\calM^h$, and $\calM^h_0$ in Remark~\ref{remark: free unstable, admissible}. Finally, we recall from~\cite{Priddy73} that the reduced cohomotopy groups $\widetilde{\pi}^*(\Sym^{tr}(V_\bullet))$ of a cofree simplicial truncated coalgebra $\Sym^{tr}(V_\bullet)$ is a free unstable $\calA^h_p$-algebra (Theorem~\ref{theorem: cohomotopy of free truncated is free unstable}).

In Section~\ref{section: BKSS} we prove Theorem~\ref{theorem: intro, E} as Corollary~\ref{corollary: BKSS, restricted Lie algebra}. First, we recall the definition of non-abelian $\Ext$-groups (Definition~\ref{definition: unstable exts}). Then we apply Theorem~\ref{definition: suspension of unstable algebra} to obtain the spectral sequence of Theorem~\ref{theorem: BKSS, coalgebras}, which computes homotopy groups of a derived mapping space $$\map_{\sCA_0}(C_\bullet, D_\bullet),\; C_\bullet,D_\bullet \in \sotrcoalg$$
from reduced cohomotopy groups $\widetilde{\pi}^*(C_\bullet)$ and $\widetilde{\pi}^*(D_\bullet)$. Finally, we use Theorem~\ref{theorem: intro, C} to derive Theorem~\ref{theorem: intro, E} from Theorem~\ref{theorem: BKSS, coalgebras}.

In Section~\ref{section: unstable koszul} we recall the definition of the lambda algebra $\Lambda$ of~\cite{6authors} and recall that the algebra $\Lambda$ is anti-isomorphic to the Koszul dual algebra $\calK^*_p$ of the algebra $\calA^h_p$ (see~\eqref{equation: koszul to steenrod is Lambda, p is odd}). Then, we construct \emph{unstable} and \emph{strongly unstable Koszul complexes} $K_\bullet(W)$ (see~\eqref{equation: unstable koszul complex}) and $K_\bullet^0(W)$ (see~\eqref{equation: strongly unstable koszul complex}) for a trivial $\calA^h_p$-module $W\in \Vect^{gr}_{\kk}$ respectively. In Proposition~\ref{proposition: unstable koszul resolution, p is odd}, we show that these complexes are acyclic. Thus, we use them to compute unstable abelian $\Ext$-groups $\Ext^s_{\calM^h}(W,\Sigma^t\kk)$ and $\Ext^s_{\calM^h_0}(W,\Sigma^t\kk)$ in terms of the algebra $\Lambda$ in Corollary~\ref{corollary: unstable exts between trivial}. 

In the final section~\ref{section: free lie} we examine the spectral sequence~\eqref{equation: ASS, Lie} in the particular case $L_\bullet=\free(V_\bullet)$ is a free restricted Lie algebra. First, in Proposition~\ref{proposition: homotopy groups of free restricted as xi-module}, we use the Curtis theorem~\cite{Curtis_lower}, to compute the homotopy groups $\pi_*(L_\xi\free(V_\bullet))$ of the $\kk$-completion $L_\xi \free(V_\bullet)$. Then, we observe from Corollary~\ref{corollary: unstable exts between trivial} that the second page of the spectral sequence~\eqref{equation: intro, ASS} is computable provided $\pi_*(V_\bullet)$ is one-dimensional, see Corollaries~\ref{corollary: ASS, degenerates, p=2} and~\ref{corollary: ASS, degenerates, p is odd, l is odd}.  This will allow us to derive Theorem~\ref{theorem: intro, F}. In Remarks~\ref{remark: ASS, highly dimensionsional} and~\ref{remark: hilton-milnor}, we discuss the spectral sequence~\eqref{equation: ASS, Lie} and the homotopy groups $\pi_*(\free(V_\bullet))$ in the case $\dim\pi_*(V_\bullet)>1$. We end the section with Remark~\ref{remark: lambdas and mu are detected} concerning the connection between the generators of the algebra $\Lambda$ and the generators of the algebra $\calA^h_p$.

In this section we heavily use Steenrod operations. We recall that the standard notation is different for $p$ is odd and $p=2$; we will enclose the case $p=2$ in parentheses.

\subsection{Steenrod operations}\label{section: steenrod} 

\begin{dfn}\label{definition: cohomotopy}
Let $V_\bullet\in \sVect_{\kk}$ be a simplicial vector space. For each $q\geq 0$, let $\pi^{q}(V_\bullet)$ denote the \emph{$q$-th cohomotopy group} of $V_\bullet$, that is $\pi^q(V_\bullet)=\Hom(\pi_q(V_\bullet),\kk)$.
Similarly, let $C_\bullet \in \sotrcoalg$ be a reduced simplicial truncated coalgebra. Set $\widetilde{\pi}^*(C_\bullet)$ to be the \emph{reduced} cohomotopy groups of $C_\bullet$: 
$$\widetilde{\pi}^*(C_\bullet)= \pi^*(\oblv(C_\bullet)) =\bigoplus_{q>0}\pi^q(C_\bullet).$$
%see Example~\ref{example: steenrod algebra, truncated}.
\end{dfn}

Let $C_\bullet\in \scoalg$ be a simplicial coalgebra over $\kk$. Then by the Eilenberg-Zilber theorem, we obtain that $\pi^*(C_\bullet)$ is a \emph{graded} commutative algebra over $\kk$. Furthermore, in~\cite{Dold61} (see also~\cite{May70operations}) A.~Dold constructed functorial \emph{Steenrod operations}
$$Sq^a\colon \pi^q(C_\bullet) \to \pi^{q+a}(C_\bullet), q,a\geq 0 \;\; \text{if $p=2$, and}$$
$$\beta^{\e}P^a\colon \pi^q(C_\bullet) \to \pi^{q+2a(p-1)+\e}(C_\bullet), q,a\geq 0,\;\e=0,1 \;\; \text{if $p>2$}.$$
These operations satisfy the following list of properties (where, by abuse of notation, $\beta^1 P^a =\beta P^a$ and $\beta^0 P^a = P^a$):
\begin{enumerate}
\item $\beta^{\e}P^a(\alpha x)=\alpha^p\beta^{\e}P^a(x)$ (resp. $Sq^a(\alpha x)=\alpha^2 Sq^a(x)$), $\alpha \in \kk$, $x\in \pi^q(C_\bullet)$.
%\item $\beta^{\e}P^a=0$ (resp. $Sq^a =0 $) if $a<0$.
\item $\beta^{\e}P^a(x)=0$ (resp. $Sq^a(x) =0 $),  if $2a+\e>q$ (resp. $a>q$) and $x\in \pi^q(C_\bullet)$.
\item $P^{a}(x)=x^p$ (resp. $Sq^a=x^2$) if $q=2a$ (resp. $q=a$).
\item Cartan formula: $$P^a=\sum_{i=0}^{a} P^i\otimes P^{a-i}\;\; \text{and} \;\; \beta P^a=\sum_{i=0}^{a}(\beta P^i\otimes P^{a-i}+ P^i\otimes \beta P^{a-i})$$ (resp. $Sq^a=\sum_{i=0}^{a} Sq^i \otimes Sq^{a-i}$) on $\pi^*(C_\bullet \times C'_\bullet)\cong \pi^*(C_\bullet)\otimes \pi^*(C'_\bullet)$. 
\item Stability: if $\sigma\colon \pi^q(C_\bullet) \xrightarrow{\cong} \pi^{q+1}(\Sigma_\bullet C_\bullet)$ is the suspension isomorphism, then $\sigma \beta^{\e} P^a =(-1)^{\e}\beta^{\e} P^a \sigma$ (resp. $\sigma Sq^a = Sq^a\sigma$).
\item Adem-Epstein relations. If $p$ is odd, $a<pb$, and $\e=0,1$, then
\begin{align}\label{equation: adem relation, p odd, 1}
\beta^{\e}P^aP^b &= \sum_{j=0}^{a+b}(-1)^{a+j}\binom{(p-1)(b-j)-1}{a-pj}\beta^{\e}P^{a+b-j}P^j;
\end{align}
and if $a\leq pb$ and $\e=0,1$, then
\begin{align}\label{equation: adem relation, p odd, 2}
\beta^{\e}P^a\beta P^b &= \sum_{j=0}^{a+b}(-1)^{a+j-1}\binom{(p-1)(b-j)-1}{a-pj-1}\beta^{\e}P^{a+b-j}\beta P^j \\
&+(1-\e)\sum_{j=0}^{a+b}(-1)^{a+j}\binom{(p-1)(b-j)}{a-pj}\beta P^{a+b-j} P^j. \nonumber
\end{align}
Similarly, if $p=2$ and $a<2b$, then
\begin{equation}\label{equation: adem relation, p=2}
Sq^aSq^b = \sum_{j=0}^{a+b}\binom{b-j-1}{a-2j}Sq^{a+b-j}Sq^j. 
\end{equation}
\item the operation $P^0$ (resp. $Sq^0$) is induced by the Verschiebung operator (see Definition~\ref{definition: coalgebra verschiebung})
$$V\colon C_\bullet \to C_\bullet^{(1)}.$$
\end{enumerate}

Recall that the \emph{mod-$p$ Steenrod algebra} is the $\F_p$-algebra $\calA_p$ generated by Steenrod operations $\beta^{\e}P^a, a\geq 0, \e=0,1$ (resp. $Sq^a,a\geq 0$) subject to Adem relations~\eqref{equation: adem relation, p odd, 1} and~\eqref{equation: adem relation, p odd, 2} (resp. the Adem relation~\eqref{equation: adem relation, p=2}) and subject to the additional relation:
$$P^0=1 \;\; \text{(resp. $Sq^0=1$)}. $$

\begin{dfn}\label{definition: homogenized steenrod algebra}
The \emph{homogenized mod-$p$ Steenrod algebra} $\mathcal{A}_p^h$ is the associative algebra over $\F_p$ generated by the elements $\beta^{\e}P^a,a\geq 0, \e=0,1$ (resp. $Sq^a,a \geq 0$ if $p=2$) subject to Adem relations~\eqref{equation: adem relation, p odd, 1} and~\eqref{equation: adem relation, p odd, 2} (resp. the relation~\eqref{equation: adem relation, p=2}) and subject to the additional relation $$P^0=0 \;\; \text{(resp. $Sq^0=0$)}.$$
\end{dfn}

\begin{rmk}\label{remark: filtration on steenrod}
One can endow the Steenrod algebra $\mathcal{A}_p$ with an increasing multiplicative \emph{weight} filtration $F_w\mathcal{A}_p$ by defining $F_0\mathcal{A}_p$ be spanned by the unit $1\in \mathcal{A}_p$ and $F_1\mathcal{A}_p$ be spanned by the set of generators $\beta^{\e}P^a,a\geq 0,\e=0,1$ (resp. $Sq^a,a\geq 0$), see~\cite{Priddy70} and~\cite{PP05}. Then the associated graded algebra $\mathrm{gr}_F \mathcal{A}_p$ is isomorphic to $\mathcal{A}_p^h$. In particular, the algebra $\mathcal{A}_p^h$ is \emph{bigraded}:
$$|\beta^{\e}P^a|=(2a(p-1)+\e,1) \;\; \text{and} \;\; |Sq^a|=(a,1). $$
We refer to the first grading as \emph{internal} and to the second one as \emph{weight}.
\end{rmk}

\begin{dfn}\label{definition: module over hsa}
A \emph{left $\mathcal{A}_p^h$-module} $M_*=\bigoplus_{q>0} M_q$ is a positively graded vector space over $\kk$ equipped with an $\F_p$-linear left action of the homogenized Steenrod algebra $\mathcal{A}_p^h$ such that
\begin{enumerate}
%\item $\kk \subset \mathcal{A}_p^h$ acts on $M_*$ by the scalar multiplication;
\item $\beta^{\e}P^a(M_q)\subset M_{q+2a(p-1)+\e}, a\geq 0, \e=0,1$ (resp. $Sq^a(M_q)\subset M_{q+a},a\geq 0$);
\item $\beta^{\e}P^a(\alpha x)=\alpha^p\beta^{\e}P^a(x)$ (resp. $Sq^a(\alpha x)=\alpha^2 Sq^a(x)$), $\alpha\in \kk$, $x\in M_q$.
\end{enumerate}
We denote by $\Mod_{\mathcal{A}_p^h}$ the abelian category of left $\mathcal{A}_p^h$-modules.
\end{dfn}

\begin{dfn}\label{definition: unstable algebra}
An \emph{unstable $\mathcal{A}_p^h$-algebra} is a left $\mathcal{A}_p^h$-module $A_*\in \Mod_{\mathcal{A}_p^h}$ which is a non-unital graded commutative algebra such that
\begin{enumerate}
\item $P^a(x)=x^p$ for all $x\in A_{2a}$ (resp. $Sq^a(x)=x^2$ for all $x\in A_a$.)
\item $\beta^{\e}P^a(x)=0$ for all $x \in A_q$ and $2a+\e>q$ (resp. $Sq^a(x)=0$ for all $x\in A_q$ and $a>q$.)
\end{enumerate}
We denote by $\calU^h$ the category of unstable $\mathcal{A}_p^h$-algebras.
\end{dfn}

\begin{exmp}\label{example: steenrod algebra, truncated}
Let $C_\bullet\in \sotrcoalg$ be a reduced simplicial truncated coalgebra. By the definition, the Verschiebung operator $V\colon C_\bullet \to C_\bullet^{(1)}$ factors through the constant coalgebra $\kk_\bullet$. Therefore $P^0$ (resp. $Sq^0$) acts by zero on $\pi^i(C_\bullet)$ for all $i>0$, and so the reduced non-unital algebra
$$\widetilde{\pi}^*(C_\bullet)= \pi^*(\oblv(C_\bullet)) =\bigoplus_{q>0}\pi^q(C_\bullet)$$ is an unstable $\mathcal{A}_{p}^h$-algebra. 
\end{exmp}

\begin{exmp}\label{example: steenrod algebra, restricted Lie}
Let $L_\bullet \in \srLie$ be a restricted Lie algebra. Then $\barW U^r(L_\bullet)$ is a reduced simplicial truncated coalgebra and
$$\widetilde{\pi}^*(\barW U^r(L_\bullet)) \cong \widetilde{H}^*(L_\bullet,\kk) $$
by Proposition~\ref{proposition: chain coalgebra and chain complex} and Definition~\ref{definition: cochain complex}. Therefore the cohomology groups $\widetilde{H}^*(L_\bullet,\kk)$ form an unstable $\calA_p^h$-algebra, cf.~\cite[Section~5]{Priddy70long} and~\cite[Theorem~8.5]{May70operations}.
\end{exmp}

\begin{dfn}\label{definition: suspension of unstable algebra}
We define the \emph{suspension} $\Sigma A_*\in \calU^h$ of an unstable $\calA^h_p$-algebra $A_*$ as follows
\begin{enumerate}
\item $(\Sigma A_*)_{q+1} = A_{q}$ for $q\geq 0$. If $x\in A_q$, then we write $\sigma x$ for the corresponding element in $(\Sigma A_*)_{q+1}$;
\item $\Sigma A_*$ has zero multiplication;
\item $\beta^{\e} P^a(\sigma x) = (-1)^{\e}\sigma \beta^{\e} P^a(x)$ for all $x\in A_q$ (resp. $Sq^a (\sigma x) = \sigma Sq^a(x)$ for all $x\in A_q$).
\end{enumerate}
Finally, $\Sigma^tA_*=\Sigma(\Sigma^{t-1}A_*)$.
\end{dfn}

\begin{exmp}\label{example: cohomotopy of suspension}
Let $C_\bullet \in \sotrcoalg$. In this section we will write $\Sigma C_\bullet\in \sotrcoalg$ for the Kan suspension $\Sigma_\bullet C_\bullet$ of $C_\bullet$, see~\cite[Section~III.5]{GoerssJardine}. Then we have
$$\widetilde{\pi}^*(\Sigma C_\bullet) \cong \Sigma \widetilde{\pi}^*(C_\bullet) $$
as unstable $\calA_p^h$-algebras.
\end{exmp}

\begin{dfn}\label{definition: unstable module}
An $\mathcal{A}_p^h$-module $M_*$ is called \emph{unstable} if $\beta^{\e}P^a(x)=0$ for all $x\in M_q$ and $2a+\e> q$ (resp. $Sq^a(x)=0$ for all $x\in M_q$ and $a> q$). We denote by $\calM^h$ the full abelian subcategory of $\Mod_{\mathcal{A}_p^h}$ spanned by unstable modules.
\end{dfn}

\begin{dfn}\label{definition: strongly unstable module}
An $\mathcal{A}_p^h$-module $M_*$ is called \emph{strongly unstable} if $\beta^{\e}P^a(x)=0$ for all $x\in M_q$ and $2a+\e\geq q$ (resp. $Sq^a(x)=0$ for all $x\in M_q$ and $a\geq q$). We denote by $\calM^h_0$ the full abelian subcategory of $\calM^h$ spanned by strongly unstable modules.
\end{dfn}

\begin{rmk}\label{remark: unstable algebras and unstable modules}
Any unstable $\mathcal{A}_p^h$-algebra is an unstable $\mathcal{A}_p^h$-module. Moreover, any strongly unstable $\calA^h_p$-module is a commutative group object in $\calU^h$ and the category $\calM^h_0$ is precisely the full subcategory of $\calU^h$ spanned by those.
\end{rmk}

The next proposition is standard.

\begin{prop}\label{proposition: unstable are monadic}
The categories $\calU^h$, $\calM^h$, and $\calM^h_0$ are monadic over the category $\Vect_{\kk}^{>0}$ of positively graded vector spaces. \qed
\end{prop}
More precisely, the last proposition means that forgetful functors 
$$\oblv_{\calU^h}\colon \calU^h \to \Vect_{\kk}^{>0}, \;\;\oblv_{\calM^h}\colon \calM^h \to \Vect_{\kk}^{>0}, \;\; \text{and} \;\; \oblv_{\calM^h_0}\colon\calM^h_0 \to \Vect_{\kk}^{>0}$$
have respectively left adjoints
\begin{equation}\label{equation: free unstable}
F_{\mathcal{U}^h} \colon \Vect_{\kk}^{>0} \to \calU^h, \;\; F_{\calM^h}\colon \Vect_{\kk}^{>0} \to \calM^h, \;\; \text{and} \;\; F_{\mathcal{M}^h_0}\colon \Vect_{\kk}^{>0} \to \calM^h_0 
\end{equation}
such that $\calU^h$ is equivalent to the category ${\mathbb{T}_{\mathcal{U}^h}}\mhyphen\Alg$ of algebras over the monad $\mathbb{T}_{\mathcal{U}^h}=\oblv_{\calU^h}\circ F_{\mathcal{U}^h}$, $\calM^h$ is equivalent to the category ${\bbT_{\calM^h}}\mhyphen \Alg$ of algebras over the monad $\bbT_{\calM^h}=\oblv_{\calM^h}\circ F_{\calM^h}$, and $\calM^h_0$ is equivalent to the category ${\mathbb{T}_{\mathcal{M}^h}}\mhyphen\Alg$ of algebras over the monad $\mathbb{T}_{\mathcal{M}^h_0}=\oblv_{\calM^h_0}\circ F_{\mathcal{M}^h_0}$

\begin{rmk}\label{remark: free unstable algebra generated by unstable module}
By Proposition~\ref{proposition: unstable are monadic}, the forgetful functor $\oblv\colon \calU^h \to \calM^h$ has a left adjoint $$\mathcal{F}\colon \calM^h \to \calU^h. $$
The functor $\mathcal{F}$ can given by the formula
$$\mathcal{F}(M_*)= \kk[M_*]/(m^p-P^a(m)\;|\; m\in M_{2a}, a> 0) $$
(resp. $\kk[M_*]/(m^2-Sq^a(m)\;|\; m\in M_a,a>0)$), where $\kk[M_*]$ is the free (non-unital) graded commutative algebra generated by $M_*$.
\end{rmk}

\begin{rmk}\label{remark: free unstable, admissible}
Recall that a (possibly void) sequence $I=(i_1,\ldots, i_k)$ is called \emph{admissible} if $i_j\geq p i_{j+1}, 1 \leq j\leq k-1$, see~\cite[Section~4]{Priddy73}. The \emph{excess} of $I$, denoted by $e(I)$, is defined by
$$e(I)=i_1 -(p-1)(i_2+\ldots i_k), \;\;\; e(\emptyset)=-1. $$
For a sequence $I$, we set $St^I=St^{i_1}\cdot \ldots \cdot St^{i_k}\in \mathcal{A}_p^h$ be a monomial in $\mathcal{A}_p^h$, where
$$
St^i=\left\{
\begin{array}{ll}
Sq^a & \mbox{if $p=2$ and $i=a$,}\\
P^a & \mbox{if $p>2$ and $i=2a(p-1),$}\\
\beta P^a & \mbox{if $p>2$ and $i=2a(p-1)+1$,}\\
0 & \mbox{otherwise.}
\end{array}
\right.
$$

Then the free unstable $\mathcal{A}_p^h$-module $F_{\mathcal{M}^h}(\iota_l)$ generated by a single element $\iota_l$ of degree~$l$ is the vector space  $\kk(St^I \iota_l)$ spanned by the elements $St^I \iota_l$, where $I$ is any admissible sequence of positive integers with $e(I)\leq (p-1)l$; whereas, the free strongly unstable $\mathcal{A}_p^h$-module $F_{\calM^h_0}(\iota_l)$ is the vector space spanned by all elements $St^I\iota_l$ with $e(I)<(p-1)l$. 

Similarly, the free unstable $\mathcal{A}_p^h$-algebra $F_{\mathcal{A}^h}(\iota_l)\cong \kk[St^I \iota_l]$ is the free graded commutative (non-unital) algebra generated by the same elements $St^I\iota_l$ with $e(I)<(p-1)l$.
\end{rmk}

\begin{dfn}\label{definition: vector space of finite type}
A graded vector space $V_{*}=\bigoplus_{q\geq 0}V_q$ is called \emph{of finite type} if $V_0=0$ and $\dim(V_q)<\infty$ for $q\geq 1$. A simplicial vector space $V_\bullet\in \sVect$ is called \emph{of finite type} if $\pi_*(V_\bullet)$ is a graded vector space of finite type.
\end{dfn}

We write $\Vect_{\kk}^{ft}$ for the category of graded vector spaces of finite type; and we write $\sVect_{\kk}^{ft}$ for the category of simplicial vector spaces of finite type.

\begin{thm}[Priddy]\label{theorem: cohomotopy of free truncated is free unstable}
Let $V_\bullet \in \sVect^{ft}_{\kk}$ be a simplicial vector space of finite type. There is a natural isomorphism of unstable $\mathcal{A}_p^h$-algebras:
$$\widetilde{\pi}^*(\Sym^{tr}(V_\bullet)) = \bigoplus_{q>0} \pi^q(\Sym^{tr}(V_\bullet))\cong F_{\mathcal{A}}(\pi^*(V_\bullet)). $$
Here $\Sym^{tr}(V_\bullet)$ is the free simplicial truncated coalgebra generated by $V_\bullet$, see Proposition~\ref{proposition: cofree truncated coalgebra}.
\end{thm}

\begin{proof}
See~\cite[Proposition~6.2.1]{Priddy73}.
\end{proof}

\subsection{Bousfield-Kan spectral sequence}\label{section: BKSS} Let $\mathbb{T}\colon \mathsf{C} \to \mathsf{C}$ be a monad on a category $\mathsf{C}$. %Recall that the ``non-abelian" $\Ext$-group $$\Ext^{s}_{\mathbb{T}}(A,A')\in \Vect_{\kk}, A,A'\in \Alg_{\mathbb{T}}$$ 
%is the \emph{right non-abelian derived functor} of the $\Hom$-functor
%$$\Hom_{\mathbb{T}}(-,A')\colon \Alg_{\bbT}^{op} \to \Vect_{\kk}. $$
The monad $\bbT$ induces the adjoint pair
\begin{equation*}
\begin{tikzcd}
F_{\bbT}: \mathsf{C} \arrow[shift left=.6ex]{r}
&{\bbT}\mhyphen\Alg(\mathsf{C}) :\oblv_{\bbT} \arrow[shift left=.6ex,swap]{l}
\end{tikzcd}
\end{equation*}
such that $\bbT=\oblv_\bbT\circ F_{\bbT}$. Given a $\bbT$-algebra $A\in \bbT\mhyphen\Alg(\mathsf{C})$, we denote by $\bbT_{\bullet}(A) \in \bbT\mhyphen\sAlg(\mathsf{C})$ the \emph{bar-construction} $B_\bullet(F_{\bbT},\bbT,A)$, i.e. $\bbT_\bullet(A)$ is an (almost-free) simplicial $\bbT$-algebra such that 
$$\bbT_q(A)=F_{\bbT}\circ \bbT^{\circ q} \circ \oblv_{\bbT}(A),\; q\geq 0.$$ Similarly, let $\bbR\colon \mathsf{C} \to \mathsf{C}$ be a comonad which induces the adjoint pair $\oblv_{\bbR}\dashv C_{\bbR}$. Given a $\bbR$-coalgebra $C\in \bbR\mhyphen \CoAlg(\mathsf{C})$, we denote by $\bbR^{\bullet}(C) \in \bbR \mhyphen \ccoalg(\mathsf{C})$ the \emph{cobar-construction} $C^{\bullet}(C_{\bbR},\bbR,C)$, i.e. $\bbR^\bullet(C)$ is an (almost-cofree) cosimplicial $\bbR$-coalgebra such that 
$$\bbR^q(C)=C_{\bbR}\circ \bbR^{\circ q} \circ \oblv_{\bbR}(C),\; q\geq 0. $$

Suppose now the category $\mathsf{C}$ is $\kk$-linear. Then $$\Hom_{\bbT}(\bbT_{\bullet}(A),A'), \; A,A'\in {\bbT}\mhyphen\Alg(\mathsf{C})$$ is a cosimplicial vector space and we define the $\Ext$-group $\Ext^s_{\bbT}(A,A')$ as the $s$-th cohomotopy group of that, i.e.
$$\Ext^s_{\bbT}(A,A')=\pi^s\Hom_{\bbT}(\bbT_{\bullet}(A),A'), \; A,A' \in {\bbT}\mhyphen\Alg(\mathsf{C}).$$
In other words, $\Ext^s_{\bbT}(A,A')$ is the \emph{right non-abelian derived functor} of the $\Hom$-functor
$$\Hom_{\mathbb{T}}(-,A')\colon ({\bbT}\mhyphen\Alg(\mathsf{C}))^{op} \to \Vect_{\kk}. $$
We refer the reader to~\cite{DoldPuppe}, \cite{Andre67} for details on non-abelian derived functors and  ``non-abelian" $\Ext$-groups. We also recommend the appendix in~\cite{Bousfield_nice} for a short exposition of the topic.

\begin{dfn}\label{definition: unstable exts}
Let $A_*,A'_*\in \calU^h$ be unstable $\calA_p^h$-algebras. We define the \emph{$s$-th unstable $\Ext$-group} $\Ext^s_{\calU^h}(A_*,A'_*)$ be the next formula
$$\Ext^s_{\calU^h}(A_*,A'_*) = \Ext^s_{\bbT_{\calU^h}}(A_*,A'_*), $$
where $\bbT_{\calU^h}=\oblv_{\calU^h} \circ F_{\calU^h}$ is the monad which defines unstable $\calA_p^h$-algebras, see~\eqref{equation: free unstable}. Similarly, 
$$\Ext^s_{\calM^h}(M_*,M'_*) = \Ext^s_{\bbT_{\calM^h}}(M_*,M'_*), \; M_*,M'_*\in \calM^h, $$
and
$$\Ext^s_{\calM^h_0}(M_*,M'_*) = \Ext^s_{\bbT_{\calM^h_0}}(M_*,M'_*), \; M_*,M'_*\in \calM^h_0. $$
\end{dfn}

\begin{rmk}\label{remark: unstable ext modules}
The categories $\calM^h$, $\calM^h_0$ are abelian and they have enough projectives. Therefore unstable $\Ext$-groups $\Ext^{s}_{\calM^h}(M_*,M'_*)$ and $\Ext^{s}_{\calM^h_0}(N_*,N'_*)$ can be computed by the following formulas
$$\Ext^{s}_{\calM^h}(M_*,M'_*)\cong H^s\big(\Hom_{\calM^h}(P_\bullet(M_*),M'_*)\big), \; M_*,M'_* \in \calM^h $$
and 
$$\Ext^{s}_{\calM^h_0}(N_*,N'_*)\cong H^s\big(\Hom_{\calM^h_0}(\widetilde{P}_\bullet(N_*),N'_*)\big), \; N_*,N'_* \in \calM^h_0, $$
where $P_\bullet(M_*) \to M_*$ is a \emph{projective} resolution of $M_*$ in $\calM^h$, and $\widetilde{P}_\bullet(N_*)\to N_*$ is a projective resolution of $N_*$ in $\calM^h_0$.
\end{rmk}

By Theorem~\ref{theorem:modelsotrcoalg}, the category $\sotrcoalg$ of reduced simplicial truncated coalgebras has a simplicial model structure. Therefore the \emph{derived mapping space} $$\map_{\sCA_0}(C_\bullet, D_\bullet) \in \sSet_*, \; C_\bullet, D_\bullet \in \sotrcoalg$$ is defined, see~\cite[Section~17]{Hirschhorn03}. We recall that $\map_{\sCA_0}(C_\bullet, D_\bullet)$ is a pointed simplicial Kan complex which is defined by
$$\map_{\sCA_0}(C_\bullet, D_\bullet) = \Map_{\sotrcoalg}(C_\bullet,RD_\bullet), $$
where $\Map_{\sotrcoalg}(C_\bullet, RD_\bullet)$ is a simplicial mapping set in the simplicial category $\sotrcoalg$ and $D_\bullet \to RD_\bullet$ is a fibrant replacement of $D_\bullet$. We point out that the derived mapping space is well-defined up to a weak equivalence and preserves weak equivalences in both variables.

\begin{thm}[Bousfield-Kan]\label{theorem: BKSS, coalgebras}
Let $C_\bullet, D_\bullet\in \sotrcoalg$ be reduced simplicial truncated coalgebras of finite type. Then there is a completely convergent spectral sequence
$$E^2_{s,t}=\Ext^s_{\calU^h}(\widetilde{\pi}^*(D_\bullet), \Sigma^t \widetilde{\pi}^*(C_\bullet)) \Rightarrow \pi_{t-s}\map_{\sCA_0}(C_\bullet,D_\bullet). $$
Here $d_r\colon E^r_{s,t}\to E^r_{s+r,t+r-1}$.
\end{thm}

\begin{proof}
Recall from Proposition~\ref{proposition:category property of trcoalg} that the category $\trcoalg$ is comonadic over $\Vect_{\kk}$; we denote by $$\bbR\colon \Vect_{\kk} \to \Vect_{\kk}$$ the resulting comonad $\bbR=\oblv\circ \Sym^{tr}$. As $\bbR(0)=0$, we extend $\bbR$ degreewise to the comonad
$$\bbR\colon \soVect_{\kk} \to \soVect_{\kk} $$
on the category of reduced simplicial vector spaces. Then we have $$\CoAlg_{\bbR}(\soVect_{\kk})\cong \sotrcoalg;$$ and for $D_\bullet \in \sotrcoalg$ we consider the cobar-construction 
$$\bbR^{\bullet}(D_\bullet) = \Sym^{tr}\circ \bbR^{\circ \bullet} \circ \oblv (D_\bullet) \in \csotrcoalg.$$
There is a natural map $D_\bullet \to \bbR^\bullet(D_\bullet)$ which induces a weak equivalence:
$$D_\bullet \xrightarrow{\simeq} \Tot \bbR^{\bullet}(D_\bullet) \in \sotrcoalg, $$
where $\Tot \bbR^{\bullet}(D_\bullet) \in \sotrcoalg$ is the totalization of the cosimplicial object $\bbR^{\bullet}(D_\bullet)$, see~\cite[Definition~18.6.3]{Hirschhorn03}. Thus we obtain weak equivalences of derived mapping spaces:
$$\map_{\sCA_0}(C_\bullet, D_\bullet) \xrightarrow{\simeq} \map_{\sCA_0}(C_\bullet, \Tot \bbR^{\bullet}(D_\bullet))\xrightarrow{\simeq }\Tot \map_{\sCA_0}(C_\bullet, \bbR^{\bullet}(D_\bullet)).  $$
By~\cite[Proposition~VIII.1.15]{GoerssJardine} (see also~\cite[Chapter~X]{BK_book}), there is a spectral sequence
\begin{align}\label{equation: BKSS, eq1}
E^2_{s,t}=\pi^s\pi_t\map_{\sCA_0}(C_\bullet,\bbR^{\bullet}(D_\bullet)) &\Rightarrow \pi_{t-s} \Tot \map_{\sCA_0}(C_\bullet, \bbR^{\bullet}(D_\bullet)) \\
&\cong \pi_{t-s}\map_{\sCA_0}(C_\bullet,D_\bullet) \nonumber
\end{align}
associated with the cosimplicial simplicial set $\map_{\sCA_0}(C_\bullet, \bbR^{\bullet}(D_\bullet)) \in \csSet_*$. 

Next, we will compute the second page of the spectral sequence~\eqref{equation: BKSS, eq1}. Namely, we show that there is an isomorphism
\begin{equation}\label{equation: BKSS, eq2}
\pi^s\pi_t\map_{\sCA_0}(C_\bullet,\bbR^{\bullet}(D_\bullet)) \cong \Ext^s_{\calU^h}(\widetilde{\pi}^*(D_\bullet), \Sigma^{t}\widetilde{\pi}^*(C_\bullet)). 
\end{equation}
Indeed, we have following isomorphisms of cosimplicial vector spaces
\begin{align*}
\pi_t\map_{\sCA_0}(C_\bullet, \bbR^{\bullet}(D_\bullet))&\cong \pi_0\Map_{\sotrcoalg}(\Sigma^t C_\bullet,\Sym^{tr}\circ \bbR^{\circ \bullet}\circ \oblv_{\bbR}(D_\bullet)) \\
&\cong \pi_0\Map_{\soVect_{\kk}}(\oblv_{\bbR}(\Sigma^{t}C_\bullet), \bbR^{\circ\bullet} \circ \oblv_{\bbR}(D_\bullet)) \\
&\cong \Hom_{\Vect^{gr}_{\kk}}(\widetilde{\pi}_*(\Sigma^t C_\bullet),\widetilde{\pi}_*(\bbR^{\circ \bullet}\circ \oblv_{\bbR}(D_\bullet)))\\
&\cong \Hom_{\Vect^{gr}_{\kk}}(\widetilde{\pi}^*(\bbR^{\circ \bullet} \circ \oblv_{\bbR}(D_\bullet)), \widetilde{\pi}^*(\Sigma^t C_\bullet))\\
&\cong \Hom_{\Vect^{gr}_{\kk}}(\bbT_{\calU^h}^{\circ \bullet}\circ \oblv_{\calU^h}(\widetilde{\pi}^*(D_\bullet)),\oblv_{\calU^h}( \widetilde{\pi}^*(\Sigma^t C_\bullet)))\\
&\cong \Hom_{\bbT_{\calU^h}}(F_{\calU^h}\circ \bbT^{\circ \bullet}_{\calU^h}\circ \oblv_{\calU^h}(\widetilde{\pi}^*(D_\bullet)), \widetilde{\pi}^*(\Sigma^t C_\bullet)) \\
&= \Hom_{\bbT_{\calU^h}}(\bbT_{\calU^h, \bullet}(\widetilde{\pi}^*(D_\bullet)), \widetilde{\pi}^*(\Sigma^t C_\bullet)) \\
&\cong \Hom_{\bbT_{\calU^h}}(\bbT_{\calU^h, \bullet}(\widetilde{\pi}^*(D_\bullet)), \Sigma^t \widetilde{\pi}^*(C_\bullet))
\end{align*}
Here the fourth isomorphism follows from the assumption that simplicial coalgebras $C_\bullet, D_\bullet \in \sotrcoalg$ are of finite type, the fifth isomorphism follows from Theorem~\ref{theorem: cohomotopy of free truncated is free unstable}, and the last one follows from Example~\ref{example: cohomotopy of suspension}; all other isomorphisms are induced by various adjunctions. By Definition~\ref{definition: unstable exts}, this implies the isomorphism~\eqref{equation: BKSS, eq2}.

Finally, the spectral sequence~\eqref{equation: BKSS, eq1} converges completely by applying the complete convergence lemma, see~\cite[Lemma~VI.2.20]{GoerssJardine} and~\cite[IX.5.4]{BK_book}. Indeed, by Remark~\ref{remark: free unstable, admissible}, all entries $E^2_{s,t}$ on the second page are finite dimensional vector spaces, and so 
\begin{equation*}
{\lim_r}^{1} E^r_{s,t}=0, \; t-s\geq 1.  \qedhere
\end{equation*}
\end{proof}

\begin{dfn}\label{definition: lie algebra, finite type}
A simplicial restricted Lie algebra $L_\bullet$ is called \emph{of finite type} if its homology groups $\widetilde{H}_*(L_\bullet;\kk)$ is a graded vector space of finite type.
\end{dfn}

\begin{cor}\label{corollary: BKSS, restricted Lie algebra}
Let $L_\bullet, L'_\bullet \in \srLie$ be $\kk$-complete simplicial restricted Lie algebras of finite type. Then there is a completely convergent spectral sequence
$$E^2_{s,t}=\Ext^s_{\calU^h}(\widetilde{H}^*(L_\bullet;\kk), \Sigma^t \widetilde{H}^*(L'_\bullet;\kk)) \Rightarrow \pi_{t-s}\map_{\sLxi}(L'_\bullet,L_\bullet). $$
Here $d_r\colon E^r_{s,t}\to E^r_{s+r,t+r-1}$.
\end{cor}
In particular, there is a completely converging spectral sequence
\begin{equation}\label{equation: ASS, Lie}
E^2_{s,t}=\Ext^s_{\calU^h}(\widetilde{H}^*(L_\bullet;\kk), \Sigma^{t+1} \kk) \Rightarrow \pi_{t-s}(L_\bullet)
\end{equation}
by taking $L'_\bullet = L_\xi(\free(\kk))$ in the spectral sequence of Corollary~\ref{corollary: BKSS, restricted Lie algebra}.

\begin{proof}
By Theorem~\ref{theorem: coalgebras and lie algebras}, there is a weak equivalence of derived mapping spaces
$$\map_{\sLxi}(L'_\bullet, L_\bullet) \simeq \map_{\sCA_0}(\barW U^r(L'_\bullet), \barW U^r(L_\bullet)). $$
By Example~\ref{example: steenrod algebra, restricted Lie} and Theorem~\ref{theorem: BKSS, coalgebras}, we obtain the required spectral sequence.
\end{proof}

\subsection{Unstable Koszul resolutions}\label{section: unstable koszul}
Recall that the \emph{lambda algebra} $\Lambda$ is the associative bigraded algebra over $\F_p$ generated by the elements $\lambda_a, a\geq 1$ of bidegree $|\lambda_a|=(2a(p-1)-1,1)$ and $\mu_a, a\geq 0$ of bidegree $|\mu_a|=(2a(p-1),1)$ (resp. $\lambda_a, a\geq 0$ of bidegree $|\lambda_a|=(a,1)$) subject to the following Adem-type relations:
\begin{enumerate}
\item If $p$ is odd, $b\geq pa$, and $\e=0,1$, then
\begin{align}\label{equation: lambda, adem1, p is odd}
\lambda_a\nu^{\e}_{b} &= \sum_{i=0}^{a+b}(-1)^{i+a+\e}\binom{(p-1)(b-i)-\e}{i-pa}\nu^{\e}_{a+b-i}\lambda_i \\
&+(1-\e)\sum_{l\geq 0}^{a+b}(-1)^{i+a+1}\binom{(p-1)(b-i)-1}{i-pa}\lambda_{a+b-i}\mu_i. \nonumber 
\end{align}
\item If $p$ is odd, $b>pa$, and $\e=0,1$, then
\begin{equation}\label{equation: lambda, adem2, p is odd}
\mu_a\nu^{\e}_{b} = \sum_{i=1}^{a+b}(-1)^{i+a}\binom{(p-1)(b-i)-1}{i-pa-1}\mu_{a+b-i}\nu^{\e}_i.
\end{equation}
\item If $p=2$ and $b>2a$, then 
\begin{equation}\label{equation: lambda, adem3, p=2}
\lambda_a\lambda_{b} = \sum_{i=1}^{a+b}\binom{b-i-1}{i-2a-1}\lambda_{a+b-i}\lambda_i. 
\end{equation}
\end{enumerate}
Here we set $\nu_a^{0}=\mu_a, a\geq 0$ and $\nu_a^1=\lambda_a,a>0$. Notice that we use the definition of the lambda algebra from~\cite[Definition~7.1]{Wellington82}, but not from the original paper~\cite{6authors}. 

In this section we compute \emph{abelian} unstable $\Ext$-groups
$$\Ext^s_{\calM^h}(W, \Sigma^{t}\kk) \;\; \text{(resp. $\Ext^s_{\calM_0^h}(W,\Sigma^t\kk)$)} $$
in terms of the algebra $\Lambda$. Here $W\in \Vect^{gr}_{\kk}$ is a graded vector space, which is considered as a left $\calA^h_p$-module equipped with the trivial action. In order to calculate these $\Ext$-groups, we construct a free resolution $K_\bullet(W) \in \calM^h$ (resp. $K^0_\bullet(W) \in \calM^h_0$) of the (resp. strongly) unstable $\calA^h_p$-module $W$, see Remark~\ref{remark: unstable ext modules}.

The homogenized Steenrod algebra $\calA_p^h$ has a Poincar\'{e}-Birkhoff-Witt (PBW) basis over $\F_p$ given by admissible monomials $$B=\{St^{i_1}\cdot \ldots  \cdot St^{i_k} \; | \; i_{j}\geq pi_{j+1}, k\geq 1\},$$ see Remark~\ref{remark: free unstable, admissible}. We refer the reader to~\cite[Section~5]{Priddy70} and~\cite[Chapter~4]{PP05} for detailed accounts on algebras with a PBW basis. Therefore $\calA_p^h$ is a Koszul algebra with respect to the weight grading, see~\cite[Theorem~5.3]{Priddy70} and~\cite[Theorem~4.3.1]{PP05}.

We denote by $\calK^*_p = \Ext^*_{\calA_p^h}(\F_p,\F_p)$ the Koszul dual algebra for $\calA_p^h$. Here the star in $\calK^*_p$ stands for the weight grading. We recall from~\cite[Sections~7.1-7.2]{Priddy70} that $\calK^*_p$ is a bigraded algebra generated by the elements $$Pr_i \in \calK^1_p, \; |Pr_i|=(i,1),$$ where $Pr_i$ is dual to $St^i$, $i>0$. Notice that S.~Priddy used a different notation in~\cite{Priddy70}: if $p=2$, $\sigma_a$ is the dual to $Sq^a$; if $p$ is odd, $\pi_a$ is the dual to $P^a$ and $\rho_b$ is the dual to $\beta P^b$. In short,
$$Pr_i=\left\{
\begin{array}{ll}
\sigma_{a} & \mbox{if $p=2$ and $i=a, a>0$}\\
\pi_{a} & \mbox{if $p>2$ and $i=2a(p-1), a>0$}\\
\rho_{a} & \mbox{if $p>2$ and $i=2a(p-1)+1, a\geq 0$,} \\
0,  & \mbox{otherwise.}
\end{array}
\right.
$$
Furthermore, there is an anti-isomorphism 
\begin{equation}\label{equation: koszul to steenrod is Lambda, p is odd}
\Phi_p\colon \calK^*_p \to \Lambda
\end{equation}
given by 
$$\Phi_p(Pr_i)=\left\{
\begin{array}{ll}
\lambda_{a-1} & \mbox{if $p=2$ and $i=a, a>0$}\\
\lambda_{a} & \mbox{if $p>2$ and $i=2a(p-1), a>0$}\\
\mu_{a} & \mbox{if $p>2$ and $i=2a(p-1)+1, a\geq 0$.}
\end{array}
\right.
$$
The algebras $\calK^*_p$ and $\Lambda$ are bigraded, however, the map $\Phi_p$ does not preserve the bidegree: if $|x|=(m,n), x\in \calK^*_p$, then $|\Phi_p(x)|=(m-n,n)$.

We say that a sequence $J=(a_1,\ldots, a_k)$ is \emph{orthogonally admissible} if $$a_j < pa_{j+1}, \; 1\leq j\leq k-1.$$ The Koszul dual algebra $\calK^*_p$ has a PBW basis given by orthogonally admissible monomials:
$$K(B)=\{Pr_{J}=Pr_{a_1}\cdot \ldots \cdot Pr_{a_k}\; | \; J=(a_1,\ldots, a_k)\;\text{is orthogonally admissible} \}. $$
We denote by $\calK_{p,*}$ the linear dual to $\calK^*_p$. The graded vector space $\calK_{p,*}$ is spanned by
$$K^\vee(B)=\{Pr^J\;|\; J=(a_1,\ldots, a_k)\;\text{is orthogonally admissible}\}, $$
where $Pr^J\in \calK_{p,*}$ is dual to $Pr_J\in \calK^*_p$. 

Finally, we recall from~\cite[Section~2.3]{PP05} that the trivial $\calA_p^h$-module $\F_p$ has the \emph{Koszul resolution} $K_\bullet$ by free $\calA_p^h$-modules:
\begin{align}\label{equation: koszul, p is odd, stable}
K_\bullet = (\ldots \to \calA_p^h \otimes \calK_{p,3} \xrightarrow{d_3} \calA_p^h \otimes \calK_{p,2} &\xrightarrow{d_2} \calA_p^h \otimes \calK_{p,1} \xrightarrow{d_1} \calA_p^h \to \F_p \to 0).
\end{align}
Here the differential $d_s$ is given by:
\begin{align}\label{equation: koszul diffrential, p is odd}
d_s(Pr^J)= \sum_{i\geq 1} St^{i}\otimes (Pr^{J}\cdot Pr_i),
\end{align}
where $J=(a_1,\ldots, a_s)$ is an orthogonally admissible sequence and we consider $\calK_{p,*}$ as a right $\calK^*_p$-module.

We now construct an (resp. strongly) unstable analog of the Koszul resolution~\eqref{equation: koszul, p is odd, stable}. %The $K^*(\calA^h_p)$-module $\Sigma^{l}K^*(\calA^h_p)$ has a basis $$Pr_J \iota_{l},\;\; Pr_J \in K(B),$$
%where $\iota_{l}\in (\Sigma^{l}K^*(\calA^h_p))_{l,0}= \kk$ is the generator which corresponds to the unit in $\kk$.
Let $J=(i_1,\ldots,i_k)$ be an orthogonally admissible sequence. We define the \emph{excess} $e(J)$ of $J$ as follows:
$$e(J)=e(i_1,\ldots,i_k)=i_k \;\;\text{if $k>0$ and}\;\; e(\emptyset)=0. $$

\begin{lmm}\label{lemma: koszul differential and excess, p is odd}
Suppose that $J$ is orthogonally admissible sequence and $i\geq 1$. Then 
$$Pr^J\cdot Pr_a = \sum_{J'}c_{J'}Pr^{J'}, \; c_{J'}\neq 0\in\F_p,$$
where all $J'$ are orthogonally admissible and $e(J')\geq e(J)$. \qed
\end{lmm}

\begin{proof}
Recall that the Koszul dual algebra $\calK^*_p$ is quadratic and all relations have a form
$$Pr_iPr_j = \sum_{(i',j')}c_{i',j'} Pr_{i'}Pr_{j'}, \; i<pj, \; c_{i',j'}\neq 0 \in \F_p, $$
where $j'>j$ and $i'\geq pj'$, see~\cite[Section~7]{Priddy70}. In other words, each sequence $(i',j')$ succeeds $(i,j)$ in the reverse lexicographical order.

Now, let $I=(i_1,\ldots,i_k)$ be any sequence. By an inductive argument and the previous paragraph, we observe that
$$Pr_I=Pr_{i_1}\cdot \ldots \cdot Pr_{i_k}=\sum_{I'}c_{I'}Pr_{I'}, \; c_{I'}\neq 0\in\F_p,$$
where all sequences $I'$ are orthogonally admissible and each $I'$ either succeeds or equal $I$ with respect to the reverse lexicographical order. This implies the lemma.
\end{proof}

\begin{dfn}\label{definition: unstable. koszul dual, p is odd}
%Let $l\geq 0$. We denote by $K^*_l(\calA^h_p)$ (resp. $K^*_{l,0}(\calA^h_p)$) the bigraded vector subspace of $\Sigma^{l}K^*(\calA^h_p)$ spanned by elements
%$$Pr_J\iota_{l}, \;\; Pr_{J} \in K(B)\;\; \text{and} \;\; e(J)\leq (p-1)l \;\; \text{(resp. $e(J)<(p-1)l$)}. $$
%We denote by $K^l_*(\calA^h_p)$ (resp. $K^{l,0}_*(\calA^h_p)$) the linear dual of $K^*_l(\calA^h_p)$ (resp. $K^*_{l,0}(\calA^h_p)$); $K^l_*(\calA^h_p)$ (resp. $K^{l,0}_*(\calA^h_p)$) has a basis
%$$Pr^{J} \iota_{l}, \;\; Pr^{J} \in K^{\vee}(B)\;\; \text{and} \;\; e(J)\leq (p-1)l \;\; \text{(resp. $e(J)<(p-1)l$)}. $$
Let $W \in \Vect^{gr}_{\kk}$ be a graded vector space. We denote by $\calK^*_p\widehat{\otimes} W$ (resp. $\calK^*_p\widetilde{\otimes} W$) the vector subspace of $\calK^*_p \otimes W$ spanned by elements
$$Pr_J\otimes w, \;\; Pr_{J} \in K(B), \; w\in W, $$ 
where $e(J)\leq (p-1)|w|$ (resp. $e(J)<(p-1)|w|$)). Dually, we denote by 
$\calK_{p,*}\widehat{\otimes} W$ (resp. $\calK_{p,*}\widetilde{\otimes} W$) the subspace of $\calK_{p,*} \otimes W$ spanned by elements
$$Pr^J\otimes w, \;\; Pr^{J} \in K^\vee(B), \; w\in W, $$ 
where $e(J)\leq (p-1)|w|$ (resp. $e(J)<(p-1)|w|$)). We point out that $|Pr_J\otimes w| = |Pr^J\otimes w| = (|J|+|w|,l(J))$.
\end{dfn}

\begin{rmk}\label{remark: unstable koszul and lambda, p=2}
If $p=2$, then we have $$\calK^*_p\widetilde{\otimes} \Sigma^l \kk =\Sigma \calK^*_{p} \widehat{\otimes} \Sigma^{l-1} \kk, \; l>0.$$ %Furthermore, the isomorphism $\Phi_2$ from~\eqref{equation: koszul to steenrod is Lambda, p is odd} maps the vector subspace 
%$$\Sigma^{-l}\calK^*_p \widehat{\otimes} \Sigma^l\kk \subset \calK^*_p$$ onto $\Lambda(l) \subset \Lambda$ of~\cite[p.~459]{BC70}.
%\end{rmk}
%\begin{rmk}%\label{remark: unstable koszul and lambda, p is odd}
Similarly, if $p$ is odd and $l=2n+1$ is odd, then we have 
$$\calK^*_{p} \widetilde{\otimes}\Sigma^l \kk =\calK^*_{p}\widehat{\otimes} \Sigma^l \kk, $$
%and the isomorphism $\Phi_p$ maps $$\Sigma^{-l}\calK^*_{p}\widehat{\otimes} \Sigma^l \kk = \Sigma^{-l}\calK^*_{p}\widetilde{\otimes} \Sigma^l\kk$$ onto $\Lambda(2n+1)\subset \Lambda$ of~\cite[(1.16)]{HarperMiller82}. Similarly,
and if $l=2n$ is even, then %$$\Phi_p(\Sigma^{-l}\calK^*_p\widehat{\otimes}\Sigma^l\kk) = \Lambda(2n),$$
%$$\Phi_p(\Sigma^{-l}\calK^*_p \widetilde{\otimes} \Sigma^l \kk) = \Phi_p(\Sigma^{-l+1}\calK^*_p \widehat{\otimes} \Sigma^{l-1} \kk) = \Lambda(2n-1), $$
%where $\Lambda(2n)\subset \Lambda$ is of~\cite[(1.15)]{HarperMiller82}.
$$\calK^*_p \widetilde{\otimes} \Sigma^l \kk = \Sigma\calK^*_p \widehat{\otimes} \Sigma^{l-1} \kk.$$
\end{rmk}

\begin{rmk}\label{remark: subspace in Lambda}
Recall that a monomial of the algebra $\Lambda$ in the $\lambda$'s and $\mu$'s (resp. the $\lambda$'s) is called \emph{admissible} if
\begin{itemize}
\item $p$ is odd and whenever $\lambda_a\lambda_b$ or $\lambda_a\mu_b$ occurs, we have $b<pa$,
\item $p$ is odd and whenever $\mu_a\lambda_b$ or $\mu_a\mu_b$ occurs, we have $b\leq pa$, and
\item $p=2$ and whenever $\lambda_a\lambda_b$ occurs, we have $b\leq pa$.
\end{itemize}
In other words, a monomial $y\in \Lambda$ is admissible if and only if $y=\Phi_p(x)$, $x\in \calK^*_p$ and $x$ is orthogonally admissible. We also recall that the algebra $\Lambda$ is filtered by vector subspaces $\Lambda(l), l\geq 0$, where $\Lambda(l)$ is spanned by the admissible monomials beginning with
\begin{itemize}
\item $\lambda_a$ with $a\leq n$ or $\mu_a$ with $a < n$ if $l=2n$ and $p$ is odd (\cite[(1.15)]{HarperMiller82}),
\item $\lambda_a$ with $a\leq n$ or $\mu_a$ with $a\leq n$ if $l=2n+1$ and $p$ is odd (\cite[(1.16)]{HarperMiller82}),
\item $\lambda_a$ with $a<l$ if $p=2$ (\cite[p. 459]{BC70}).
\end{itemize}
Then we observe that the anti-isomorphism $\Phi_p\colon \calK^*_p \to \Lambda$ maps the vector subspace $$\Sigma^{-l}\calK^*_p\widehat{\otimes}\Sigma^l\kk \subset \calK^*_p$$ onto $\Lambda(l)\subset \Lambda$.
\end{rmk}

\begin{rmk}\label{remark: unstable koszul and lambda, any p}
Let $W\in \Vect^{gr}_{\kk}$ be a graded vector space. Following~\cite[Definition~9.3]{Wellington82}, let us denote by  $W\widehat{\otimes}\Lambda$ the subspace of $W\otimes \Lambda$ spanned by $w\otimes y\in W\otimes \Lambda$ such that $y\in \Lambda$ is an admissible monomial beginning with $\nu_a^\e$ (resp. $\lambda_a$) such that $2a-\e <|w|$ (resp. $a<|w|$). Similarly, $W\widetilde{\otimes}\Lambda$ is the subspace of $W\widehat{\otimes}\Lambda$ spanned by $w\otimes y$, $y\in \Lambda$ is an admissible monomials beginning with $\nu_a^{\e}$ (resp. $\lambda_a$) such that $2a<|w|$ (resp. $a+1<|w|$). Vector spaces $W\widehat{\otimes}\Lambda$ and $W\widetilde{\otimes}\Lambda$ are bigraded:
$$|w\otimes y|=(|w|,0)+|y|, \; w\in W, \; y\in \Lambda. $$

Remarks~\ref{remark: unstable koszul and lambda, p=2} and~\ref{remark: subspace in Lambda} show that the isomorphism $\Phi_p$ defines (after suitable shifts) following isomorphisms
$$\hat\Phi_p\colon \calK^*_p\widehat{\otimes} W \xrightarrow{\cong} W\widehat{\otimes} \Lambda, \;\; \tilde\Phi_p\colon \calK^*_p\widetilde{\otimes} W \xrightarrow{\cong} W\widetilde{\otimes} \Lambda, \;\; W\in \Vect^{gr}_{\kk}. $$
%where right hand sides are defined in~\cite[Definition~9.3]{Wellington82}. 
Furthermore, if $|x|=(m,n)$, $x\in \calK^*_p\widehat{\otimes} W$ (resp. $\calK^*_p\widetilde{\otimes} W$), then $|\hat{\Phi}_p(x)|= |\tilde{\Phi}_p(x)|=(m-n,n)$.
\end{rmk}

We introduce more notation. Let $y=Pr^{J} \in K^{\vee}(B)$ be an orthogonally admissible monomial. We write $y(w), w\in W, W\in \Vect^{gr}_{\kk}$ for the following element of $\calK_{p,*}\widehat{\otimes} W$ (resp. $\calK_{p,*}\widetilde{\otimes} W$):
$$y(w)=\left\{
\begin{array}{ll}
Pr^{J}\otimes w & \mbox{if $e(J)\leq (p-1)|w|$ (resp. $e(J)<(p-1)|w|$),}\\
0, & \mbox{otherwise.}
\end{array}
\right.
$$
We extend this notation linearly to any element $y\in \calK_{p,*}$. Given the differential~\eqref{equation: koszul diffrential, p is odd}, we produce a map of free unstable $\calA^h_p$-modules
\begin{equation}\label{equation: koszul differential, unstable, p is odd}
d_s^{un}\colon F_{\calM^h}(\calK_{p,s}\widehat{\otimes} W^{(s)}) \to F_{\calM^h}(\calK_{p,s-1}\widehat{\otimes} W^{(s-1)})
\end{equation}
given on the generators $Pr^{J}\otimes w$ by the following formula:
\begin{align*}
d_s^{un}(Pr^{J}\otimes w) = \sum_{i\geq 1}St^{i}((Pr^J\cdot Pr_i)(w)).
\end{align*}
Here $St^i(x)=0$ if $i>(p-1)|x|$ and $W^{(s)}$ is the Frobenius twist of $W$, see Definition~\ref{definition: frobenius twist}. The Frobenius twist is necessary since the operations $St^i \in \calA^h_p$ are only semi-linear (Definition~\ref{definition: module over hsa}). We also define by the same formula a map of free strongly unstable $\calA^h_p$-modules:
\begin{equation}\label{equation: koszul differential, strongly unstable, p is odd}
d_s^{sun}\colon F_{\calM^h_0}(\calK_{p,s}\widetilde{\otimes}W^{(s)}) \to F_{\calM^h_0}(\calK_{p,s}\widetilde{\otimes} W^{(s-1)}),
\end{equation}
where $St^i(x)=0$ if $i\geq (p-1)|x|$. By Lemma~\ref{lemma: koszul differential and excess, p is odd} and $d_{s-1}\circ d_s =0 $, we have $$d_{s-1}^{un}\circ d_{s}^{un}=0\;\; \text{(resp. $d_{s-1}^{sun}\circ d_{s}^{sun} =0$)},\;s>0,$$ and so we obtained the \emph{unstable Koszul complex} 
\begin{equation}\label{equation: unstable koszul complex}
K_\bullet(W) = (F_{\calM^h}(\calK_{p,\bullet}\widehat\otimes W^{(\bullet)}), d^{un})
\end{equation}
%\begin{align}\label{equation: koszul complex, unstable, p is odd}
%\ldots\to F_{\calM^h}(K_2^l(\calA_p^h)) &\xrightarrow{d^{un}_2} F_{\calM^h}(K_1^l(\calA_p^h)) \xrightarrow{d^{un}_1} F_{\calM^h}(K_0^l(\calA_p^h)) \to \Sigma^{l}\kk \to 0
%\end{align}
for the unstable $\calA_p^h$-module $W \in \Vect_{\kk}^{gr}\subset \calM^h$. Similarly, we have the \emph{strongly unstable Koszul complex} 
\begin{equation}\label{equation: strongly unstable koszul complex}
K^0_\bullet(W)=(F_{\calM^h_0}(\calK_{p,\bullet}\widetilde\otimes W^{(\bullet)}), d^{sun})
\end{equation}
%\begin{align}\label{equation: koszul complex, strongly unstable, p is odd}
%\ldots\to F_{\calM^h_0}(K_2^{l,0}(\calA_p^h)) &\xrightarrow{d^{sun}_2} F_{\calM^h_0}(K_1^{l,0}(\calA_p^h)) \xrightarrow{d^{sun}_1} F_{\calM^h_0}(K_0^{l,0}(\calA_p^h)) \to \Sigma^{l}\kk \to 0
%\end{align}
for the strongly unstable $\calA_p^h$-module $W \in \Vect_{\kk}^{gr}\subset\calM^h_0$.

\begin{prop}\label{proposition: unstable koszul resolution, p is odd}
Both complexes $K_\bullet(W)$ and $K^0_\bullet(W)$ are acyclic for all $W\in \Vect^{gr}_{\kk}$.
\end{prop}

\begin{proof}
We prove the proposition for $K_\bullet(W)$ only; the argument for $K^0_\bullet(W)$ is almost identical, and we leave it to the reader to complete the details. By Remark~\ref{remark: free unstable, admissible}, the complex $K_\bullet(W)$ has a basis $St^{I}(Pr^J\otimes w),$
where $I=(i_1,\ldots, i_k)$ is an admissible sequence, $J=(a_1,\ldots, a_s)$ is an orthogonally admissible sequence, $e(J)\leq (p-1)|w|$, and $$e(I)\leq (p-1)(|J|+|w|)= (p-1)(a_1+\ldots+a_s+|w|).$$

We denote by $K_\bullet(W)_n, n\geq 0$ the vector subspace of $K_\bullet(W)$ spanned by elements $St^I(Pr^J\otimes w)$ such that $l(I)+l(J)=n$. Note that the differential~\eqref{equation: koszul differential, unstable, p is odd} preserves each $K_{\bullet}(W)_n, n\geq 0$.

Let us consider the lexicographical order on the set $\mathcal{I}_n$ of sequences of length $n$: $\alpha=(\alpha_1,\ldots,\alpha_n) \preceq \beta=(\beta_1,\ldots, \beta_n)\in \mathcal{I}_n$ if and only if there exists $h$ such that $$\alpha_1=\beta_1,\ldots, \alpha_{h-1}=\beta_{h-1}\;\; \text{and} \;\; \alpha_h<\beta_h.$$
Define an $\mathcal{I}_n$-valued decreasing filtration $F^{\alpha}K_\bullet(W)_n, \alpha \in\mathcal{I}_n$ on the vector space $K_\bullet(W)_n$ by the rule:
$$F^{\alpha}K_\bullet(W)_n = \mathrm{span}(St^I(Pr^J\otimes w)\; |\; (I,J)\succeq \alpha), $$
where $(I,J)$ is the concatenation of the sequences $I$ and $J$. By the Adem relations~\eqref{equation: adem relation, p odd, 1}, \eqref{equation: adem relation, p odd, 2}, \eqref{equation: adem relation, p=2} in $\calA_p^h$, and the relations~\cite[Section~7]{Priddy70} in $\calK^*_p$, the differential $d^{un}$ preserves subspaces $F^{\alpha}K_\bullet(W)_n$, $\alpha\in \mathcal{I}_n$.

A straightforward computation also shows that the associated graded complex $gr^{\alpha}K_\bullet(W)_n=F^{\alpha}/F^{\alpha'}$ (where $\alpha'$ succeeds $\alpha$ in $\mathcal{I}_n$) has the induced differential given by the rule 
%$$
%d(St^I(Pr^J\otimes w))=\left\{
%\begin{array}{ll}
%St^I St^{a_1}(Pr^{J_0}\otimes w), & \mbox{if $J=(a_1,J_0)$ and $(I,a_1)$ is admissible,}\\
%0, & \mbox{otherwise.}
%\end{array}
%\right.
%$$
$$d(St^I(Pr^J\otimes w)) = St^I St^{a_1}(Pr^{J_0}\otimes w) $$
provided $J=(a_1,J_0)$ and the sequence $(I,a_1)$ is admissible; and $d(St^I(Pr^J\otimes w))$ is zero otherwise. Hence the complex $gr^{\alpha}K_\bullet(W)_n$ is acyclic for each $\alpha\in \mathcal{I}_n$ and each $n\geq 0$. In this way, the original complex $K_\bullet(W)$ is also acyclic.
\end{proof}

By combining Proposition~\ref{proposition: unstable koszul resolution, p is odd} with Remarks~\ref{remark: unstable ext modules} and~\ref{remark: unstable koszul and lambda, any p}, we obtain immediately the following statement.

\begin{cor}\label{corollary: unstable exts between trivial}
Let $W\in \Vect^{gr}_{\kk}$ be a graded vector space. Then there are natural isomorphisms:
$$\Ext^{s}_{\calM^h}(W,\Sigma^t\kk)\cong (\calK^s_p \widehat{\otimes} (W^*)^{(s)})_t\xrightarrow{\hat{\Phi}_p} ((W^*)^{(s)}\widehat{\otimes} \Lambda_s)_{t-s},\; t\geq s\geq 0, $$
\begin{equation*}
\Ext^{s}_{\calM^h_0}(W,\Sigma^t\kk)\cong (\calK^s_p \widetilde{\otimes} (W^*)^{(s)})_t \xrightarrow{\tilde{\Phi}_p} ((W^*)^{(s)}\widetilde{\otimes} \Lambda_s)_{t-s},\; t\geq s\geq 0. \eqno\qed
\end{equation*}
\end{cor}

\subsection{Free restricted Lie algebra}\label{section: free lie}
In this section we sketch how to compute the homotopy groups $\pi_*\free(V_\bullet)$ of the free simplicial restricted Lie algebra $\free(V_\bullet), V_\bullet\in\sVect_{\kk}^{ft}$ in terms of the algebra $\Lambda$ by using the spectral sequence~\eqref{equation: ASS, Lie}. Of course, the description of $\pi_*\free(V_\bullet)$ is well-known since the seventies (see~\cite[Theorem~8.5]{BC70} and~\cite[Proposition~13.2]{Wellington82}), so our computation is only illustrative.

By Proposition~\ref{proposition: free lie algebras, fresse}, there is an isomorphism
\begin{equation}\label{equation: free restricted, weight}
\pi_*(\free(V_\bullet))\cong \pi_*\Big(\bigoplus_{n\geq 1}L^r_n(V_\bullet)\Big) \cong \bigoplus_{n\geq 1} \pi_*\left((\bfLie_n \otimes V_\bullet^{\otimes n})^{\Sigma_n}\right).
\end{equation}
We write $\pi_{*,n}(\free(V_\bullet))$ for the direct summand $\pi_{*}((\bfLie_n\otimes V_\bullet^{\otimes n})^{\Sigma_n})$ in $\pi_*(\free(V_\bullet))$.

%\begin{dfn}\label{definition: isotropic Lie algebra}
%A Lie algebra $L$ over $\kk$ is called \emph{isotropic} if $[x,x]=0$ for all $x\in L$. Let $\Lie$ denote the category of isotropic Lie algebras over the field $\kk$. Note that if $p\neq 2$, then any Lie algebra is isotropic.
%\end{dfn}

Recall that $\Lie$ is the category of Lie algebras over $\kk$, see Section~\ref{section: notation}.

\begin{lmm}\label{lemma: isotropic Lie algebras}
The forgetful functor $\oblv\colon \Lie \to \Vect_{\kk}$ admits a left adjoint
$$L\colon \Vect_{\kk} \to \Lie $$
such that $$\oblv \circ L(V) \cong \bigoplus_{n\geq 1} L_n(V), \;\; V\in \Vect_{\kk}$$ and $L_n(V)$ is the image of $(\bfLie_n\otimes V^{\otimes n})_{\Sigma_n}$ under the norm map
$$(\bfLie_n \otimes V^{\otimes n})_{\Sigma_n} \to (\bfLie_n \otimes V^{\otimes n})^{\Sigma_n}.$$
\end{lmm} 

\begin{proof}
\cite[Proposition~1.2.16]{Fresse04}. 
\end{proof}

\begin{rmk}\label{remark: isotropic}
Note that if $p\neq 2$, then $L_n(V) = (\bfLie_n \otimes V^{\otimes n})_{\Sigma_n}$.
\end{rmk}

\begin{prop}\label{proposition: homotopy groups of free restricted as xi-module}
Let $V_\bullet \in \sVect_{\kk}$ be a connected simplicial vector space, i.e. $\pi_0(V_\bullet)=0$. Then
\begin{enumerate}
\item $\pi_i(\free(V_\bullet))$ is a free $\kk\{\xi\}$-module for each $i\geq 0$;
\item The $\xi$-adic completion $\widehat{\pi}_i(\free(V_\bullet))$ is isomorphic as a vector space to the infinite direct product 
$$\widehat{\pi}_i(\free(V_\bullet))\cong \prod_{n\geq 1}\pi_{i,n}(\free(V_\bullet)),\; i\geq 0. $$
\end{enumerate}
\end{prop}

\begin{proof}
We write %$L_n(V_\bullet)=(\bfLie_n\otimes V_\bullet^{\otimes n})_{\Sigma_n}$ for the \emph{$n$-th Lie power} of $V_\bullet$. The norm map induces a map
$$\iota_n\colon L_n(V_\bullet) \to (\bfLie_n \otimes V_\bullet^{\otimes n})^{\Sigma_n} = L^r_n(V_\bullet) $$
for the map induced by the norm map. The $p$-operation $\xi \colon \free(V_\bullet) \to \free(V_\bullet)$ induces maps
$$\xi_n\colon L^r_n(V_\bullet) \to L^r_{pn}(V_\bullet),\; n\geq 1$$
Together, $\iota_{pn}$ and $\xi_n$ give a splitting of simplicial sets
\begin{equation}\label{equation: curtis, splitting}
\xi_n+\iota_{pn}\colon L^r_n(V_\bullet)\times L_{pn}(V_\bullet) \xrightarrow{\cong} L^r_{pn}(V_\bullet).
\end{equation}
Furthermore, the map $\iota_n\colon L_n(V_\bullet) \to L^r_n(V_\bullet)$ is an isomorphism if $p \nmid n$.

Let us denote by $c\in \N$ the \emph{connectivity} of $V_\bullet$; i.e. $c$ is the largest integer such that $\pi_{i}(V_\bullet)=0$ for all $i < c$. (Note that this definition is consistent with the definition of connectivity which is used in the subsequent references, but it differs from Definition~\ref{definition: connectivity, objects} used before.) By the assumption, $c \geq 1$. % because $V_\bullet$ is of finite type (Definition~\ref{definition: vector space of finite type}). 
By the Curtis theorem, the simplicial vector space $L_n(V_\bullet)$ is $(c+ \lceil \log_2 n \rceil)$-connected (see~\cite{Curtis_lower} for the original proof and~\cite{IRS20} for a shorter one). We fix an integer $N$ such that $i<c+\lceil \log_2 N \rceil$ and $N>p$. Then the splitting~\eqref{equation: curtis, splitting} implies that
$$\xi=\xi_{n*}\colon \pi_{i,n}(\free (V_\bullet)) \to \pi_{i,pn}(\free(V_\bullet)) $$
is a monomorphism for all $n$ and $\xi_{n*}$ is an isomorphism as soon as $pn>N$. Furthermore, $\pi_{i,n}(\free(V_\bullet))=0$ for $n>N$ and $p\nmid n$.

Let $\mathcal{I}$ be the set of integers $m$ such that $N/p< m \leq N$. Given $n\in \N$, we denote by $\nu_p(n)$ its $p$-valuation and we set $k(n)=np^{-\nu_p(n)}$. We construct a monomorphism of left $\kk\{\xi\}$-modules
\begin{equation}\label{equation: curtis, monomorphism}
i=\oplus i_n\colon \pi_{i}(\free(V_\bullet))=\bigoplus_{n\geq 1}\pi_{i,n}(\free(V_\bullet))\hookrightarrow \bigoplus_{m\in \mathcal{I}} M_{\nu_p(m)} \otimes_{\kk} \pi_{i,m}(\free(V_\bullet)),
\end{equation}
where $M_{k} \in \Mod_{\kk\{\xi\}}, k\geq 0$ is the left cyclic submodule of the twisted Laurent polynomials $\kk\{\xi^{\pm}\}$ generated by $\xi^{-k}$, see Section~\ref{section: derived xi-complete modules}. Since %$\pi_{i,m}(\free(V_\bullet)), m\in \mathcal{I}$ are finite dimensional vector spaces and 
$M_k\cong \kk\{\xi\}, k\geq 0$, the monomorphism~\eqref{equation: curtis, monomorphism} implies the first part of the proposition.

We define the components $i_n$ of~\eqref{equation: curtis, monomorphism} as follows. If $k(n)>N$, then we have $\pi_{i,n}(\free(V_\bullet))=0$; and so we set $i_n=0$. If $k(n)\leq N$, then there exists a unique $d_n \in \Z$ such that $m(n)=np^{d_n} \in \mathcal{I}$; and then we set 
$$i_n\colon \pi_{i,n}(\free(V_\bullet)) \to M_{\nu_p(m(n))}\otimes_{\kk} \pi_{i,m(n)}(\free(V_\bullet)),$$
$$i_n(x)=\xi^{-d_n}\otimes \xi^{d_n}(x). $$
The map $i_n$ is well-defined by the Curtis theorem, and this is straightforward to check that $i=\oplus i_n$ is a monomorphism of left $\kk\{\xi\}$-modules.

We prove the second part. By the isomorphism~\eqref{equation: free restricted, weight} and the first part of the proposition, the $\xi$-adic completion $\widehat{\pi}_i(\free(V_\bullet))$ is the subset of the product $\prod_{n\geq 1} \pi_{i,n}(\free(V_\bullet))$ consisting of all sequences that $\xi$-adically converge to zero. By the monomorphism~\eqref{equation: curtis, monomorphism}, the free $\kk\{\xi\}$-module $\pi_i(\free(V_\bullet))$ is generated by the subspace $$\bigoplus_{n\leq N}\pi_{i,n}(\free(V_\bullet)$$ such that all elements of higher weights are obtained by iterating the $p$-operation~$\xi$, which multiply the weight by~$p$. Therefore,  any sequence in $\prod_{n\geq 1} \pi_{i,n}(\free(V_\bullet))$ is $\xi$-adically convergent to zero.
\end{proof}

The last proposition together with Corollary~\ref{corollary: F-completion, homotopy groups} and Lemma~\ref{lemma: derived completion coincides with usual} immediately imply the following statement.

\begin{cor}\label{corollary: homotopy groups of F-completion of free}
Let $V_\bullet \in \sVect_{\kk}$ be a connected simplicial vector space, i.e. $\pi_0(V_\bullet)=0$. Then the homotopy groups $\pi_*(L_\xi \free(V_\bullet))$ of the $\kk$-completion $L_\xi \free(V_\bullet)$ are isomorphic as vector spaces to
\begin{equation*}
\pi_*(L_\xi \free(V_\bullet))\cong \widehat{\pi}_*(\free(V_\bullet)) \cong \prod_{n\geq 1}\pi_{*,n}(\free(V_\bullet)).  \eqno\qed
\end{equation*}
\end{cor}

In this way, the spectral sequence~\eqref{equation: ASS, Lie} looks as follows
\begin{equation}\label{equation: ASS, wedge of spheres}
E^2_{s,t}=\Ext_{\calU^h}^s(\Sigma{\pi}^*(V_\bullet), \Sigma^{t+1}\kk) \Rightarrow \prod_{n\geq 1}\pi_{t-s,n}(\free(V_\bullet)). 
\end{equation}
We compute the second page $E^2_{s,t}$ provided the homotopy groups $\pi_*(V_\bullet)$ are a one-dimensional vector space and show that the spectral sequence~\eqref{equation: ASS, wedge of spheres} degenerates in this case. The case $\dim \pi_*(V_\bullet)>1$ will be discussed in Remark~\ref{remark: ASS, highly dimensionsional}.

Let $\pi_l(V_\bullet)\cong\kk, l>0$ and $\pi_*(V_\bullet)=0, \ast\neq l$. Then there are three cases: $p=2$ and $l$ is any; $p$ is odd, $l$ is even; $p$ is odd, $l$ is odd. We consider them separately.

\subsubsection{\texorpdfstring{$p=2$}{p = 2} or \texorpdfstring{$p$}{p} is odd and \texorpdfstring{$l$}{l} is even}\label{section: ASS, p=2; p is odd, l is even}
In these cases, the unstable $\calA_p^h$-algebra $\Sigma^{l+1}\kk$ is equal to $\mathcal{F}(\Sigma^{l+1}\kk)$, where $\Sigma^{l+1} \kk$ is considered as an unstable $\calA_p^h$-module (Remark~\ref{remark: free unstable algebra generated by unstable module}). Therefore we have an isomorphism
$$\Ext_{\calU^h}^s(\Sigma^{l+1}\kk,\Sigma^{t+1}\kk) \cong \Ext_{\calM^h}^s(\Sigma^{l+1}\kk, \Sigma^{t+1}\kk), $$
and so we can directly apply Corollary~\ref{corollary: unstable exts between trivial}.

\begin{cor}\label{corollary: ASS, degenerates, p=2}
Let $V_\bullet \in \sVect_\kk$ be a simplicial vector space of finite type such that $\pi_l(V_\bullet)\cong \kk, l\geq 1$ and $\pi_*(V_\bullet)=0, \ast\neq l$. Suppose that either $p=2$ or $p$ is odd and $l$  is even. Then the spectral sequence~\eqref{equation: ASS, wedge of spheres}:
\begin{equation}\label{equation: ASS, p=2, 1-dim}
E^2_{s,t}=\Ext_{\calU^h}^s(\Sigma{\pi}^*(V_\bullet), \Sigma^{t+1}\kk) \Rightarrow \prod_{n\geq 1}\pi_{t-s,n}(\free(V_\bullet)). 
\end{equation}
degenerates at the second page. Moreover, $\pi_{*,n}(\free(V_\bullet))=0$ if $n\neq p^h, h\in \N$ and
\begin{equation}\label{equation: unstable ext is a piece in homotopy groups}
\Ext_{\calU^h}^s(\Sigma{\pi}^*(V_\bullet), \Sigma^{t+1}\kk) \cong \pi_{t-s,p^s}(\free(V_\bullet)), \; t\geq s\geq 0. 
\end{equation}
\end{cor}

\begin{proof}
By Corollary~\ref{corollary: unstable exts between trivial}, there are isomorphisms:
$$E^2_{s,t}\cong (\Sigma\pi_*(V_\bullet)^{(s)}\widehat{\otimes}\Lambda_s)_{t+1-s}\cong\pi_*(V_\bullet)^{(s)} \otimes \Lambda_{t-l-s,s}(l+1) $$
First, assume that the field $\kk$ is algebraically closed, $\kk=\overline{\kk}$. Then there are no non-trivial natural transformations 
$$(-)^{(s)} \to (-)^{(s')} \colon \Vect_{\kk} \to \Vect_{\kk}, \; s\neq s'$$
between different Frobenius twists. Since the differentials $d_r, r\geq 2$ of~\eqref{equation: ASS, p=2, 1-dim} are natural in $V_\bullet$, they are all zeros. 

Since the spectral sequence~\eqref{equation: ASS, p=2, 1-dim} is completely convergent, we obtain a natural isomorphism:
\begin{equation}\label{equation: degeneration, eq1, p=2}
\prod_{s\geq 0} \Ext_{\calU^h}^s(\Sigma\pi^*(V_\bullet), \Sigma^{t+1}\kk) \cong \prod_{n\geq 1} \pi_{t-s,n}(\free(V_\bullet)).  
\end{equation}
The multiplicative group $\kk^\times$ acts on $V_\bullet$ by multiplication, and so both sides are $\kk^\times$-representations. We derive the isomorphism~\eqref{equation: unstable ext is a piece in homotopy groups} by comparing isotypic components.

If the field $\kk$ is not algebraically closed, then we have isomorphisms:
$$\Ext_{\calU^h}^s(\Sigma\pi^*(V_\bullet), \Sigma^{t+1}\kk) \otimes_{\kk} \overline{\kk} \cong \Ext_{\calU^h}^s(\Sigma\pi^*(V_\bullet\otimes_{\kk}\overline{\kk}), \Sigma^{t+1}\overline{\kk}),\; t\geq s\geq 0 $$
$$\pi_{*,n}(\free(V_\bullet)) \otimes_{\kk}\overline{\kk} \cong \pi_{*,n}(\free(V_\bullet\otimes_{\kk}\overline{\kk})),\; n\geq 0 $$
$$\Ext_{\calU^h}^s(\Sigma\pi^*(V_\bullet\otimes_{\kk}\overline{\kk}), \Sigma^{t+1}\overline{\kk}) \cong \pi_{t-s,p^s}(\free(V_\bullet\otimes_{\kk}\overline{\kk})),\; t\geq s\geq 0. $$
Together, these isomorphisms imply that the isomorphism~\eqref{equation: unstable ext is a piece in homotopy groups} holds even for non-algebraically closed fields. 
%$$\Ext_{\calU^h}^s(\Sigma\pi^*(V_\bullet), \Sigma^{t+1}\kk) \cong \pi_{t-s,2^s}(\free(V_\bullet)), t\geq s\geq 0, $$
Finally, the isomorphism~\eqref{equation: unstable ext is a piece in homotopy groups} implies the isomorphism~\eqref{equation: degeneration, eq1, p=2} of direct products, which gives the degeneration of the spectral sequence~\eqref{equation: ASS, p=2, 1-dim}.
\end{proof}

\subsubsection{\texorpdfstring{$p$ is odd, $l$ is odd}{p is odd and l is odd}}\label{section: ASS, p is odd, l is odd} In this case, we have
$$\widetilde{H}^*(\free(V_\bullet);\kk)\cong \kk[x]/x^2, \; |x|=l+1.$$
This unstable $\calA^h_p$-algebra is not $\mathcal{F}(M_*)$ for any $M_*\in \calM^h$, which slightly complicates the problem.

Let $A_*$ be a non-unital graded commutative $\kk$-algebra. We denote by $AQ_q(A_*) \in \Vect^{gr}_{\kk}$ the \emph{$q$-th Andr\'{e}-Quillen homology} of $A_*$; i.e. $AQ_*(A_*)$ is the left non-abelian derived functor of 
$$A_*\mapsto AQ_0(A_*)= A_*/A_*^2.$$
We refer the reader to~\cite{Andre67} and~\cite{Quillen_rings} for further details. If $A_* \in \calU^h$ is an unstable $\calA^h_p$-algebra, then $AQ_q(A_*)\in \calM^h_0, q\geq 0$ are strongly unstable $\calA^h_p$-modules (Definition~\ref{definition: strongly unstable module}). Similar to~\cite[Theorem~2.5]{Miller84} and~\cite[Theorem~4.3]{Goerss90_BKSS}, we obtain the strongly convergent (Grothendieck) spectral sequence:
\begin{equation}\label{equation: grothendieck SS}
E_2^{r,q}=\Ext_{\calM^h_0}^r(AQ_q(A_*),\Sigma^{t+1}\kk) \Rightarrow  \Ext^{r+q}_{\calU^h}(A_*,\Sigma^{t+1}\kk), \; A_*\in\calU^h, \; t\geq 0.
\end{equation}

\begin{lmm}\label{lemma: grothendieck ss, 1-dim}
Suppose that $\mathrm{char}(\kk)>2$. Let $V_\bullet\in \sVect_{\kk}$ be a simplicial vector space such that $\pi_l(V_\bullet)\cong \kk$, $l$ is odd, $l\geq 1$, and $\pi_*(V_\bullet)=0$ if $*\neq l$. Then the spectral sequence~\eqref{equation: grothendieck SS} degenerates at the second page, there are isomorphisms:
$$E_2^{r,0}\cong (\calK_{p}^r\widetilde{\otimes}(\Sigma\pi_*(V_\bullet))^{(r)})_{t+1}, \;\; E_2^{r,1}\cong (\calK^r_p\widetilde{\otimes}((\Sigma\pi_*(V_\bullet))^{\otimes 2})^{(r)})_{t+1}, \;r\geq t\geq 0,$$
and there is a natural (in $V_\bullet$) splitting
$$\Ext^s_{\calU^h}(\Sigma\pi^*(V_\bullet),\Sigma^{t+1}\kk)\cong \Ext^s_{\calM^h_0}(\Sigma\pi^*(V_\bullet),\Sigma^{t+1}\kk) \oplus \Ext^{s-1}_{\calM^h_0}((\Sigma\pi^*(V_\bullet))^{\otimes 2},\Sigma^{t+1}\kk). $$
\end{lmm}

\begin{proof}
%Note that the trivial graded commutative algebra $\Sigma\pi^*(V_\bullet)$ has a free resolution
%$$0 \to \kk[(\Sigma\pi^*(V_\bullet))^{\otimes 2}] \to \kk[\Sigma\pi^*(V_\bullet)] \to \Sigma\pi^*(V_\bullet) \to 0. $$
There exists a pushout square
$$
\begin{tikzcd}
\kk[(\Sigma \pi^*(V_\bullet))^{\otimes 2}] \arrow[swap]{d} \arrow{r}
	& 0 \arrow{d}\\
\kk[\Sigma \pi^*(V_\bullet)]  \arrow{r}
& \Sigma \pi^*(V_\bullet)
\end{tikzcd}
$$
of non-unital graded commutative $\kk$-algebras. Here $\kk[\Sigma \pi^*(V_\bullet)]$ and $\kk[(\Sigma \pi^*(V_\bullet))^{\otimes 2}]$ are free non-unital algebras, and the left vertical arrow is induced by the inclusion $(\Sigma \pi^*(V_\bullet))^{\otimes 2} \subset \kk[\Sigma \pi^*(V_\bullet)]$. Since the Andr\'{e}-Quillen homology maps pushout squares to long exact sequences, we have $$AQ_0(\Sigma\pi^*(V_\bullet))\cong \Sigma\pi^*(V_\bullet), \;\; AQ_1(\Sigma\pi^*(V_\bullet))\cong (\Sigma\pi^*(V_\bullet))^{\otimes 2},$$ and
$AQ_q(\Sigma\pi^*(V_\bullet))=0$ if $q>1$. 

By Corollary~\ref{corollary: unstable exts between trivial}, there are isomorphisms
$$E_2^{r,0}\cong (\calK_{p}^r\widetilde{\otimes}(\Sigma\pi_*(V_\bullet))^{(r)})_{t+1}, \;\; E_2^{r,1}\cong (\calK^r_p\widetilde{\otimes}((\Sigma\pi_*(V_\bullet))^{\otimes 2})^{(r)})_{t+1}, \;r\geq t\geq 0,$$
and $E^{r,q}_2=0$ if $q>1$. By the same argument with $\kk^\times$-action as in Corollary~\ref{corollary: ASS, degenerates, p=2}, we get that the spectral sequence~\eqref{equation: grothendieck SS} degenerates at the second page and there is a natural splitting:
\begin{align*}
\Ext^s_{\calU^h}(\Sigma\pi^*(V_\bullet), \Sigma^{t+1}\kk)&\cong E_2^{s,0}\oplus E_2^{s-1,1}. \qedhere
%&= \Ext^s_{\calM^h_0}(\Sigma\pi^*(V_\bullet),\Sigma^{t+1}\kk) \oplus \Ext^{s-1}_{\calM^h_0}((\Sigma\pi^*(V_\bullet))^{\otimes 2},\Sigma^{t+1}\kk)
\end{align*}
\end{proof}

Using the computation in Lemma~\ref{lemma: grothendieck ss, 1-dim} we obtain an analog of Corollary~\ref{corollary: ASS, degenerates, p=2} for odd $l$ as well. The proof of Corollary~\ref{corollary: ASS, degenerates, p is odd, l is odd} is absolutely parallel to the proof of Corollary~\ref{corollary: ASS, degenerates, p=2}, so we leave it to the reader to complete the details.

\begin{cor}\label{corollary: ASS, degenerates, p is odd, l is odd}
Suppose that $\mathrm{char}(\kk)>2$. Let $V_\bullet\in \sVect_{\kk}$ be a simplicial vector space such that $\pi_l(V_\bullet)\cong \kk$, $l$ is odd, $l\geq 1$, and $\pi_*(V_\bullet)=0$ if $*\neq l$. Then the spectral sequence~\eqref{equation: ASS, wedge of spheres}:
\begin{equation*}\label{equation: ASS, p is odd, 1-dim}
E^2_{s,t}=\Ext_{\calU^h}^s(\Sigma{\pi}^*(V_\bullet), \Sigma^{t+1}\kk) \Rightarrow \prod_{n\geq 1}\pi_{t-s,n}(\free(V_\bullet)). 
\end{equation*}
degenerates at the second page. Moreover, $\pi_{*,n}(\free(V_\bullet))=0$ if $n\neq p^h$ or $n\neq 2p^h$, $h\in \N$, and there are isomorphisms
\begin{equation*}
\Ext_{\calU^h}^s(\Sigma{\pi}^*(V_\bullet), \Sigma^{t+1}\kk) \cong \pi_{t-s,p^s}(\free(V_\bullet))\oplus \pi_{t-s,2p^{s-1}}(\free(V_\bullet)),
\end{equation*}
\begin{equation*}
\pi_{t-s,p^s}(\free(V_\bullet))\cong \Ext_{\calM^h_0}^s(\Sigma{\pi}^*(V_\bullet), \Sigma^{t+1}\kk) \cong\pi_*(V_\bullet)^{(s)}\otimes \Lambda_{t-l-s,s}(l),
\end{equation*}
\begin{align*}
\pi_{t-s,2p^{s-1}}(\free(V_\bullet))&\cong \Ext_{\calM^h_0}^{s-1}((\Sigma{\pi}^*(V_\bullet))^{\otimes 2}, \Sigma^{t+1}\kk) \\ 
&\cong ((\pi_*(V_\bullet))^{\otimes 2})^{(s-1)}\otimes \Lambda_{t-2l-s,s-1}(2l+1)
\end{align*}
for $t\geq s\geq 0$. \qed
\end{cor}

\begin{rmk}\label{remark: ASS, highly dimensionsional}
Let $V_\bullet \in \sVect_\kk^{ft}$ be a simplicial vector space of finite type such that $\dim\pi_*(V_\bullet)>1$. Then one still can apply spectral sequences~\eqref{equation: ASS, wedge of spheres} and~\eqref{equation: grothendieck SS} in order to compute the homotopy groups $\pi_*(\free(V_\bullet))$ of the free simplicial restricted Lie algebra $\free(V_\bullet) \in \srLie$. However, the Andr\'{e}-Quillen homology groups of the trivial algebra $\Sigma\pi^*(V_\bullet)$ are quite large in this case, see~\cite{Goerss90} and~\cite[Theorem~8.18]{AB21}. Therefore we can not expect that any of two spectral sequences~\eqref{equation: ASS, wedge of spheres} or~\eqref{equation: grothendieck SS} degenerate at the second page and an additional expertise is required.
\end{rmk}

\begin{rmk}\label{remark: hilton-milnor}
Nevertheless, one can use the (algebraic) Hilton-Milnor theorem to compute $\pi_*(\free(V_\bullet))$, $V_\bullet\in \sVect_\kk$. Let $V_{1,\bullet},V_{2,\bullet} \in \sVect_{\kk}$ be simplicial vector spaces and let $L$ be the free Lie algebra over the integers on two symbols $i_1$ and $i_2$ and let $B\subset L$ be the Hall basis for $L$, see~\cite[p. 512]{Whitehead_elements}. Then a word $w\in B$ is an iterated Lie bracket 
$$w=[i_{j_1}[i_{j_2},\ldots,i_{j_s}]], $$
where $j_t\in \{1,2\}$, $1 \leq t\leq s$. We associate to $w$ the iterated tensor product
$$w(V_{1,\bullet},V_{2,\bullet})=V_{j_1,\bullet}\otimes\ldots \otimes V_{j_s,\bullet}$$
and the canonical inclusion
$$l_w\colon w(V_{1,\bullet},V_{2,\bullet})\to \oblv \circ \free(V_{1,\bullet}\oplus V_{2,\bullet}) $$
such that $l_w(v^{j_1}_{1}\otimes \ldots \otimes v^{j_s}_s) = [v^{j_1}_{1}[v^{j_2}_2,\ldots,v^{j_s}_s]]$, where $v^{j_t}_t \in V_{j_t,\bullet}$. These maps determine the map
$$l\colon \bigoplus_{w\in B} \oblv \circ \free(w(V_{1,\bullet},V_{2,\bullet})) \to \oblv \circ \free(V_{1,\bullet}\oplus V_{2,\bullet}), $$
which is an isomorphism, see e.g.~\cite[Example~8.7.4]{Neisendorfer10}. Therefore one can reduce the problem of calculating $\pi_*(\free(V_\bullet))$, $\dim\pi_*(V_\bullet)>1$ to the considered case $\dim\pi_*(V_\bullet)=1$.
\end{rmk}

\begin{rmk}\label{remark: xi and mu_0 and lambda_0}
Let $V_*\in \Vect_{\kk}^{gr}$ be a graded vector space. Let us denote by $\free(V_*) \in \srLie$ the simplicial vector space generated by $\Gamma V_*\in \sVect_{\kk}$, where $\Gamma$ is the inverse of the normalized chain complex functor $N$, see Section~\ref{section: notation}. By Proposition~\ref{proposition: free lie algebras, fresse}, there is a canonical embedding $V_* \subset \pi_*(\free(V_*))$. By Proposition~\ref{proposition: homotopy groups of free restricted as xi-module}, we observe that $\xi(v)\in \pi_*(\free(V_*))$ is non-zero for any $v\in V_*\subset \pi_*(\free(V_*))$. Finally, Corollaries~\ref{corollary: ASS, degenerates, p=2} and~\ref{corollary: ASS, degenerates, p is odd, l is odd} imply that $\xi(v)$ is a non-zero multiple of $v\otimes \mu_0 \in \pi_{*,p}(\free(V_*))$ (resp. $v\otimes \lambda_0\in\pi_{*,2}(\free(V_*))$).
\end{rmk}

\begin{rmk}\label{remark: lambdas and mu are detected}
Let $x\in \pi_m\free(\Sigma^n\kk)$ be a homotopy class $$x\in \pi_m\free(\Sigma^n\kk)=[\free(\Sigma^m\kk),\free(\Sigma^n\kk)].$$ Consider cofiber sequences 
$$\free(\Sigma^m\kk) \xrightarrow{x} \free(\Sigma^n\kk) \to \cofib(x), $$
$$\free (\Sigma^n\kk) \to \cofib(x) \to \free(\Sigma^{m+1}\kk). $$
The second one implies that $\widetilde{H}^q(\cofib(x);\kk)=0$ if $q\neq n+1, m+2$, and there are canonical generators: 
$$h_{n+1}\in \widetilde{H}^{n+1}(\cofib(x);\kk)\cong \widetilde{H}^{n+1}(\free(\Sigma^n\kk);\kk),$$
$$h_{m+2}\in \widetilde{H}^{m+2}(\cofib(x);\kk)\cong \widetilde{H}^{m+2}(\free(\Sigma^{m+1}\kk);\kk).$$ 
We say that a cohomology operation $P$ \emph{detects} $x$ if $P(h_{n+1})=h_{m+2}$. Let $\iota_n\in \pi_n\free(\Sigma^n\kk)$ be the canonical generator. Then Corollary~\ref{corollary: ASS, degenerates, p=2} implies
\begin{itemize}
\item if $p=2$, then the element $\iota_n\otimes \lambda_i\in \pi_{n+i}(\free(\Sigma^n\kk))$, $0\leq i\leq n$ is detected by $Sq^{i+1}$.
\item if $p$ is odd and $n=2k$, then the element $$\iota_n\otimes \lambda_{i}\in \pi_{n+2i(p-1)-1}(\free(\Sigma^n\kk)),\;  1\leq i\leq k$$ is detected by the Steenrod operation $P^i$.
\item if $p$ is odd and $n=2k$, then the element $$\iota_n\otimes \mu_{i}\in \pi_{n+2i(p-1)}(\free(\Sigma^n\kk)), \; 0\leq i\leq k$$ is detected by the Steenrod operation $\beta P^i$.
\end{itemize}
This characterization of the generators in the algebra $\Lambda$ seems to be folklore. However, the only case covered in the literature is $p=2$, see~\cite[p.~515]{Priddy70short}; the case of odd primes seems to be missing.
\end{rmk}

\bibliographystyle{alpha}
\bibliography{references}

\end{document}